\documentclass[a4paper,11pt]{article}

\usepackage{amsmath,amsthm}
\usepackage{amssymb}
\usepackage{breqn}
\usepackage{enumerate}
\usepackage{graphicx}
\usepackage{bm}
\usepackage{bbm}
\usepackage[affil-it]{authblk}
\usepackage{tabu}
\usepackage{bold-extra}
\usepackage[hang,flushmargin]{footmisc}
\usepackage[hypertex]{hyperref}
\usepackage{mathtools}
\usepackage{relsize}
\usepackage{scalerel}
\usepackage{xcolor}
\usepackage{titlesec}
\usepackage{apptools}
\usepackage{appendix}

\newtheorem{theorem}{Theorem}
\newtheorem{corollary}[theorem]{Corollary}
\newtheorem{lemma}[theorem]{Lemma}
\newtheorem{proposition}[theorem]{Proposition}

\newtheorem{conjecture}[theorem]{Conjecture}
\newtheorem{claim}[theorem]{Claim}

\theoremstyle{definition}
\newtheorem{defin}[theorem]{Definition}

\AtAppendix{\counterwithin{theorem}{section}}

\titleformat{\section}[hang]{\scshape\large\bfseries\filcenter}{\S\thesection}{4pt}{}
\titleformat{\subsection}[hang]{\scshape\bfseries}{\thesubsection.}{4pt}{}

\allowdisplaybreaks

\newcommand{\tss}[1]{\textsuperscript{#1}}

\newcommand{\on}[1]{
	\operatorname{#1}
}

\newcommand\mder{{\Delta\!\!\!\!\!\hbox{\raisebox{0.3ex}{\tiny\ \textbullet}}}\,}

\newcommand{\tdt}{\times\cdots\times}

\newcommand{\tightoverset}[2]{
  \mathop{#2}\limits^{\vbox to -.5ex{\kern-1.15ex\hbox{$#1$}\vss}}}
\newcommand{\conv}{
	\tightoverset{\boldsymbol{-}}{\ast}
}

\newcommand\restr[2]{{
  \left.\kern-\nulldelimiterspace 
  #1 
  \vphantom{\big|} 
  \right|_{#2} 
}}

\makeatletter
\newcommand{\subalign}[1]{%
  \vcenter{%
    \Let@ \restore@math@cr \default@tag
    \baselineskip\fontdimen10 \scriptfont\tw@
    \advance\baselineskip\fontdimen12 \scriptfont\tw@
    \lineskip\thr@@\fontdimen8 \scriptfont\thr@@
    \lineskiplimit\lineskip
    \ialign{\hfil$\m@th\scriptstyle##$&$\m@th\scriptstyle{}##$\hfil\crcr
      #1\crcr
    }%
  }%
}
\makeatother

\newcommand\blfootnote[1]{%
  \begingroup
  \renewcommand\thefootnote{}\footnote{#1}%
  \addtocounter{footnote}{-1}%
  \endgroup
}

\newcommand\ssk[1]{
	\substack{#1}
}

\newcommand\ex{\mathop{\mathbb{E}}}

\newcommand{\exx}{
  \mathop{
    \mathchoice{\vcenter{\hbox{\larger[4]$\mathbb{E}$}}}
               {\kern0pt\mathbb{E}}
               {\kern0pt\mathbb{E}}
               {\kern0pt\mathbb{E}}
  }\displaylimits
}

\makeatletter
\newcommand*\bcdot{\mathpalette\bigcdot@{0.5}}
\newcommand*\bigcdot@[2]{\mathbin{\vcenter{\hbox{\scalebox{#2}{$\m@th#1\bullet$}}}}}
\makeatother

\makeatletter
\def\blfootnote{\gdef\@thefnmark{}\@footnotetext}
\makeatother

\newcommand\id{\mathbbm{1}}

\begin{document}
\begin{center}\Large\noindent{\bfseries{\scshape Approximate quadratic varieties}}\\[24pt]\normalsize\noindent{\scshape Luka Mili\'cevi\'c\dag}\\[6pt]
\end{center}
\blfootnote{\noindent\dag\ Mathematical Institute of the Serbian Academy of Sciences and Arts\\\phantom{\dag\ }Email: luka.milicevic@turing.mi.sanu.ac.rs}

\footnotesize
\begin{changemargin}{1in}{1in}
\centerline{\sc{\textbf{Abstract}}}
\phantom{a}\hspace{12pt}~A classical result in additive combinatorics, which is a combination of Balog-Szemer\'edi-Gowers theorem and a variant of Freiman's theorem due to Ruzsa, says that if a subset $A$ of $\mathbb{F}_p^n$ contains at least $c |A|^3$ additive quadruples, then there exists a subspace $V$, comparable in size to $A$, such that $|A \cap V| \geq \Omega_c(|A|)$. Motivated by the fact that higher order approximate algebraic structures play an important role in the theory of uniformity norms, it would be of interest to find higher order analogues of the mentioned result.\\
\phantom{a}\hspace{12pt}~In this paper, we study a quadratic version of the approximate property in question, namely what it means for a set to be an approximate quadratic variety. It turns out that information on the number of additive cubes, which are 8-tuples of the form $(x, x+ a, x+ b, x+ c , x+ a + b, x+ a + c, x+ b + c, x+ a + b + c)$, in a set is insufficient on its own to guarantee quadratic structure, and it is necessary to restrict linear structure in a given set, which is a natural assumption in this context. With this in mind, we say that a subset $V$ of a finite vector space $G$ is a \emph{$(c_0, \delta, \varepsilon)$-approximate quadratic variety} if $|V| = \delta |G|$, $\|\id_V - \delta\|_{\mathsf{U}^2} \leq \varepsilon$ and $V$ contains at least $c_0\delta^7 |G|^4$ additive cubes. Our main result is the structure theorem for approximate quadratic varieties, stating that such a set has a large intersection with an exact quadratic variety of comparable size.
\end{changemargin}
\normalsize
\section{Introduction}

Approximate algebraic structures are a major topic of study in additive combinatorics. First result concerning such structures, and a cornerstone of the area, is Freiman's theorem~\cite{Freiman}. As we shall work in vector spaces over a fixed prime field $\mathbb{F}_p$ in this paper, we begin by recalling an analogue of Freiman's theorem in that setting, due to Ruzsa~\cite{RuzsaFp}.

\begin{theorem}[Ruzsa~\cite{RuzsaFp}]\label{frthm} Let $G$ be a finite-dimensional vector space over a prime field $\mathbb{F}_p$. Suppose that $A \subseteq G$ satisfies $|A + A| \leq K |A|$.\footnote{This condition is called \emph{small doubling} and set obeying this condition are traditionally called \emph{approximate groups}.} Then there exists a subspace $U \leq G$ of size $|U| \leq O_K(|A|)$ such that $A \subseteq U$.\end{theorem}

Let us record here a variant of that theorem which has a condition involving additive quadruples, which are $(x,y,z,w)$ such that $x + y = z + w$, instead of the small doubling assumption. The variant is obtained after an application of Balog-Szemer\'edi-Gowers theorem~\cite{bsgthm,Gowers4AP}, and appears more naturally in the context of higher order Fourier analysis. 

\begin{theorem}\label{bsgfreimanintro} Let $G$ be a finite-dimensional vector space over a prime field $\mathbb{F}_p$. Suppose that $A \subseteq G$ and that $A$ has at least $c|A|^3$ additive quadruples. Then there exists a subspace $U \leq G$ of size $|U| \leq O_c(|A|)$ such that $|U \cap A| \geq \Omega_c(|A|)$.\end{theorem}

We may think of the sets satisfying the assumptions of Theorem~\ref{bsgfreimanintro} as \emph{approximate cosets} (we choose this wording as the term \emph{approximate subgroup} has a standard meaning of having a small doubling). The way such a result is relevant in higher order Fourier analysis is via following consequence, which was proved by Gowers and plays an important role in his proof of an inverse theorem for $\mathsf{U}^3$ uniformity norm~\cite{Gowers4AP} (we shall recall the definition of this norm slightly later).

\begin{theorem}[Gowers~\cite{Gowers4AP}]\label{approxHommInvIntro}Let $G$ and $H$ be finite-dimensional vector spaces over $\mathbb{F}_p$. Let $A \subseteq G$ be a subset and let $\phi \colon A \to H$ be a map which \emph{respects} at least $c |G|^3$ additive quadruples in $A$, meaning that $\phi(x) + \phi(y) = \phi(z) + \phi(w)$ holds for at least $c |G|^3$ quadruples $(x,y,z,w) \in A^4$ such that $x + y = z + w$. Then there exists an affine map $\Phi \colon G\to H$ such that $\phi = \Phi$ holds for at least $\Omega_c(|G|)$ points in $A$.\end{theorem}

Strictly speaking, Gowers obtained Theorem~\ref{approxHommInvIntro} in the case of cyclic groups, but the same proof works in $\mathbb{F}_p^n$. We may think of this theorem as a structure theorem for approximate affine homomorphisms. This is an approximate structure in the context of maps between vector spaces.\\

Finally, let us mention another approximate algebraic structure. We have already mentioned in passing the uniformity norms, whose definition we now recall. Firstly, we recall that the \emph{discrete multiplicative derivative operator} $\mder_a$ for shift $a \in G$ is defined by $\mder_a f(x) = f(x + a)\overline{f(x)}$ for functions $f \colon G \to \mathbb{C}$.\\
\indent Let $f \colon G \to \mathbb{C}$ be a function. The \emph{Gowers uniformity norm} $\|f\|_{\mathsf{U}^k}$, defined by Gowers in~\cite{GowerskAP}, is given by the formula
\[\|f\|_{\mathsf{U}^k} = \Big(\exx_{x,a_1,\dots, a_k} \mder_{a_1} \dots \mder_{a_k} f(x) \Big)^{1/2^k}.\]

These norms are defined by a combinatorial average, which in the case of $\mathsf{U}^3$ norm is taken over additive cubes, which will be crucial for the rest of the paper. To be precise, an \emph{additive cube} is an 8-tuple of the form
\[\Big(x, x+ a, x+ b, x+ c , x+ a + b, x+ a + c, x+ b + c, x+ a + b + c\Big)\]
where $x,a,b,c \in G$ are any 4 elements of $G$.\\ 

In his inverse theorem for $\mathsf{U}^3$ norm, Gowers obtained a partial description of functions with large value of the norm, which was sufficient for the purposes of proving Szemer\'edi's theorem for arithmetic progressions of length 4. A qualitatively optimal inverse theorem was later obtained by Green and Tao~\cite{GreenTaoU3} for groups of odd order. In the vector space case their result is the following.

\begin{theorem}[Green and Tao~\cite{GreenTaoU3}]\label{GTU3thm}Suppose that $f \colon G \to \mathbb{D}$ satisfies $\|f\|_{\mathsf{U}^3} \geq c_0$ and assume $p \geq 3$. Then there exists a quadratic polynomial $q \colon G \to \mathbb{F}_p$ such that $\Big|\ex_{x} f(x) \exp\Big(\frac{2 \pi i q(x)}{p}\Big)\Big| \geq \Omega_{c_0}(1)$.\end{theorem}

We may thus think of functions with large $\mathsf{U}^3$ norm as approximate quadratic phases.\\

When it comes to higher order analogues of the results above, we have inverse theorems for $\mathsf{U}^k$ norms, for higher $k$. These were proved by Green, Tao and Ziegler~\cite{GTZ} for cyclic groups, by Bergelson, Tao and Ziegler~\cite{BTZ} for finite vector spaces over prime fields of sufficiently large characteristic and Tao and Ziegler~\cite{TaoZiegler} for all finite vector spaces. In particular, in vector spaces in the case of high characteristic where $p \geq k$, Theorem~\ref{GTU3thm} of Green and Tao generalizes to the following result.

\begin{theorem}[Bergelson, Tao, Ziegler~\cite{BTZ}] Suppose that $f \colon G \to \mathbb{D}$ satisfies $\|f\|_{\mathsf{U}^k} \geq c_0$ and assume $p \geq k$. Then there exists a degree $k-1$ polynomial $q \colon G \to \mathbb{F}_p$ such that $\Big|\ex_{x} f(x) \exp\Big(\frac{2 \pi i q(x)}{p}\Big)\Big| \geq \Omega_{c_0}(1)$.\end{theorem}

Similarly to Theorem~\ref{GTU3thm}, we may think of this theorem as a structure theorem for approximate polynomial forms.\\

Returning now to Theorem~\ref{approxHommInvIntro}, we remark that there are satisfactory higher order generalizations of approximate affine homomorphisms, namely approximate polynomials and their multilinear variant, Freiman multihomorphisms, whose definitions we now recall.\\
\indent Let $G$ and $H$ be two abelian groups, let $A \subseteq G$ be a subset and let $F \colon A \to H$ be a map. Similarly to $\mder_a$, let us write $\Delta_a F$ for the map from $A \cap A - a$ to $H$ given by $x \mapsto F(x + a) - F(x)$. We say that $F$ is an \emph{$\varepsilon$-approximate polynomial of degree at most $d$} if $\Delta_{a_1}\dots\Delta_{a_{d+1}} F(x) = 0$ holds for at least $\varepsilon |G|^{d+2}$ choices of $(d+2)$-tuples $(a_1, \dots, a_{d+1}, x) \in G^{d+2}$ (and all $2^{k+1}$ resulting arguments of $F$ lie in its domain).\\
\indent On the other hand, a \emph{Freiman multihomomorphism of order $d$} is a map $\Phi \colon A \to H$, where now $A \subseteq G^d$, such that $\Phi$ is a Freiman homomorphism in each principal direction, namely for each direction $i \in [d]$, and each element $(x_1, \dots, x_{i-1}, x_{i+1}, \dots, x_d)$ of $G_1\tdt G_{i-1}\times G_{i+1}\tdt G_d$, the map that sends each $y_i$ such that $(x_1, \dots,$ $x_{i-1},$ $y_i,$ $x_{i+1}, \dots,$ $x_d) \in A$ to $\phi(x_1, \dots,$ $x_{i-1},$ $y_i,$ $x_{i+1}, \dots,$ $x_d)$ respects all additive quadruples.\\
\indent Manners~\cite{MannersApprox} proved a structure theorem for approximate polynomials over cyclic groups and Gowers and the author~\cite{FMulti} proved a structure theorem for Freiman multihomomorphisms in finite vector spaces. Furthermore, in the same paper, Gowers and the author obtained a structure theorem for approximate polynomials in finite vector spaces in the case of high characteristic.\\
\indent Our goal in this paper is to make a step towards completing the picture and study the quadratic analogue of Theorem~\ref{bsgfreimanintro}. Table~\ref{TableMissingThm} summarizes the discussion of various approximate algebraic structures of linear and higher order.\\

\begin{table}
\centering
\begin{tabular}{ |c| c | c | c |}
\hline
 Degree & Sets & Maps $G \to H$ & Forms $G \to \mathbb{F}_p$ \\ \hline
  Linear &  Theorem 2 & Theorem 3 & Inverse theorem for $\|\cdot\|_{\mathsf{U}^2}$ norm \\ \hline
  Higher &  \textbf{??} &  \begin{tabular}{@{}c@{}c@{}}Structure theorem for \\ approximate polynomials and \\Freiman multihomorphisms\end{tabular}  & Inverse theorem for $\|\cdot\|_{\mathsf{U}^k}$ norm \\\hline
\end{tabular}
\caption{Various approximate algebraic structures.}
\label{TableMissingThm}
\end{table}

Returning to Theorem~\ref{bsgfreimanintro}, we may phrase the assumption on the number of additive quadruples as follows. Let $A$ be a subset of $G$ and let $x, a, b \in G$ be chosen uniformly and independently at random. Then 
\[\mathbb{P}(x + a + b \in A| x + a, x + b, x \in A) = \frac{\mathbb{P}(x + a + b , x + a, x + b, x \in A)}{\mathbb{P}(x + a, x + b, x \in A)} = \frac{Q}{|A|^3},\]
where $Q$ is the number of additive quadruples in $A$. Thus, $A$ has at least $c_0 |A|^3$ additive quadruples if and only the probability of 'completing an additive quadruple' $\mathbb{P}(x + a + b \in A| x + a, x + b, x \in A)$ is at least $c_0$. In other words, $A$ is approximately closed under the relevant operation.\\ 

In order to find a quadratic generalization, it is natural to count the additive cubes. Observe that if $V = \{x \in G \colon q(x) = 0\}$ is a quadratic variety defined by a quadratic map $q \colon G \to \mathbb{F}_p^r$, then $V$ is closed under 'completing additive cubes' due to the following identity
\[\phi(x + a + b + c) = \phi(x + a + b) + \phi(x + a + c) + \phi(x + b + c) - \phi(x + a) - \phi(x + b) - \phi(x + c) + \phi(x).\]

However, this time being approximately closed under completing additive cubes is no longer sufficient on its own to guarantee that $V$ has any quadratic structure at all. Consider the following example. Let $G = U \oplus T$ be a direct sum and let $\pi \colon G \to T$ be the projection onto $T$ associated with this direct sum. Let $S$ be a Sidon subset (a set without non-trivial additive quadruples) of $T$ and set $V = U + S$. Observe that if $x + a + b, x + a , x +b, x\in V$ then we have $\pi(x + a + b) + \pi(x) = \pi(x + a) + \pi(x + b)$, and, since $S$ is Sidon, this means that $\pi(x) \in \{\pi(x + a), \pi(x+b)\}$ so $a \in U$ or $b \in U$. Suppose now that 7 points $x, x+ a, x+ b, x+ c , x+ a + b, x+ a + c, x+ b + c$ all belong to $V$. If $a,b,c \in U$, then $x + a + b + c \in V$. Otherwise, suppose without loss of generality that $a \notin U$. Observation about additive quadruples in $V$ implies that $b,c \in U$, so $x + a + b + c \in x + a + U \subseteq V$, so actually $V$ is closed under completing additive cubes. \\
\indent The set $V$ above has no quadratic structure, and has some, but not significant linear structure. Since our aim is to obtain a purely quadratic structure theorem, we impose further condition on $V$ and assume that it has negligible linear structure. This leads us to the following definition, formulated by Gowers (personal communication), which also forbids the example above.

\begin{defin}\label{approxvardefn}Let $V \subseteq G$. We say that $V$ is a \emph{$(c_0, \delta, \varepsilon)$-approximate quadratic variety} if $|V| = \delta |G|$, $\|\id_V - \delta\|_{\mathsf{U}^2} \leq \varepsilon$ and $\ex_{x,a,b,c} \mder_{a,b,c}\id_V(x) = c_0\delta^7$.\end{defin}

We think of parameters $c_0, \delta$ and $\varepsilon$ as satisfying the relationship $c_0 \gg \delta \gg \varepsilon$. As a further motivation for imposing the $\mathsf{U}^2$ condition, we note the following connection with approximate quadratic polynomials.

\begin{proposition}\label{approximatePolynomialsAndApproximateVars}Let $G$ and $H$ be two finite-dimensional vector spaces over $\mathbb{F}_p$ and let $A \subseteq G$. Let $F \colon A \to H$ satisfy $\Delta_a \Delta_b\Delta_cF(x) = 0$ for at least $c_0 |G|^4$ choices of $(x,a,b,c) \in G^4$ (and all 8 arguments of $F$ belong to $A$) and let $\xi$ be a positive quantity. Let $d \in \mathbb{N}$ and let $\varepsilon > 0$ be given. Provided that $\varepsilon$ and $\dim G$ are sufficiently small and large, respectively, in terms of $c_0, d$ and $\xi$, the following holds.\\
\indent Suppose that $F$ respects at most $\varepsilon |G|^3$ additive quadruples in $A$. Then there exists a cost $w_0 + W$ of a subspace of codimension $O_\varepsilon(1)$ such that, if we choose a subspace $U$ of codimension $d$ uniformly at random, then the probability that the set
\[V = \{x \in A \colon x \in w_0 + W, F(x) \in U\}\]
is a $(c', \delta, \xi)$-approximate quadratic variety for some  $\delta \in [2^{-9}c_0^8 p^{-d}, 2p^{-d}]$ and $c' \in [c_0^82^{-9}, 2]$, is at least $0.99$.
\end{proposition}

We remark that the assumption of $F$ respecting very few additive quadruples is quite natural, as otherwise it becomes an approximate homomorphism instead, which is much simpler and to which we may apply Theorem~\ref{approxHommInvIntro}.\\

Gowers also conjectured that an approximate quadratic variety necessarily has a large intersection with an exact quadratic variety of comparable size. Notice the resemblance to Theorem~\ref{bsgfreimanintro}.

\begin{conjecture}\label{mainconj}Fix a prime $p \geq 3$. Let $V \subseteq G$ be a $(c_0, \delta, \varepsilon)$-approximate quadratic variety. Suppose that $\varepsilon$ is sufficiently small in terms of $\delta$ and $c_0$. Then there exists a quadratic variety $Q$ such that $|Q| \leq O_{c_0}(|V|)$ and $|Q \cap V| \geq  \Omega_{c_0}(|V|)$.\end{conjecture}

Having discussed approximate quadratic varieties and the motivation for the definition, we are now ready to state our main result of this paper, which is the resolution of Conjecture~\ref{mainconj} with reasonable dependencies on $c_0$ and $\delta$. 

\begin{theorem}\label{approxQuadVarThm} There exists an absolute constant $D \geq 1$ such the following holds. Assume that $p \geq 3$. Let $V \subseteq G$ be a $(c_0, \delta, \varepsilon)$-approximate quadratic variety. Suppose that
\[\varepsilon \leq  (2^{-1}\delta)^{\exp\big(\log^D (O_p(c_0^{-1}))\big)}.\]
Then there exists a quadratic variety $Q$ such that $|Q| \leq \exp\Big(\exp(\log^D (O_p(c_0^{-1})))\Big) \cdot |V|$ and $|Q \cap V| \geq  \exp\Big(-\exp(\log^D (O_p(c_0^{-1})))\Big) \cdot |V|$.
\end{theorem}

\vspace{\baselineskip}

Let us remark that if we only had assumptions $|V| = \delta |G|$ and $\ex_{x,a,b,c} \mder_{a,b,c}\id_V(x) = c_0\delta^7$, the theorem above would imply that $V$ is either closely related to a quadratic variety or that $\|\id_V - \delta\|_{\mathsf{U}^2} \geq \Omega_{c_0, \delta}(1)$, becoming $\Omega_\delta(1)$ when $c_0 \geq \delta$, and proving existence of weak linear structure. A strong linear structure would be obtaining a non-trivial Fourier coefficient in the large spectrum of $\id_V$, namely finding $r\not =0$ such that $|\widehat{\id_V}(r)| \geq \Omega_{c_0}(\delta)$. However, the example based on Sidon sets shows that for some absolute constant $\theta > 0$ we may have $|\widehat{\id_V}(r)| \leq \delta^{1 + \theta}$ for all non-zero $r$, and so one cannot hope for a stronger result in a qualitative sense.\\
\indent Furthermore, the methods of this paper could likely prove the theorem above in the case of low characteristic (when $p = 2$) as well, but that would require the use of non-classical polynomials~\cite{TaoZiegler} in the final stages of the proof and a suitable modification of the statement of the result.\\

Finally, let us also note resemblance to nilspaces, introduced by Antol\'in Camarena and Szegedy~\cite{nilspaceDefn}, whose definition we now recall. A \emph{nilspace} is a set $X$ together with a collection of sets $\mathsf{C}^n(X) \subseteq X^{\{0,1\}^n}$, for each non-negative integer $n$, satisfying the following axioms:
\begin{itemize}
\item[\textbf{(i)}] (Composition) For every morphism $\phi \colon \{0,1\}^m \to \{0,1\}^n$ and every $c \in \mathsf{C}^n(X)$, we have $c \circ \phi \in \mathsf{C}^m(X)$. 
\item[\textbf{(ii)}] (Ergodicity) $\mathsf{C}^1(X) = X^{\{0,1\}}$.
\item[\textbf{(iii)}] (Corner completion) Let $c' \colon \{0,1\}^n \setminus \{(1,\dots, 1)\} \to X$ be such that every restriction of $c'$ to an $(n-1)$-face containing $(0,\dots,0)$ is in $\mathsf{C}^{n-1}(X)$. Then there exists $c \in \mathsf{C}^n(X)$ such that $c(v) = c'(v)$ for all $v \not= (1,\dots, 1)$. 
\end{itemize}
If every $(k+1)$-corner has a unique completion, then we say that $X$ is a \emph{$k$-step nilspace}. (In the theory of nilspaces, one also imposes topology on cubes, but here we are interested primarily in algebraic aspects, so we skip such details.) Such nilspaces occur naturally in the higher order Fourier analysis as they arise in the nilspace approach to the inverse theorems for uniformity norms.\\

The key assumption in Theorem~\ref{approxQuadVarThm} can be expressed as 
\[\mathbb{P}(x + a + b + c \in V| x, x + a, x+b, x+c, x+ a + b, x + a + c , x + b + c \in V) \geq c_0\]
so we may think of a an approximate quadratic variety as a combinatorial counterpart of a 2-step nilspace (though in the case of nilspaces the cube collections are abstract and thus more general).

\subsection{Proof overview}

Let us begin with some motivation. If $V$ is a $(c_0, \delta, \varepsilon)$-approximate quadratic variety, then we expect that $V$ is a $c_0$-dense subset of a quadratic variety $Q = \{x \in G \colon \beta(x,x) = 0\}$, where $\beta \colon G \times G \to \mathbb{F}_p^r$ is a symmetric bilinear map, and we know that $V$ is very $\mathsf{U}^2$-quasirandom. Furthermore, as we think of parameters $c_0, \delta$ and $\varepsilon$ as satisfying the relationship $c_0 \gg \delta \gg \varepsilon$, $Q$ has to have density comparable to $\delta$ and to inherit $\mathsf{U}^2$-quasirandomness from $V$. Since $Q$ is a quadratic variety, this in turn implies that $\beta$ is of very high rank. Our goal now is to identify $\beta$ using $V$. To that end, consider the intersection $V \cap V - a$. This is a subset of $\{x \in G \colon \beta(x + a, x + a) = \beta(x,x) = 0\} = \{x \in G \colon \beta(a,x) = -2^{-1}\beta(a,a), \beta(x,x) = 0\} = V \cap \{x \in G \colon \beta(a,x) = -2^{-1}\beta(a,a)\}$. Thus, we expect $V \cap V - a$ to be $\mathsf{U}^2$-quasirandom subset of the subspace coset $\{x \in G \colon \beta(a,x) = -2^{-1}\beta(a,a)\}$. Given $\mathsf{U}^2$-quasirandomness, the set of large values of the convolution $\id_{V \cap V - a} \conv \id_{V \cap V - a}$ should then be equal to the whole subspace $\{x \in G \colon \beta(a,x) = 0 \}$. This motivates the study of subspaces $W_a$ that arise from convolution of $\id_{V\cap V-a}$ via methods such as Bogolyubov argument. We expect these subspaces $W_a$ to be related to $\{x \in G \colon \beta(a,x) = 0\}$.\\
\indent Hence, we now have a collection of subspaces $W_a$, for $a \in G$, of density comparable to $\delta$. However, this is not an arbitrary family of subspaces. In fact, the algebraic structure of indices plays an important role. Namely, if $a_1, a_2, a_3, a_4$ form an additive quadruple, thus $a_1 + a_2 = a_3 + a_4$, we have that $x \in W_{a_1} \cap W_{a_2} \cap W_{a_3}$ implies $0 = \beta(a_1, x) + \beta(a_2, x) - \beta(a_3, x) = \beta(a_1 + a_2 - a_3, x) = \beta(a_4, x)$ and thus we expect that a relationship along the lines of
\begin{equation}W_{a_1} \cap W_{a_2} \cap W_{a_3} \subseteq W_{a_4}\label{linearsystemProp}\end{equation}
holds. The first major step of the proof of Theorem~\ref{approximatePolynomialsAndApproximateVars} achieves this and it appears as Theorem~\ref{mainstep1thm} later in the paper.

\begin{theorem}[\textbf{Step 1}]\label{step1IntroThm}There exists an absolute constant $D \geq 1$ such that the following holds. Let $V \subseteq G$ be a $(c_0, \delta, \varepsilon)$-approximate quadratic variety and suppose that $\varepsilon \leq \exp\Big(- \log^D(2c_0^{-1})\Big) \delta^{288}$. Then there exist a quantity $c_1$, a set $A \subseteq G$ and a collection of subspaces $W_a \leq G$ indexed by elements $a \in A$ such that
\begin{itemize}
\item[\textbf{(i)}] $\exp(-\log^D (2c_0^{-1})) \leq c_1 \leq 1$,
\item[\textbf{(ii)}] $|A| \geq c_1 |G|$,
\item[\textbf{(iii)}] for each $a \in A$ and $b \in W_a$ we have ${\conv}^{(8)} \id_{V \cap V -a}(b) \geq c_1 \delta^{15}$,
\item[\textbf{(iv)}] $c_1 \leq \frac{|W_a|}{\delta|G|} \leq c_1^{-1}$ holds for all $a \in A$,
\item[\textbf{(v)}] for $r \in [9]$, for all but at most $D \varepsilon \delta^{-32r} |G|^r$ choices of $(a_1, \dots, a_r) \in A^r$ we have
\[\frac{|W_{a_1} \cap W_{a_2} \cap \dots \cap W_{a_r}|}{\delta^r|G|} \leq c_1^{-1},\] 
 and,
\item[\textbf{(vi)}] for at least $c_1 |A|^6$ 6-tuples $(b_1, b_2, b_3, x_2, y_3, z_1) \in A^6$ we have that $b_1 + b_2 - b_3,  x_2 - b_2 + b_3, y_3  + b_1 - b_3, b_1 + b_2 - z_1\in A$ and 
\[c_1 \delta^6 |G| \leq |W_{b_1} \cap W_{b_2} \cap W_{b_3} \cap W_{b_1 + b_2 - b_3} \cap W_{x_2} \cap W_{x_2 - b_2 + b_3} \cap W_{y_3} \cap W_{y_3  + b_1 - b_3} \cap W_{z_1} \cap W_{b_1 + b_2 - z_1}|.\]
\end{itemize}
\end{theorem}

\vspace{\baselineskip}

We remark that ${\conv}^{(8)} \id_{V \cap V -a}$ stands for an iterated convolution of the function $\id_{V \cap V -a}$ with itself, see~\eqref{iteratedConvDefn} for a precise definition. Notice that item \textbf{(vi)} is considerably stronger than what we mentioned in the motivation discussion, where an approximate version of $W_{b_1} \cap W_{b_2} \cap W_{b_3} \subseteq W_{b_1 + b_2 - b_3}$ would be $|W_{b_1} \cap W_{b_2} \cap W_{b_3} \cap W_{b_1 + b_2 - b_3}| \geq c_1 \delta^3 |G|$. We shall return to this discrepancy in the discussion of the second major step of the proof.\\
\indent We shall refer to a collection of subspaces in satisfying all properties but \textbf{(iii)} in conclusion of the theorem above as an \emph{approximate quasirandom linear systems of subspaces}. To explain the terminology, linearity stands for property~\eqref{linearsystemProp} for additive quadruples $(a_1, a_2, a_3, a_4)$, quasirandomness suggests that intersections of a vast majority subspaces behave as given by property \textbf{(v)} in the conclusion of the theorem, and all properties are of an approximate form rather than exact.\\

The second major step of the proof is to characterize approximate quasirandom linear systems of subspaces. Recall that we expect $W_a$ to be related to $\{x \in G \colon \beta(a,x) = 0\}$. With this in mind, we expect that $(W_a)_{a \in A}$ comes from some bilinear map. The second step is to obtain such a result, which appears as Theorem~\ref{approximatelinearsyssub}. The assumptions of the theorem are essentially the same as the conclusion of Theorem~\ref{step1IntroThm}.

\begin{theorem}[\textbf{Step 2}]\label{step2IntroThm}There exists an absolute constant $D \geq 1$ such that the following holds. Let $c > 0$ and let $d$ be a positive integer. Let $A \subseteq G$ be a set of size $|A| \geq c|G|$ and let $W_a \leq G$ be a subspace of codimension $d$ for each $a \in A$. Suppose that 
\[|W_{a_1} \cap W_{a_2} \cap \dots \cap W_{a_r}| \leq Kp^{-rd}|G|\]
holds for all but at most $\eta |G|^r$ $r$-tuples $(a_1, a_2, \dots, a_r) \in A^r$ for each $r \in [9]$. Assume furthermore that for at least $c |G|^6$ triples $(a, b_1, b_2, b_3, x_2, x_3, y_1, y_3, z_1, z_2) \in A^{10}$ we have 
\begin{itemize}
\item 
\[a = b_1 + b_2 - b_3,\,\, x_3 = x_2 - b_2 + b_3, \,\, y_1 = y_3 + b_1 - b_3,\,\, z_2 = b_1 + b_2 - z_1,\]
and
\item the subspaces
\[W_a \cap W_{b_1} \cap W_{b_2} \cap W_{b_3} \cap W_{x_2} \cap W_{x_3} \cap W_{y_1} \cap W_{y_3} \cap W_{z_1} \cap W_{z_2}\]
has size at least $K^{-1} p^{-6d}|G|$.
\end{itemize}
Then, provided $\eta \leq 2^{-31}c^3$, there exist parameters $c' \geq \exp\Big(-\exp\Big((\log (2c^{-1}) + \log_p K)^{D}\Big)\Big)$ and $r \leq \exp\Big((\log (2c^{-1}) + \log_p K)^D\Big)$, set $A' \subseteq A$ and a map $\Phi \colon G \times \mathbb{F}_p^d \to G$, affine in the first variable and linear in the second, such that $|A'| \geq c'  |G|$ and for each $a \in A'$ we have $|\on{Im} \Phi(a, \cdot) \cap W_a^\perp| \geq c' p^d$. Moreover, there exists a subspace $\Lambda \leq \mathbb{F}_p^d$ of dimension $r$ such that whenever $\lambda \notin \Lambda$ we have
\[\exx_{x,y} \omega\Big(\Phi(x, \lambda) \cdot y\Big) \leq \Big(\eta {c'}^{-2}\Big)^{1/2r}.\]
\end{theorem}

\vspace{\baselineskip}

Finally, once we have obtained a biaffine map $\Phi$ that parametrizes the family of subspaces $(W_a)_{a \in A}$, we have to relate it to the approximate variety $V$ that we start with. This is the content of the final major step of the proof. The proposition below appears as Proposition~\ref{finalsteprop}. 

\begin{proposition}[\textbf{Step 3}]\label{step3IntroThm}There is an absolute constant $D \geq 1$ such that the following holds. Let $c, \delta, \varepsilon > 0$ and $d \in \mathbb{N}$ be such that $c \leq \delta p^{d} \leq c^{-1}$. Suppose that $\varepsilon \leq (2^{-1}c \delta)^D$. Let $V \subseteq G$ be a set of density $\delta$ such that $\|\id_V - \delta\|_{\mathsf{U}^2} \leq \varepsilon$. Suppose that we are also given a subset $A \subseteq G$ of size $|A| \geq c |G|$, a subspace $W_a \leq G$ for each $a \in A$ and a bilinear map $\beta \colon G \times G \to \mathbb{F}_p^d$ such that 
\begin{itemize}
\item[\textbf{(i)}] for each $\lambda \in \mathbb{F}_p^d \setminus \{0\}$ we have $\on{bias} \lambda \cdot \beta \leq \varepsilon$,
\item[\textbf{(ii)}] for each $a \in A$ we have $|W_a \cap \{b \in G \colon \beta(a,b) = 0\}| \geq c p^{-d} |G|$,
\item[\textbf{(iii)}] for each $a \in A$ and $b \in W_a$ we have ${\conv}^{(8)} \id_{V \cap V - a}(b) \geq c \delta^{15}$. 
\end{itemize}
Then there exists a quadratic variety $Q \subseteq G$ of size $|Q| \leq (2c^{-1})^D \delta |G|$ such that $|Q \cap V| \geq \exp\Big(-\log^D (2c^{-1})\Big)\delta |G|$. Moreover, $Q$ is defined as $\{x \in G \colon \gamma(x,x) - \psi(x) = \mu\}$ for a symmetric bilinear map $\gamma \colon G \times G \to \mathbb{F}_p^{\tilde{d}}$, an affine map $\psi \colon G \to \mathbb{F}_p^{\tilde{d}}$ and $\mu \in \mathbb{F}_p^{\tilde{d}}$, where $\on{bias} \lambda \cdot \gamma \leq \varepsilon$ for all $\lambda \not= 0$, for some $d - O(\log_p (2c^{-1})) \leq \tilde{d} \leq d$.
\end{proposition}

\vspace{\baselineskip}

Let us now say something about the proof of each step.\\

When it comes to \textbf{Step 1}, we rely heavily on the $\mathsf{U}^2$-uniformity assumption on $V$, which itself is sufficient to prove some properties in the conclusion of Theorem~\ref{step1IntroThm}. Most basic consequence of uniformity is a standard fact that intersections of translates of $V$ behave quasirandomly (see Lemma~\ref{varietyTranslatesprops1}). We can control more involved expressions as well, for example
\begin{align*}&\exx_{x, a_1, a_2, a_3} \id_{V \cap V - a_1} \conv \id_{V \cap V - a_1}(x) \id_{V \cap V - a_2} \conv \id_{V \cap V - a_2}(x) \\
&\hspace{2cm}=\exx_{x, y, z, a_1, a_2, a_3} \id_{V \cap V - a_1}(y + x) \id_{V \cap V - a_1}(y) \id_{V \cap V - a_2}(z + x) \id_{V \cap V - a_2}(z) \\
&\hspace{2cm}=\exx_{x, y, z, a_1, a_2, a_3} \id_{V}(y + x + a_1) \id_V(y + x) \id_{V}(y + a_1) \id_V(y) \id_{V}(z + x + a_2) \id_V(z + x) \id_{V}(z + a_2) \id_V(z) \end{align*}
can be shown to be about $\delta^8$. If we think of $W_a$ as the set of elements $x$ where $\id_{V \cap V - a_1} \conv \id_{V \cap V - a_1}(x)$ is about $\delta^3$, then the bound above shows that $|W_{a_1} \cap W_{a_2}|$ is $O(\delta^2|G|)$ most of the time. A more involved argument of this form can be used to prove property \textbf{(v)} of Theorem~\ref{step1IntroThm}.\\
\indent We remark that this resembles the true complexity problem introduced by Gowers and Wolf~\cite{TrueCompl}. However, these are instances of systems of the true complexity 1, and the arguments of this form that we need in this paper are elementary, as we are concerned with the $\mathsf{U}^2$ norm.\\
\indent Let us also note that there are expressions that cannot be directly controlled by $\mathsf{U}^2$ norm, but where we can still easily prove correct upper bounds by neglecting some of the terms. For example, the distribution of $|V \cap V - a \cap V - b \cap V - a - b|$ depends on the quadratic structure of $V$, but it is always at most $|V \cap V - a \cap V - b|$ which is about $\delta^3 |G|$ vast majority of the time.\\ 
\indent However, in order to prove an approximate version of property~\eqref{linearsystemProp}, we need to use the assumption of having many additive cubes in $V$ and, simplifying things greatly by assuming that $W_a$ is the set of elements $x$ where $\id_{V \cap V - a_1} \conv \id_{V \cap V - a_1}(x)$ is about $\delta^3$, we would need a bound of the form
\begin{align*}&\Omega_{c_0}(\delta^{15}) \leq \exx_{a_1, a_2, a_3, x} \id_{V \cap V - a_1} \conv \id_{V \cap V - a_1}(x) \id_{V \cap V - a_2} \conv \id_{V \cap V - a_2}(x)\\
&\hspace{4cm}\id_{V \cap V - a_3} \conv \id_{V \cap V - a_3}(x) \id_{V \cap V - a_1 - a_2 + a_3} \conv \id_{V \cap V - a_1 - a_2 + a_3}(x).\end{align*}
Unlike proving upper bounds on expressions that are not directly controllable by $\mathsf{U}^2$ norm, it is surprisingly challenging to prove such a lower bound. To get the lower bound, we first observe the following \emph{duality} property of convolutions of indicator functions of intersections
\[\id_{V \cap V - a} \conv \id_{V \cap V - a}(b) =\id_{V \cap V - b} \conv \id_{V \cap V - b}(a).\]
This observation allows us to turn the expression above into
\begin{align}&\exx_{a_1, a_2, a_3, x} \id_{V \cap V - x} \conv \id_{V \cap V - x}(a_1) \id_{V \cap V - x} \conv \id_{V \cap V - x}(a_2)\nonumber\\
&\hspace{4cm}\id_{V \cap V - x} \conv \id_{V \cap V - x}(a_3) \id_{V \cap V - x} \conv \id_{V \cap V - x}(a_1 + a_2 - a_3).\label{hardlowerbound}\end{align}
Using the properties of $V$, we show that for many $x$, the set $S_x$ of $a\in G$ such that $\id_{V \cap V - x} \conv \id_{V \cap V - x}(a) \geq \Omega_{c_0}(\delta^3)$ is $\Omega_{c_0}(1)$-dense subset of some subspace $U_x$ of size $|U_x| \geq \Omega_{c_0}(\delta |G|)$ (observe that this is related to other properties in the conclusion of Theorem~\ref{step1IntroThm}). Then a lower bound on~\eqref{hardlowerbound} follows from the fact that $S_x$ has at least $\Omega_{c_0}(\delta^3 |G|^3)$ additive quadruples. Interestingly, we needed to use results in the spirit of Bogolyubov argument in this stage of the proof, as it was essential that we have the subspace structure.\\ 
\indent Finally, as we actually need to prove property \textbf{(vi)} of Theorem~\ref{step1IntroThm} instead of a simpler variant such as~\eqref{linearsystemProp}, the proof is more involved than the sketch above, but the sketch indicates the key ideas.\\

In \textbf{Step 2}, the following observation, stated as Lemma~\ref{addquadshomsexist}, is the main tool of passing from the given family of subspaces to approximate bilinear structure. Namely, if four subspaces $U_1, \dots, U_4$ of dimension $d$ satisfy $K^{-1} p^{3d} \leq |U_{i_1} + U_{i_2} + U_{i_3}|$ for any three distinct indices $i_1, i_2$ and $i_3$ and $|U_1 + U_2 + U_3 + U_4| \leq Kp^{3d}$, then for any linear isomorphism $\phi_4 \colon \mathbb{F}_p^d \to U_4$ there exist linear isomorphisms $\phi_i \colon \mathbb{F}_p^d \to U_i$ for $i \in [3]$ such that
\begin{equation}\on{rank}(\phi_1 + \phi_2  - \phi_3 - \phi_4) \leq O(\log_p K).\label{addquadsubconcl}\end{equation}
We remark that we do not only say that there are 4 isomorphisms satisfying~\eqref{addquadsubconcl}, but rather that there exist suitable $\phi_2, \phi_3$ and $\phi_4$ for \emph{any given} $\phi_1$, which will be crucial in the proof. This fact also shows that we may think of subspaces satisfying~\eqref{linearsystemProp} as additive quadruples of subspaces. Observe that these properties are satisfied by the orthogonal complements of subspaces in our family (we may easily ensure that all $W_a^\perp$ are of same size).\\ 
\indent The approach to proving Theorem~\ref{step2IntroThm} is to find a suitable element $a \in A$, fix an isomorphism $\theta \colon \mathbb{F}_p^d \to W_a^\perp$ and apply the observation above to additive quadruples in $A$ involving $a$ to get isomorphisms to other $W_b^\perp$. This will give us an approximate homomorphism between $G$ and the space of linear maps $\on{Hom}(\mathbb{F}_p^d,G)$, where the distance between elements is measured by rank, from which we shall obtain an exact homomorphism. We achieve this by using a result of Kazhdan and Ziegler~\cite{ApproxCohomology} (see Theorem~\ref{friemanforlinearhom}), which we reproved using the inverse theorem for Freiman bihomomorphisms~\cite{U4paper,KimLiTidor} to get quantitative bounds.\\
\indent It turns out that in this step, the weaker assumption of having many triples $(b_1, b_2, b_3)$ such that $|W_{b_1} \cap W_{b_2} \cap W_{b_3} \cap W_{b_1 + b_2 - b_3}| \geq c_1 \delta^3 |G|$ is not sufficiently strong for the proof to work and that we need more involved configurations of points. The reason is that in defining the linear isomorphisms $\mathbb{F}_p^d \to W_b^\perp$ we need to ensure that they are well-defined and that these linear isomorphisms also respect many additive quadruples.\\

Finally, in \textbf{Step 3}, we first need to show that the given map $\beta$ can be replaced by a symmetric bilinear map. To that end, we use Green-Tao symmetry argument~\cite{GreenTaoU3} to first show that $(x,y) \mapsto \lambda \cdot (\beta(x,y) - \beta(y,x))$ has small rank for many vectors $\lambda \in \mathbb{F}_p^d$. To pass to an exactly symmetric map, we have to make use of the solution of partition versus analytic problem for trilinear forms, first proved by Green and Tao~\cite{GreenTaoPolys} in the case of polynomials, with essentially optimal bounds obtained by Adiprasito, Kazhdan and Ziegler~\cite{AKZ} and by Cohen and Moshkovitz~\cite{CohenMoshkovitz}. Note that it the high characteristic case (when $p \geq 3$) we may usually replace an approximately bilinear form $\gamma(x,y)$ by a simple mean $2^{-1}(\gamma(x,y) + \gamma(y,x))$. However, as we are concerned with $\beta$ whose codomain is subspace of somewhat large dimension, we have to use partition versus analytic problem at some point, even if we use the mean trick at each coordinate of $\beta$ (see for example Lemma 2.8 in~\cite{LukaU56}). Once we may assume that $\beta$ is symmetric, we expect that the approximate variety $V$ comes from a variety defined by $\beta(x,x) + \gamma(x) = \lambda$ for some linear map $\gamma$, which we still have to identify. To that end, we consider intersection of $V \cap V - a$ with $\{x \in G \colon \beta(a,x) = \lambda\}$ for various $\lambda$. As it turns out, for many $a$, there is a value $\lambda(a)$ for which this intersections is almost the whole of $V \cap V - a$. We then use additional graph-theoretic arguments to show that $\lambda(a)$ is an approximate homomorphism, from which we may then pass to a linear map and conclude the proof.

\subsection{Regularity lemmas are insufficient}

A natural approach to proving Theorem~\ref{approxQuadVarThm} is to use arithmetic regularity lemmas of Green and Tao~\cite{GreenTaoReg}. In this short subsection, we briefly discuss why such an approach is not helpful for this problem. In this setting, we could in principle find a bilinear map $\beta \colon G \times G \to \mathbb{F}^r$ and functions $d \colon \mathbb{F}_p^r \to \mathbb{D}$, $f_{\text{err}}, f_{\text{unif}} \colon G \to \mathbb{D}$ such that 
\[\id_V(x) = d(\beta(x,x)) + f_{\text{err}}(x) + f_{\text{unif}}(x),\]
where $f_{\text{err}}$ has small $L^2$ norm, $f_{\text{unif}}$ has extremely small $\|\cdot\|_{\mathsf{U}^3}$ norm and
\[d(\lambda) = \frac{\ex_{x} \id_V(x) \id(\beta(x,x) = \lambda)}{\ex_{x}\id(\beta(x,x) = \lambda)},\]
which comes from projection of $\id_V$ onto layers $\{x \colon \beta(x,x) = \lambda\}$ defined by $\beta$. For simplicity, we ignore the $\ell^2$ error terms and assume that we have perfect approximation, and also that $\beta$ itself is of high rank (we expect this to be the case from the expected structure of $V$). In particular, function $d$ simplifies to $d(\lambda) = |G|^{-1}p^r|V \cap \{x \colon \beta(x,x) = \lambda\}|$. Observe also that the high rank of $\beta$ implies that
\[(x, a, b, c) \mapsto \Big(\beta(x,x), \beta(x + a,x + a), \dots, \beta(x + b + c,x + b + c)\Big)\]
is equidistributed in $(\mathbb{F}_p^r)^7$ (here we listed all points of additive cube associated with $x,a,b,c$ except $x + a + b + c$), while $\beta$ being a bilinear form implies the identity
\begin{align*}\beta&(x + a + b + c, x+ a + b + c) = \beta(x,x) - \beta(x + a, x + a) - \beta(x + b, x + b) - \beta(x + c, x + c)\\
&\hspace{1cm} + \beta(x +a + b, x + a + b) + \beta(x +a +c, x + a + c) + \beta(x +b+c, x + b+c).\end{align*}
Using these facts and given that the error terms are negligible, the count of 3-dimensional additive cubes in $V$ essentially becomes
\[c_0 \delta^7 \leq \exx_{\lambda_1, \dots, \lambda_7 \in \mathbb{F}_p^r} d(\lambda_1)\dots d(\lambda_7) d(\lambda_1 + \lambda_2 + \lambda_3 + \lambda_4 - \lambda_5 - \lambda_6 - \lambda_7) = \sum_{\gamma \in \mathbb{F}^r} |\hat{d}(\gamma)|^8.\] 

Note that $d(\lambda) \leq 1$ for all $\lambda$ and that
\[\sum_{\gamma \in \mathbb{F}^r_p} |\hat{d}(\gamma)|^2 = \exx_{\lambda \in \mathbb{F}_p^r} |d(\lambda)|^2 \leq \exx_{\lambda \in \mathbb{F}_p^r} d(\lambda) = \frac{1}{|G|} \sum_{\lambda \in \mathbb{F}_p^r}|V \cap \{x \colon \beta(x,x) = \lambda\}| = \frac{|V|}{|G|} = \delta.\]

We thus get some $\gamma$ such that $|\hat{d}(\gamma)| \geq \sqrt[6]{c_0} \delta$. We may assume that $\gamma \not= 0$ as the contribution from $\gamma = 0$ above is $\delta^8$.\footnote{In fact, we expect that $V$ is the union of $\{\beta = \lambda\}$ for some subspace $\Lambda$ such that $|\Lambda| = \delta p^r$. If so, we would also have $d(\lambda) = \id_\Lambda(\lambda)$ and thus $\hat{d}(\gamma) = \frac{|\Lambda|}{p^r} \id_{\Lambda^\perp}(\gamma) = \delta\id_{\Lambda^\perp}(\gamma)$.} Thus,
\begin{align*}\sqrt[6]{c_0} \delta \leq & \Big|\exx_{\lambda \in \mathbb{F}_p^r} |G|^{-1}p^r|V \cap \{x \colon \beta(x,x) = \lambda\}| \omega^{-\lambda \cdot \gamma} \Big|\\
= & \Big|\sum_{\lambda \in \mathbb{F}_p^r} \frac{|V \cap \{x \colon \beta(x,x) = \lambda\}|}{|G|} \omega^{-\lambda \cdot \gamma} \Big|\\
= & \Big|\sum_{\mu \in \mathbb{F}_p} \sum_{\ssk{\lambda \in \mathbb{F}_p^r\\\text{s.t. }\gamma \cdot \lambda = \mu}} \frac{|V \cap \{x \colon \beta(x,x) = \lambda\}|}{|G|} \omega^{-\mu} \Big|\\
= & \Big|\sum_{\mu \in \mathbb{F}_p} \frac{|V \cap \{x \colon \gamma \cdot \beta(x,x) = \mu\}|}{|G|} \omega^{-\mu} \Big|\\
= & \Big|\sum_{\mu \in \mathbb{F}_p} \Big(\frac{|V \cap \{x \colon \gamma \cdot \beta(x,x) = \mu\}|}{|G|} - \delta\Big) \omega^{-\mu} \Big|.\end{align*}
Hence, for some $\mu \in \mathbb{F}_p$ and codimension 1 quadratic variety $B = \{x \colon \gamma \cdot \beta(x, x) = \mu\}$ we conclude that $V$ gets a density increment of $p^{-1}\sqrt[6]{c_0} \delta $ on $B$. In particular, after $s$ steps, we may only guarantee density $\delta\Big(1 + p^{-1}\sqrt[6]{c_0}\Big)^s$ on codimension $s$ quadratic variety. Proceeding in this fashion, provided $c_0$ is sufficiently small, we expect that we need $s = K \log_p \delta^{-1}$ steps in order to get density $\Omega_{c_0}(1)$ on some quadratic variety, where $K$ is arbitrarily large. But, such a variety would then be of codimension $s$ and would have density $p^{-s} \leq \delta^K$ as our bilinear map is quasirandom, so we cannot obtain the claimed result in this fashion.\\

\noindent\textbf{Acknowledgements.} I would like to thank Tim Gowers for introducing me to the problem of proving the structure theorem for approximate quadratic varieties. This work was supported by the Ministry of Science, Technological Development and Innovation of the Republic of Serbia through the Mathematical Institute of the Serbian Academy of Sciences and Arts.

\section{Preliminaries}

\noindent\textbf{Notation.} The fixed prime $p$ is assume to be at least 3 throughout the paper and dependencies on $p$ in bounds are suppressed.\\
\indent In this paper, we shall frequently consider subsets of products of two vectors spaces $G$ and $H$. Given $X \subset G \times H$ and an element $x \in G$, we write $X_{x \bcdot} = \{y \in H \colon (x,y) \in X\}$ for the \emph{(vertical) slice} of $X$ in the column indexed by $x$. Likewise, for an element $y \in H$, we write $X_{\bcdot y} = \{x \in G \colon (x,y) \in X\}$ for the \emph{(horizontal) slice} of $X$ in the row indexed by $y$.\\
For two functions $f,g \colon G \to \mathbb{R}$ we define \emph{convolution} as $f \conv g(x) = \ex_{y \in G} f(y + x) \overline{g(x)}$ (note that this is non-standard, as we average over pairs of elements whose difference is $x$, rather than their sum). We also need \emph{iterated convolution} (of \emph{order} $2k$) of a function $f \colon G \to \mathbb{R}$, defined as 
\begin{equation}\label{iteratedConvDefn} {\conv}^{(2k)} f(a) = \ex_{x_1, \dots, x_{2k-1} \in G} f(x_1)f(x_2)\cdots f(x_{2k-1}) f\Big(\sum_{\ell \in [2k-1]}(-1)^{\ell + 1} x_i - a\Big).\end{equation}
Note that the order of iterated convolution $2k$ stands for the number of terms, rather than number of convolutions needed to give the expression above. Also note that there are no complex conjugates in the expression above. This is done to simplify the notation for iterated convolution as it will only be used for real-valued functions in this paper.\\
The \emph{discrete multiplicative derivative operator} $\mder_a$ for shift $a \in G$ is defined by $\mder_a f(x) = f(x + a)\overline{f(x)}$ for functions $f \colon G \to \mathbb{C}$.\\
To save writing in situations where we have many indices of variables appearing in predictable patterns, we use the following convention. Instead of denoting a sequence of length $m$ by $(x_1, \dots, x_m)$, we write $x_{[m]}$, and for $I\subset[m]$ we write $x_I$ for the subsequence with indices in $I$.\\

We need a few auxiliary results. The first one is a robust version of Bogolyubov-Ruzsa lemma, which is essentially due to Schoen and Sisask and builds upon the work of Sanders. 

\begin{theorem}\label{addquadsFreiman} Let $A \subset G$ be a subset having at least $\alpha |A|^3$ additive quadruples. Then there exists a subspace $V \subseteq 2A -2A$ of size $|V| \geq \exp\Big(- O\Big(\log^{O(1)}(2c^{-1})\Big)\Big)|A|$ such that the following holds. Every $y \in V$ can be expressed as $y = a_1 + a_2 - a_3 - a_4$ with $a_1, a_2, a_3, a_4 \in A$ in at least $\alpha^{O(1)}|A|^3$ many ways.\end{theorem}

\begin{proof}Apply Balog-Szemer\'edi-Gowers theorem to find a subset $A' \subseteq A$ such that $|A'| \geq \Omega(\alpha^{O(1)}|A|) $ and $|A' + A'| \leq O(\delta^{-O(1)} |A'|)$. The theorem follows from results of Sanders (Theorem A.2 for arbitrary prime $p$ instead of $p = 2$ in~\cite{Sanders}) and Schoen and Sisask (Theorem 5.1 in~\cite{SchSisRob}).\end{proof}

A closely related result is the inverse theorem for approximate homomorphisms, which we already stated in the introduction as Theorem~\ref{approxHommInvIntro}. We make use of a more efficient version, which is proved using Balog-Szemer\'edi-Gowers theorem and Sanders's results on Bogolyubov-Ruzsa lemma~\cite{Sanders}.

\begin{theorem}\label{approxHommInv}Let $G$ and $H$ be finite-dimensional vector spaces over $\mathbb{F}_p$. Let $A \subseteq G$ be a subset and let $\phi \colon A \to H$ be a map which respects at least $c |G|^3$ additive quadruples in $A$. Then there exists an affine map $\Phi \colon G\to H$ such that $\phi = \Phi$ holds for at least $\exp\Big(-\log^{O(1)} (2 c^{-1})\Big)$ points in $A$.\end{theorem}

Recall that the \emph{bias} of a multilinear form $\phi \colon G^k \to \mathbb{F}_p$ is defined as 
\[\on{bias} \phi = \exx_{x_1, \dots, x_k} \omega^{\phi(x_1, \dots, x_k)},\]
where $\omega = \exp\Big(\frac{2 \pi i}{p}\Big)$. This quantity is a measure of how far from being quasirandom the given form $\phi$ is. We need the inverse theorem for biased trilinear forms. First results in this spirit were proved by Green and Tao~\cite{GreenTaoPolys} for the case of polynomials and a multilinear variant was proved by Bhowmick and Lovett~\cite{BhowLov}. The version below, which has essentially optimal bounds, is due to Adiprasito, Kazhdan and Ziegler~\cite{AKZ} and due to Cohen and Moshkovitz~\cite{CohenMoshkovitz}.

\begin{theorem}[Inverse theorem for biased trilinear forms]\label{invTheoremBiased}Suppose that $\phi \colon G \times G \times G \to \mathbb{F}_p$ is a trilinear form such that $\on{bias} \phi \geq c$. Then there exists a positive integer $r \leq O(\log_p c^{-1})$, linear forms $\alpha_1, \dots,$ $\alpha_r,$ $\beta_1, \dots,$ $\beta_r,$ $\gamma_1, \dots,$ $\gamma_r \colon G \to \mathbb{F}_p$ and bilinear forms $\alpha'_1, \dots,$ $\alpha'_r,$ $\beta'_1, \dots,$ $\beta'_r,$ $\gamma'_1, \dots,$ $\gamma'_r \colon G \times G\to \mathbb{F}_p$ such that
\[\phi(x,y,z) = \sum_{i \in [r]} \alpha_i(x) \alpha'_i(y,z) + \beta_i(y) \beta_i'(x,z) + \gamma_i(z) \gamma'_i(x,y)\]
holds for all $x, y, z \in G$.
\end{theorem}

Next, we need a result on the number of certain arrangements of points (related to additive quadruples) inside dense subsets of vector spaces.

\begin{lemma}\label{multipleAddQuadsBound}Let $G$ be a finite-dimensional vector space over $\mathbb{F}_p$ and let $A \subset G$ be a subset of density $c$. Then
\begin{align*}&\exx_{b_1, b_2, b_3, x_2, y_3, z_1 \in G} \id_{A}(b_1 + b_2 - b_3) \id_{A}(b_1) \id_{A}(b_2) \id_{A}(b_3) \id_{A}(x_2) \id_{A}(x_2 - b_2 + b_3) \\
&\hspace{6cm}\id_{A}(y_3) \id_{A}(y_3 + b_1 - b_3) \id_{A}(z_1) \id_{A}(b_1 + b_2 - z_1) \geq c^{32}.\end{align*}
\end{lemma}

\begin{proof}Since the density of $A$ is $c$ we have\footnote{In the first step we prove the standard fact that the number of additive quadruples in a set in $G$ of density $c$ is at least $c^4 |G|^3$. We could have simply stated that fact, but since the rest of the proof uses identical method, we opted to start from density assumption.}
\begin{align*}c^4 \leq \Big(\exx_x \id_A(x)\Big)^4 = \Big(\exx_{x, y} \id_A(x) \id_A(y)\Big)^2 =& \Big(\exx_{d} \Big(\exx_x \id_A(x) \id_A(x + d)\Big)\Big)^2 \hspace{2cm}\text{(change of variables)}\\
 \leq &\exx_{d} \Big(\exx_x \id_A(x) \id_A(x + d)\Big)^2\hspace{2cm}\text{(by Cauchy-Schwarz)}\\
 = &\exx_{b_1, b_2, b_3} \id_{A}(b_1 + b_2 - b_3) \id_{A}(b_1) \id_{A}(b_2) \id_{A}(b_3).\end{align*}

Let us make a change of variables and use $a = b_1 + b_2 - b_3$ instead of $b_3$. Then
\begin{align*}c^8 \leq &\Big(\exx_{a, b_1, b_2} \id_{A}(a) \id_{A}(b_1) \id_{A}(b_2) \id_{A}(b_1 + b_2 - a)\Big)^2 = \Big(\exx_{a, b_1} \id_{A}(a) \id_{A}(b_1) \Big(\exx_{b_2} \id_{A}(b_2) \id_{A}(b_1 + b_2 - a)\Big)\Big)^2\\
\leq & \exx_{a, b_1} \id_{A}(a) \id_{A}(b_1) \Big(\exx_{b_2} \id_{A}(b_2) \id_{A}(b_1 + b_2 - a)\Big)^2\hspace{2cm}\text{(by Cauchy-Schwarz)}\\
= & \exx_{a, b_1} \id_{A}(a) \id_{A}(b_1) \exx_{b_2, x_2} \id_{A}(b_2) \id_{A}(b_1 + b_2 - a) \id_{A}(x_2) \id_{A}(b_1 + x_2 - a)\\
= & \exx_{a, b_1, b_2, x_2} \id_{A}(a) \id_{A}(b_1) \id_{A}(b_2) \id_{A}(b_1 + b_2 - a) \id_{A}(x_2) \id_{A}(b_1 + x_2 - a).\end{align*}

Apply Cauchy-Schwarz inequality another time to get
\begin{align*}c^{16} \leq & \Big(\exx_{a, b_2, x_2} \id_{A}(a) \id_{A}(b_2)   \id_{A}(x_2)\Big( \exx_{b_1} \id_{A}(b_1) \id_{A}(b_1 + b_2 - a)\id_{A}(b_1 + x_2 - a)\Big)\Big)^2\\
\leq &\exx_{a, b_2, x_2} \id_{A}(a) \id_{A}(b_2)   \id_{A}(x_2)\Big( \exx_{b_1} \id_{A}(b_1) \id_{A}(b_1 + b_2 - a)\id_{A}(b_1 + x_2 - a)\Big)^2\\
= &\exx_{a, b_2, x_2} \id_{A}(a) \id_{A}(b_2)   \id_{A}(x_2) \exx_{b_1, y_1} \id_{A}(b_1)\id_{A}(y_1) \id_{A}(b_1 + b_2 - a) \id_{A}(y_1 + b_2 - a)\id_{A}(b_1 + x_2 - a)\id_{A}(y_1 + x_2 - a)\\
\leq &\exx_{a, b_1, b_2, x_2, y_1}\id_{A}(a)\id_{A}(b_1) \id_{A}(b_2)\id_{A}(b_1 + b_2 - a)   \id_{A}(x_2) \id_{A}(b_1 + x_2 - a)\id_{A}(y_1) \id_{A}(y_1 + b_2 - a),\end{align*}
where we omitted term $\id_{A}(y_1 + x_2 - a)$ in the last line, which is fine as all terms take values in the interval $[0,1]$.\\
\indent Make another change of variables and use $y_3 = y_1 + b_2 - a$ instead of $y_1$ so we get 
\[c^{16} \leq \exx_{a, b_1, b_2, x_2, y_3}\id_{A}(a)\id_{A}(b_1) \id_{A}(b_2)\id_{A}(b_1 + b_2 - a)   \id_{A}(x_2) \id_{A}(b_1 + x_2 - a)\id_{A}(y_3) \id_{A}(y_3 + a - b_2).\]
Make a further change of variables and use $b_3 = b_1 + b_2 - a$ instead of $b_2$. Thus
\begin{align*}c^{32} \leq &\Big(\exx_{a, b_1, b_3, x_2, y_3} \id_{A}(a)\id_{A}(b_1) \id_{A}(b_3 + a - b_1)\id_{A}(b_3)   \id_{A}(x_2) \id_{A}(b_1 + x_2 - a)\id_{A}(y_3) \id_{A}(y_3 + b_1 - b_3)\Big)^2 \\
= & \Big(\exx_{a, b_3, x_2, y_3} \id_{A}(a) \id_{A}(b_3)   \id_{A}(x_2) \id_{A}(y_3) \Big(\exx_{b_1} \id_{A}(b_1) \id_{A}(b_3 + a - b_1) \id_{A}(b_1 + x_2 - a) \id_{A}(y_3 + b_1 - b_3)\Big)\Big)^2\\
\leq &\exx_{a, b_3, x_2, y_3} \id_{A}(a) \id_{A}(b_3)   \id_{A}(x_2) \id_{A}(y_3) \Big(\exx_{b_1} \id_{A}(b_1) \id_{A}(b_3 + a - b_1) \id_{A}(b_1 + x_2 - a) \id_{A}(y_3 + b_1 - b_3)\Big)^2\\
= & \exx_{a, b_3, x_2, y_3} \id_{A}(a) \id_{A}(b_3)   \id_{A}(x_2) \id_{A}(y_3) \exx_{b_1, z_1} \id_{A}(b_1) \id_{A}(b_3 + a - b_1) \id_{A}(b_1 + x_2 - a) \id_{A}(y_3 + b_1 - b_3)\\
&\hspace{6cm}\id_{A}(z_1) \id_{A}(b_3 + a - z_1) \id_{A}(z_1 + x_2 - a) \id_{A}(y_3 + z_1 - b_3)\\
\leq &  \exx_{a, b_1, b_3, x_2, y_3, z_1} \id_{A}(a) \id_{A}(b_1) \id_{A}(b_3 + a - b_1) \id_{A}(b_3)   \id_{A}(x_2)  \id_{A}(b_1 + x_2 - a) \\
&\hspace{8cm}\id_{A}(y_3)\id_{A}(y_3 + b_1 - b_3) \id_{A}(z_1) \id_{A}(b_3 + a - z_1) 
\end{align*}
where we used Cauchy-Schwarz inequality in the first inequality above and omitted terms $ \id_{A}(z_1 + x_2 - a)$ and $\id_{A}(y_3 + z_1 - b_3)$ in the second. We make the final change of variables and use $b_2 = b_3 + a - b_1$ instead of $a$. Thus 
\begin{align*}c^{32} \leq &  \exx_{b_1,b_2, b_3, x_2, y_3, z_1} \id_{A}(b_1) \id_{A}(b_2) \id_{A}(b_3) \id_{A}(b_1 + b_2 - b_3)   \id_{A}(x_2)  \id_{A}(x_2 - b_2 + b_3) \\
&\hspace{6cm}\id_{A}(y_3)\id_{A}(y_3 + b_1 - b_3) \id_{A}(z_1) \id_{A}(b_1 + b_2 - z_1).\qedhere\end{align*}
\end{proof}

We also need the fact that linear maps of very high rank between spaces of same dimension can be efficiently modified to yield an isomorphism. The proof is elementary.

\begin{lemma}\label{nearisomlemma}Let $A$ and $B$ be two vector spaces of dimension $d$ and let $\phi \colon A \to B$ be a linear map of rank $d - \ell$. Then there exists an isomorphism $\psi \colon A \to B$ such that $\on{rank}(\phi - \psi) \leq \ell$.\end{lemma}

\begin{proof}By the rank-nullity theorem, the kernel $K$ of $\phi$ has dimension $\ell$. Let $U$ be an arbitrary subspace such that $A = K \oplus U$, and let $\pi \colon A \to K$ be the resulting projection onto $K$. Then $I = \phi(U)$ is the image of the map $\phi$. Let $B = I \oplus V$ for some subspace $V$. Then $\dim V = d - \dim I = \ell = \dim K$, so there exists a linear isomorphism $\theta \colon K \to V$. Define $\psi = \phi + \theta \circ \pi$. We claim that $\psi$ is isomorphism. It suffices to prove that $\psi$ is injective. To that end, let $x \in A$ be such that $\psi(x) = 0$. Then $\phi(x) + \theta(\pi(x)) = 0$. However, $\phi(x) \in I$ and $\theta(\pi(x)) \in V$, so as $I \cap V = 0$, it follows that $\phi(x) = 0$ and $\theta(\pi(x)) = 0$. Hence, $x \in K$ and $\pi(x) = 0$, so $x = 0$, as desired.\end{proof}

Let $G$ and $H$ be finite-dimensional vector spaces over $\mathbb{F}_p$ and let $A \subseteq G$ be a subset $G \times H$. Let $\phi \colon A \to K$ be a map. We say that $\phi$ \emph{respects} all horizontal additive quadruples if for all $y \in H$ and $x_1, x_2, x_3, x_4 \in G$ such that $x_1 + x_2 = x_3 + x_4$ and $(x_i, y) \in A$ for $i \in [4]$ we have $\phi(x_1, y) + \phi(x_2, y) = \phi(x_3, y) + \phi(x_4, y)$. Analogously, we say that $\phi$ \emph{respects} all vertical additive quadruples if the same condition holds with the roles of $G$ and $H$ reversed. If $\phi$ respects all horizontal and vertical additive quadruples we say that $\phi$ is a \emph{Freiman bihomomorphism}. It turns out that global biaffine maps are essentially the only sources of Freiman bihomomorphisms. This theorem was first proved in~\cite{U4paper} by Gowers and the author and the bounds were later improved by Lovett (personal communication), and by Kim, Li and Tidor~\cite{KimLiTidor} by optimizing some steps in the original proof.

\begin{theorem}\label{FreimanBihomomorphism} Let $A \subseteq G$ be a subset $G \times H$ of density $c$. Suppose that $\phi \colon G \times H \to K$ is a Freiman bihomomorphism. Then there eixsts a biaffine map $\Phi \colon G \times H \to K$ such that $\Phi = \phi$ holds for at least $c'|G||H|$ points in $A$, where $c' = \exp\Big(-\exp(O(\log^{O(1)} c^{-1}))\Big)$.\end{theorem}

In this paper, we shall use the following corollary. 

\begin{corollary}\label{friemanforlinearhom}Let $G, H, K$ be finite-dimensional vector spaces over $\mathbb{F}_p$ and let $A \subseteq G$ be a subset of density $c$ and let for each $a \in A$ be given a linear map $\phi_a \colon H \to K$. Suppose that for at least $c |G|^3$ additive quadruples $(a_1, a_2, a_3, a_4) \in A^4$ we have 
\begin{equation}\label{rankapproxbound}\on{rank} \Big(\phi_{a_1} + \phi_{a_2} - \phi_{a_3} - \phi_{a_4}\Big) \leq r.\end{equation}
Then there exists a map $\Phi \colon G \times H \to K$, which is affine in the first coordinate and linear in the second, such that $\on{rank}(\Phi(a, \cdot) - \phi_a) \leq \exp\Big((\log c^{-1} + r)^{O(1)}\Big)$ for at least $\exp\Big(-\exp\Big((\log c^{-1} + r)^{O(1)}\Big)\Big)|G|$ of $a \in A$, where $\Phi(a, \cdot)$ stands for the map from $H$ to $K$ given by $b \mapsto \Phi(a,b)$.\end{corollary}

Let us remark that a similar result was proved by Kazhdan and Ziegler in~\cite{ApproxCohomology} by using the inverse theorem for $\mathsf{U}^4$ norm. Unlike the result above, Kazhdan and Ziegler assume that bound~\eqref{rankapproxbound} holds for all choices of additive quadruples in $G$. It is very likely that their proof can be modified to work in this setting as well, but it turns out that property~\eqref{rankapproxbound} almost implies that $\phi$ is a Freiman bihomomorphism, so applying Theorem~\ref{FreimanBihomomorphism} gives a shorter deduction. This is no surprise as Theorem~\ref{FreimanBihomomorphism} was initially proved in order to give a quantitative inverse theorem for $\mathsf{U}^4$ norm.

\begin{proof}Consider map $\psi \colon A \times H \to K$ given by $\psi(a,x) = \phi_a(x)$. Since each $\phi_a$ is linear, the map $\psi$ respects all vertical additive quadruples. When it comes to horizontal additive quadruples, given any additive quadruple $(a_1, a_2, a_3, a_4) \in A^4$ such that 
\[\on{rank} \Big(\phi_{a_1} + \phi_{a_2} - \phi_{a_3} - \phi_{a_4}\Big) \leq r,\]
it follows that the codimension of kernel of $\phi_{a_1} + \phi_{a_2} - \phi_{a_3} - \phi_{a_4}$ is at most $r$, so we have at least $p^{-r} |H|$ of elements $y \in H$ for which $\psi$ respects horizontal additive quadruple $\Big((a_1, y), (a_2, y), (a_3, y), (a_4, y)\Big)$. By averaging and using Theorem~\ref{approxHommInv} we may pass to a subset $B \subseteq A \times H$, of size $|B| \geq \exp\Big(-(\log c^{-1} + r)^{O(1)}\Big) |G||H|$, on which $\psi$ is a Freiman bihomomorphism. By Theorem~\ref{FreimanBihomomorphism} there exists a further subset $B' \subseteq B$ and biaffine map $\Phi \colon G \times H\to K$ such that $\Phi = \psi$ on $B'$, with $|B'| \geq c_1 |G||H|$ where $c_1 \geq  \exp\Big(-\exp\Big((\log c^{-1} + r)^{O(1)}\Big)\Big)$. Define $\tilde{\Phi}(x,y) = \Phi(x,y) - \Phi(x,0)$, which is linear in $y$. Take now any $a \in A$ for which $|B'_{a \bcdot}| \geq \frac{c_1}{2} |H|$. For each $y \in B'_{a \bcdot}$ we have $\Phi(a,y) = \phi_a(y)$. By Cauchy-Schwarz inequality, we have $\tilde{\Phi}(a, y-y') = \Phi(a,y) - \Phi(a,y') = \phi_a(y) - \phi_a(y') = \phi_a(y-y')$  for at least $\frac{c_1^{2}}{4} |H|^2$ choices of $(y,y')$. Thus, $\on{rank}\Big(\tilde{\Phi}(a, \cdot) - \phi_a\Big) \leq \log_p (4c_1^{-2})$, for each such $a$, of which there are at least $\frac{c_1}{2} |G|$.\end{proof}

\section{From approximate quadratic varieties to approximate quasirandom linear systems of subspaces} 

In this section, we begin the study of approximate quadratic varieties. Our goal is to obtain an approximate quasirandom linear systems of subspaces that is closely related to the given quadratic variety. The main result of this section is the following theorem.

\begin{theorem}\label{mainstep1thm}There exists an absolute constant $D \geq 1$ such that the following holds. Let $V \subseteq G$ be a $(c_0, \delta, \varepsilon)$-approximate quadratic variety and suppose that $\varepsilon \leq \exp\Big(- \log^D(2c_0^{-1})\Big) \delta^{288}$. Then there exist a quantity $c_1$, a set $A \subseteq G$ and a collection of subspaces $W_a \leq G$ indexed by elements $a \in A$ such that
\begin{itemize}
\item[\textbf{(i)}] $\exp(-\log^D (2c_0^{-1})) \leq c_1 \leq 1$,
\item[\textbf{(ii)}] $|A| \geq c_1 |G|$,
\item[\textbf{(iii)}] for each $a \in A$ and $b \in W_a$ we have ${\conv}^{(8)} \id_{V \cap V -a}(b) \geq c_1 \delta^{15}$,
\item[\textbf{(iv)}] $c_1 \leq \frac{|W_a|}{\delta|G|} \leq c_1^{-1}$ holds for all $a \in A$,
\item[\textbf{(v)}] for $r \in [9]$, for all but at most $D \varepsilon \delta^{-32r} |G|^r$ choices of $(a_1, \dots, a_r) \in A^r$ we have
\[\frac{|W_{a_1} \cap W_{a_2} \cap \dots \cap W_{a_r}|}{\delta^r|G|} \leq c_1^{-1},\] 
 and,
\item[\textbf{(vi)}] for at least $c_1 |A|^6$ 6-tuples $(b_1, b_2, b_3, x_2, y_3, z_1) \in A^6$ we have that $b_1 + b_2 - b_3,  x_2 - b_2 + b_3, y_3  + b_1 - b_3, b_1 + b_2 - z_1\in A$ and 
\[c_1 \delta^6 |G| \leq |W_{b_1} \cap W_{b_2} \cap W_{b_3} \cap W_{b_1 + b_2 - b_3} \cap W_{x_2} \cap W_{x_2 - b_2 + b_3} \cap W_{y_3} \cap W_{y_3  + b_1 - b_3} \cap W_{z_1} \cap W_{b_1 + b_2 - z_1}|.\]
\end{itemize}
\end{theorem}

Throughout the section $V, c_0, \delta$ and $\varepsilon$ will be fixed and $V$ will be a $(c_0, \delta, \varepsilon)$-approximate quadratic variety.

\subsection{Elementary estimates}

We frequently need control expressions such as $\exx_{a,b,x} \mder_{a,b}\id_V(x)$ using the fact that the $\mathsf{U}^2$ norm of the difference $\id_V - \delta$ is small. In order to be efficient we record the following lemma which says that if we can find two variables in such an expression that both appear in a single occurrence of $\id_V$, then we may replace that term with $\delta$.

\begin{lemma}\label{u2controlBasic}Let $f_1, \dots, f_r \colon G \to \mathbb{D}$ be functions and let $\lambda_{i,j} \in \mathbb{F}_p$ be coefficients for $i \in [0,r]$, $j \in [s]$. Suppose that for some distinct indices $a , b \in [s]$ we have $\lambda_{i,a}\lambda_{i,b}$ non-zero if and only if $i = 0$. Then
\[\Big|\Big(\exx_{x_1, \dots, x_s} \id_V\Big(\sum_{j \in [s]} \lambda_{0, j} x_j\Big)\prod_{i \in [r]} f_i\Big(\sum_{j \in [s]} \lambda_{i, j} x_j\Big)\Big) - \delta\Big(\exx_{x_1, \dots, x_s} \prod_{i \in [r]} f_i\Big(\sum_{j \in [s]} \lambda_{i, j} x_j\Big)\Big)\Big| \leq \varepsilon.\]\end{lemma}

\begin{proof}Write $g(x) = \id_V(x) - \delta$. Let $I$ be the set of all indices $i \in [r]$ such that $\lambda_{i,a}\not=0$. Standard applications of Cauchy-Schwarz inequality give
\begin{align*} \Big|\exx_{x_1, \dots, x_s} g\Big(\sum_{j \in [s]} \lambda_{0, j} x_j\Big)\prod_{i \in [r]} f_i\Big(\sum_{j \in [s]} \lambda_{i, j} x_j\Big)\Big|^4 \leq &\Big|\exx_{x_{[s] \setminus \{a\}}} \Big|\exx_{x_a} g\Big(\sum_{j \in [s]} \lambda_{0, j} x_j\Big)\prod_{i \in I} f_i\Big(\sum_{j \in [s]} \lambda_{i, j} x_j\Big)\Big|^2\Big|^2\\
=  &\Big|\exx_{x_{[s]}, y} \mder_{\lambda_{0, a}y}g\Big(\sum_{j \in [s]} \lambda_{0, j} x_j\Big) \prod_{i \in I} \mder_{\lambda_{i, a}y}f_i\Big(\sum_{j \in [s]} \lambda_{i, j} x_j\Big)\Big|^2\\
\leq & \exx_{x_{[s]\setminus \{b\}}, y} \Big|\exx_{x_b} \mder_{\lambda_{0, a}y}g\Big(\sum_{j \in [s]} \lambda_{0, j} x_j\Big)\Big|^2\\
=&\exx_{x_{[s]}, y,z} \mder_{\lambda_{0, a}y}\mder_{\lambda_{0, b}z}g\Big(\sum_{j \in [s]} \lambda_{0, j} x_j\Big) = \|g\|_{\mathsf{U}^2}^4.\qedhere\end{align*}\end{proof}

Note that we may still use this lemma if a term $\id_V$ has a variable $x_j$ that occurs only at that place as we can introduce another dummy variable $y$ and make a change of variables $x_j \mapsto x_j + y$.\\

We now use Lemma~\ref{u2controlBasic} to prove some easy properties of the set $V$. The most basic consequence of $\mathsf{U}^2$ uniformity is that translates of $V$ have roughly the same intersection size.

\begin{claim}\label{varietyTranslatesprops1}Let $r$ be a positive integer. We have
\[\exx_{a_1, \dots, a_r} \Big||V \cap (V - a_1) \cap (V - a_2) \cap \dots \cap (V - a_r)| - \delta^{r+1} |G|\Big|^2 \leq 4(r + 1) \varepsilon.\]
In particular, for all but at most $4(r+1)\varepsilon^{1/2} |G|^r$ $r$-tuples $(a_1, \dots, a_r) \in G$ we have $\Big||V \cap (V - a_1) \cap (V - a_2) \cap \dots \cap (V - a_r)| - \delta^{r+1} |G|\Big| \leq \varepsilon^{1/4} |G|$.\end{claim}

\begin{proof}Consider the expression
\begin{align}\exx_{a_1, \dots, a_r} &\Big||V \cap (V - a_1) \cap (V - a_2) \cap \dots \cap (V - a_r)| - \delta^{r+1} |G|\Big|^2\nonumber\\
 = &|G|^2 \exx_{a_1, \dots, a_r} \Big|\exx_x \id_V(x)\id_V(x + a_1) \dots \id_V(x + a_r) - \delta^{r+1}\Big|^2 \nonumber\\
= &|G|^2 \Big(\exx_{\ssk{a_1, \dots, a_r\\x,y}}\id_V(x)\id_V(x + a_1) \dots \id_V(x + a_r) \id_V(y)\id_V(y + a_1) \dots \id_V(y + a_r)\nonumber\\
&\hspace{7cm} - 2\delta^{r+1}\id_V(x)\id_V(x + a_1) \dots \id_V(x + a_r) + \delta^{2r + 2}\Big).\label{varietyTranslatesprops1eq1}\end{align}
Applying Lemma~\ref{u2controlBasic} $2r + 2$ times to terms $\id_V(x+a_i)$ using $x$ and $a_i$ for $i \in [r]$, $\id_V(y+a_i)$ using $y$ and $a_i$ for $i \in [r]$, $\id_V(x)$ using $x$ and $\id_V(y)$ using $y$ in that order,\footnote{Note that the first $2r$ applications of Lemma~\ref{u2controlBasic} have turned terms $\id_V(x+a_i)$ and $\id_V(y+a_i)$ into copies of $\delta$ and after those $2r$ steps $x$ and $y$ only appear in terms $\id_V(x)$ and $\id_V(y)$.} we see that 
\[\Big|\exx_{\ssk{a_1, \dots, a_r\\x,y}}\id_V(x)\id_V(x + a_1) \dots \id_V(x + a_r) \id_V(y)\id_V(y + a_1) \dots \id_V(y + a_r) - \delta^{2r + 2}\Big| \leq (2r + 2)\varepsilon.\]
Applying Lemma~\ref{u2controlBasic} to the simpler expression $\exx_{a_1, \dots, a_r, x}\id_V(x)\id_V(x + a_1) \dots \id_V(x + a_r)$ in a similar fashion, we conclude that expression~\eqref{varietyTranslatesprops1eq1} is at most $(4r + 4)\varepsilon$. The second part of the conclusion follows after averaging.\end{proof}

Next, we give an upper bound on intersections $|V \cap (V-a)\cap (V-b) \cap (V-a-b)|$. We may no longer give a lower bound for the following reason. Recall that we expect $V$ to be closely related to a quadratic variety $Q = \{x \in G \colon \beta(x,x) = 0\}$, defined by some high rank symmetric bilinear map $\beta \colon G \times G \to \mathbb{F}_p^r$. If $x \in V \cap (V-a)\cap (V-b) \cap (V-a-b)$ then $\beta(x,x) = \beta(x + a, x + a) = \beta(x + b, x + b) = \beta(x + a + b, x + a + b) = \lambda$, so $\beta(a,b) = 0$. Notice also that $Q \cap Q - a \cap Q - b\subseteq Q - a - b$. Thus, we expect that the size of intersection $|V \cap (V-a)\cap (V-b) \cap (V-a-b)|$ is about $\delta^3 |G|$ when $\beta(a,b) = 0$, and significantly smaller otherwise. 

\begin{claim}\label{popDiffProp1a}For all but at most $12\varepsilon^{1/2}|G|^2$ pairs of elements $(a, b) \in G^2$ we have $\Big||V \cap (V-a)\cap (V-b)| -\delta^3|G|\Big| \leq \varepsilon^{1/4}|G|$. In particular $|V \cap (V-a)\cap (V-b) \cap (V-a-b)| \leq (\delta^3 + \varepsilon^{1/4})|G|$ holds for all but at most $12\varepsilon^{1/2}|G|^2$ pairs of elements $(a, b) \in G^2$.\end{claim}

\begin{proof}Consider
\begin{align*}\exx_{a,b}\Big|\frac{1}{|G|}|V \cap (V-a)\cap (V-b)| -\delta^3\Big|^2 = &\exx_{\ssk{a,b\\x,y}} \id_V(x)\id_V(x + a)\id_V(x+b)\id_V(y)\id_V(y + a)\id_V(y+b)\\
&\hspace{2cm} -2\delta^3 \id_V(x)\id_V(x + a)\id_V(x+b) + \delta^6.\end{align*}
Applying Lemma~\ref{u2controlBasic} 6 times for terms $\id_V(x + a)$ using $x$ and $a$, $\id_V(x + b)$ using $x$ and $b$, $\id_V(x)$ using $x$, $\id_V(y + a)$ using $y$ and $a$, $\id_V(y + b)$ using $y$ and $b$ and $\id_V(y)$ using $y$ in that order, we see that 
\[\Big|\exx_{\ssk{a,b\\x,y}} \id_V(x)\id_V(x + a)\id_V(x+b)\id_V(y)\id_V(y + a)\id_V(y+b) - \delta^6\Big| \leq 6\varepsilon.\]
Similar argument for $\ex_{x,a,b} \id_V(x)\id_V(x + a)\id_V(x+b)$ shows that the expression above is at most $12\varepsilon$. The first part of the claim now follows.\\
\indent For the second part of the claim simply note that $|V \cap (V-a)\cap (V-b) \cap (V-a-b)| \leq |V \cap (V-a)\cap (V-b)|$.\end{proof}

We need one more claim of this form. Notice that we have 8 $\id_V$ terms in the expression below, yet the upper bound is about $\delta^5$, which is due to algebraic dependencies between arguments of $\id_V$.

\begin{claim}\label{popDiffPropcycles}
For all but at most $20\sqrt{\varepsilon} |G|^4$ of $(a,d_1,d_2, d_3) \in G^4$ we have 
\begin{equation}\exx_{x}\id_V(x)\id_V(x+a)\id_V(x+d_1)\id_V(x+d_1+a)\id_V(x+d_1 -d_2)\id_V(x+d_1-d_2 + a)\id_V(x+d_3)\id_V(x+d_3+a) \leq \delta^5 + \sqrt[4]{\varepsilon}.\label{4cycleexpbound}\end{equation}
\end{claim}

Having terms $\id_V(x+d_2)\id_V(x+d_2 + a)$ instead of $\id_V(x+d_1 -d_2)\id_V(x+d_1-d_2 + a)$ would be more natural, but we opted for this formulation as we get this expression when the claim is applied.

\begin{proof}Notice that the expression in question is at most 
\[\exx_{x}\id_V(x)\id_V(x+a)\id_V(x+d_1)\id_V(x+d_1 -d_2)\id_V(x+d_3).\]
We now show that
\[\exx_{a,d_1, d_2, d_3}\Big|\exx_{x}\id_V(x)\id_V(x+a)\id_V(x+d_1)\id_V(x+d_1 -d_2)\id_V(x+d_3) - \delta^5\Big|^2 \leq 20\varepsilon.\]
By (by now) usual argument based on Lemma~\ref{u2controlBasic} we get that
\begin{align*}&\exx_{a,d_1, d_2, d_3}\Big(\exx_{x}\id_V(x)\id_V(x+a)\id_V(x+d_1)\id_V(x+d_1 -d_2)\id_V(x+d_3)\Big)^2\\
&\hspace{2cm}=\exx_{\ssk{a,d_1, d_2, d_3\\x,y}}\id_V(x)\id_V(x+a)\id_V(x+d_1)\id_V(x+d_1 -d_2)\id_V(x+d_3)\\
&\hspace{6cm}\id_V(y)\id_V(y+a)\id_V(y+d_1)\id_V(y+d_1 -d_2)\id_V(y+d_3)\end{align*}
differs from $\delta^{10}$ by at most $10\varepsilon$. Namely, we apply Lemma~\ref{u2controlBasic} for term $\id_V(x + d_3)$ using $x$ and $d_3$, for term $\id_V(x + d_1  - d_2)$ using $x$ and $d_2$, for term $\id_V(x + d_1)$ using $x$ and $d_1$, for term $\id_V(x + a)$ using $x$ and $a$, for term $\id_V(x)$ using $x$, in that order, and similarly for terms involving $y$. Similarly,
\[\exx_{a,d_1, d_2, d_3, x}\id_V(x)\id_V(x+a)\id_V(x+d_1)\id_V(x+d_1 -d_2)\id_V(x+d_3)\]
differs from $\delta^5$ by at most $5\varepsilon$. The claim follows after averaging.\end{proof}

\subsection{Additive structure}

In this subsection show that the set of large values iterated convolutions of $\id_{V \cap V -a}$ has significant additive structure for many $a \in G$. We begin with an important technical definition.\\
\indent Let $\eta > 0$ be a parameter. We say that an element $a \in G$ is \emph{$\eta$-regular} if it satisfies 
\begin{itemize}
\item[\textbf{(i)}] $\frac{1}{2}\delta^2 |G|\leq |V \cap V - a| \leq 2\delta^2 |G|$,
\item[\textbf{(ii)}] for all but at most $\eta |G|$ of $b \in G$ we have
\[|V \cap (V-a) \cap (V-b) \cap (V-a-b)| \leq 2\delta^3 |G|,\]
and
\item[\textbf{(iii)}] for all but at most $\eta|G|^3$ of triples $(d_1, d_2, d_3) \in G$ we have
\[\exx_{x}\id_V(x)\id_V(x+a)\id_V(x+d_1)\id_V(x+d_1+a)\id_V(x+d_1 -d_2)\id_V(x+d_1-d_2 + a)\id_V(x+d_3)\id_V(x+d_3+a) \leq 2\delta^5.\]
\end{itemize}

Let us show that regular elements are abundant.

\begin{proposition}\label{regularitycount}Suppose that $\varepsilon \leq 2^{-4}\delta^{20}$. Then all but at most $40\sqrt[4]{\varepsilon}|G|$ elements are $\varepsilon^{1/4}$-regular.\end{proposition}

\begin{proof}By Claim~\ref{varietyTranslatesprops1}, provided that $\varepsilon \leq 2^{-4}\delta^8$, we have
\[\Big||V \cap (V - a)| - \delta^2 |G|\Big| \leq \frac{1}{2}\delta^2 |G|\]
for all but at most $8\varepsilon^{1/2} |G|$ elements $a \in G$. Next, provided that $\varepsilon \leq \delta^{12}$, it follows from Claim~\ref{popDiffProp1a} that 
\begin{equation}\label{vabineqreg}|V \cap (V-a)\cap (V-b) \cap (V-a-b)| \leq 2\delta^3|G|\end{equation}
holds for all but at most $12\varepsilon^{1/2}|G|^2$ pairs of elements $(a, b) \in G^2$. Hence, the number of $a$ such that~\eqref{vabineqreg} fails for at least $\varepsilon^{1/4}|G|$ elements $b \in G$ is at most $12\varepsilon^{1/4}|G|$. Finally, using Claim~\ref{popDiffPropcycles} and provided $\varepsilon \leq \delta^{20}$, we conclude that the number of elements $a$ such that~\eqref{4cycleexpbound} fails for at least $\sqrt[4]{\varepsilon}|G|^3$ triples $(d_1, d_2, d_3) \in G^3$ is at most $20\varepsilon^{1/4}|G|$. The proposition now follows.\end{proof}

The following proposition finds desired additive structure if we are given a regular element.

\begin{proposition}\label{popDiffAddQuads}Let $\eta > 0$ and let $a \in G$ be an $\eta$-regular element. Let $P \subseteq G$ be a set and let $c' >0$ be such that 
\begin{equation}\label{addquadsl2assumption}\exx_{b} \id_P(b) \Big(\id_{V \cap V -a} \conv \id_{V \cap V -a}(b)\Big)^2 \geq c' \delta^7.\end{equation}
Suppose that $\eta \leq 2^{-41}{c'}^8\delta^8$. Then there exists a subset $\tilde{P} \subseteq P$ which contains at least $2^{-42}{c'}^8\delta^3 |G|^3$ additive quadruples, has size $\frac{c'\delta}{16} |G| \leq |\tilde{P}| \leq 2^9{c'}^{-2}\delta|G|$ and every element $b \in \tilde{P}$ satisfies
\[\id_{V \cap V -a} \conv \id_{V \cap V -a}(b) \geq \frac{1}{8}c' \delta^3.\] 
\end{proposition}

\begin{proof}By assumption \textbf{(ii)} of $\eta$-regularity, for all but at most $\eta |G|$ of $b \in G$ we have
\begin{equation}\label{popdiffscondquadsproof}\id_{V \cap V -a} \conv \id_{V \cap V -a}(b) = \ex_x \id_{V \cap (V-a)}(x+b) \id_{V \cap (V-a)}(x) = \frac{1}{|G|}|V \cap (V-a) \cap (V-b) \cap (V-a-b)| \leq 2\delta^3.\end{equation}
By removing $b$ for elements $b \in G$ when this inequality fails, we may find a subset $P' \subseteq P$ of size $|P'| \geq |P| - \eta |G|$ such that~\eqref{popdiffscondquadsproof} holds for all $b \in P'$. From~\eqref{addquadsl2assumption} we deduce that 
\[\exx_{b} \id_{P'}(b) \Big(\id_{V \cap V -a} \conv \id_{V \cap V -a}(b)\Big)^2 \geq c' \delta^7 -\eta.\]
Let $\tilde{P} \subseteq P'$ be the set of all elements $b \in P'$ such that $\id_{V \cap V -a} \conv \id_{V \cap V -a}(b) \geq \frac{1}{8}c' \delta^3$. Then
\begin{align*}\exx_{b} \id_{P' \setminus \tilde{P}}(b) \Big(\id_{V \cap V -a} \conv \id_{V \cap V -a}(b)\Big)^2 \leq &\frac{1}{8}c' \delta^3 \exx_{b} \id_{P' \setminus \tilde{P}}(b) \id_{V \cap V -a} \conv \id_{V \cap V -a}(b)\\
\leq &\frac{1}{8}c' \delta^3 \exx_{b}\id_{V \cap V -a} \conv \id_{V \cap V -a}(b)\\
\leq &\frac{1}{8}c' \delta^3 \Big(\frac{|V \cap V -a|}{|G|}\Big)^2 \leq \frac{1}{2}c' \delta^7,\end{align*}
where we used assumption \textbf{(i)} of $\eta$-regularity in the last step.\\
\indent Hence 
\[\exx_{b} \id_{\tilde{P}}(b) \Big(\id_{V \cap V -a} \conv \id_{V \cap V -a}(b)\Big)^2 \geq \frac{1}{2}c' \delta^7 -\eta,\]
which, since  $\tilde{P} \subseteq P'$, also gives $|\tilde{P}| \geq \Big(\frac{c'\delta}{8} - \eta \delta^{-6}\Big) |G| \geq \frac{c'\delta}{16} |G|$, since $\eta \leq \frac{c' \delta^7}{16}$. The same argument as the one above and the fact that $\tilde{P} \subseteq P'$ imply
\[\frac{|\tilde{P}|}{|G|}\Big(\frac{1}{8}c' \delta^3\Big)^2 \leq \exx_{b} \id_{\tilde{P}}(b) \Big(\id_{V \cap V -a} \conv \id_{V \cap V -a}(b)\Big)^2 \leq 2\delta^3  \exx_{b}\id_{V \cap V -a} \conv \id_{V \cap V -a}(b) \leq 8\delta^7\]
so $|\tilde{P}| \leq 2^9{c'}^{-2}\delta|G|$.\\

We say that $b$ is a \emph{popular difference} if $b \in \tilde{P}$. Consider the bipartite graph $\Gamma$ with both vertex classes $\mathcal{C}_1$ and $\mathcal{C}_2$ being a copy of $V \cap (V-a)$ and edges between $x \in \mathcal{C}_1$ and $y \in \mathcal{C}_2$ if $x-y$ is a popular difference. By a \emph{cycle of length 4} we mean an ordered quadruple of vertices $v_1, v_2, v_3, v_4$ such that $v_1v_2, \dots, v_4 v_1$ are edges.\footnote{It would be more precise to call these quadruples homomorphisms from $C_4$ to the given graph.} This is a graph of density 
\[\frac{\sum_{b \in \tilde{P}} \id_{V \cap V -a} \conv \id_{V \cap V -a}(b)|G|}{|V \cap (V-a)|^2} \geq \frac{\frac{c'\delta}{16} |G| \cdot \frac{1}{8}c' \delta^3|G|}{4\delta^4 |G|^2} \geq 2^{-9} {c'}^2\]
so it has $2^{-36}{c'}^8|V \cap (V-a)|^4$ cycles of length 4. Notice that each such cycle gives an additive quadruple in $\tilde{P}$: namely if $x_1, x_2$ are vertices in class $\mathcal{C}_1$ and $y_1$ and $y_2$ are vertices in class $\mathcal{C}_2$, then $p_{i,j} = x_i - y_j$ is a popular difference and $p_{1,1} + p_{2,2} = p_{1,2} + p_{2,1}$. On the other hand, if we fix an additive quadruple $(d_1, d_2, d_3, d_2 + d_3 - d_1)$ of popular differences it gives 
\[\sum_{y} \id_{V \cap (V-a)}(y) \id_{V \cap (V-a)}(y + d_1) \id_{V \cap (V-a)}(y + d_1 -d_2) \id_{V \cap (V-a)}(y + d_3)\]
cycles of length 4. This quantity is bounded above by $2\delta^5 |G|$ for all but at most $\eta|G|^3$ of triples $(d_1, d_2, d_3) \in G$ by assumption \textbf{(iii)} of $\eta$-regularity, while it is always at most $|G|$. Let $Q$ be the number of additive quadruples in $\tilde{P}$. By double-counting we get
\[2^{-40}{c'}^8\delta^8|G|^4 \leq 2^{-36}{c'}^8|V \cap (V-a)|^4 \leq Q \cdot 2\delta^5 |G| + \eta |G|^3 \cdot |G|.\]
Since $\eta  \leq 2^{-41}{c'}^8\delta^8$, we conclude that $Q \geq 2^{-42}{c'}^8\delta^3 |G|^3$, as desired.\end{proof}

Let $\xi > 0$ be a parameter to be chosen later (it will be a function of $c$) and let $A^\xi\subseteq G$ be the set of all elements $a \in G$ which are $\varepsilon^{1/4}$-regular and satisfy 
\[\exx_{b} \Big(\id_{V \cap V -a} \conv \id_{V \cap V -a}(b)\Big)^2 \geq \xi \delta^7.\]

We first show that $A^\xi$ is large.

\begin{claim}\label{Alargeness} Suppose that $\xi \leq c_0/2$ and $\varepsilon \leq 2^{-32} c_0^4 \delta^{28}$. We have $|A^\xi| \geq \frac{c_0}{40} |G|$ and 
\begin{equation}\label{l2smallcontribreg}\exx_{a,b} \id_{G \setminus A^\xi}(a)\Big(\id_{V \cap V -a} \conv \id_{V \cap V -a}(b)\Big)^2 \leq \xi \delta^7 + 40\sqrt[4]{\varepsilon}.\end{equation}
\end{claim}

\begin{proof}Write $A = A^\xi$ to simplify the notation. The fact that $V$ is a $(c_0, \delta, \varepsilon)$-approximate quadratic variety implies that
\begin{equation}\exx_{a,b} \Big(\id_{V \cap V -a} \conv \id_{V \cap V -a}(b)\Big)^2 \geq c_0\delta^7.\label{approxvarassumptionl2bounds}\end{equation}
Let $I$ be the set of elements $a \in G$ which are not $\varepsilon^{1/4}$-regular. Proposition~\ref{regularitycount} implies $|I| \leq 40\sqrt[4]{\varepsilon}|G|$. Their contribution to~\eqref{approxvarassumptionl2bounds} is
\[\exx_{a,b} \id_{I}(a)\Big(\id_{V \cap V -a} \conv \id_{V \cap V -a}(b)\Big)^2 \leq \exx_a \id_I(a) \leq 40\sqrt[4]{\varepsilon}.\]
Furthermore, the contribution from $G \setminus (I \cup A)$ is
\[\exx_{a,b} \id_{I^c \cap A^c}(a)\Big(\id_{V \cap V -a} \conv \id_{V \cap V -a}(b)\Big)^2 = \exx_a \id_{I^c \cap A^c}(a) \bigg(\exx_{b}\Big(\id_{V \cap V -a} \conv \id_{V \cap V -a}(b)\Big)^2 \bigg)\leq \xi \delta^7,\]
from which we obtain~\eqref{l2smallcontribreg}.\\
\indent To prove the bounds on $|A|$, notice first that if $a$ is $\varepsilon^{1/4}$-regular then
\begin{align*}\exx_{b}\Big(\id_{V \cap V -a} \conv \id_{V \cap V -a}(b)\Big)^2 = &\exx_b \frac{|V \cap (V - a) \cap (V - b) \cap (V - a - b)|^2}{|G|^2} \\
\leq &\sqrt[4]{\varepsilon} + 2\delta^3\exx_b \frac{|V \cap (V - a) \cap (V - b) \cap (V - a - b)|}{|G|} \\
\leq &\sqrt[4]{\varepsilon} + 2\delta^3 \frac{|V \cap V - a|^2}{|G|^2} \leq \sqrt[4]{\varepsilon} + 8\delta^7 \leq 10\delta^7,\end{align*}
provided $\varepsilon \leq \delta^{28}$, where we used assumption \textbf{(ii)} of $\sqrt[4]{\varepsilon}$-regularity in the second line and assumption \textbf{(i)} of $\sqrt[4]{\varepsilon}$-regularity in the third line. Combining~\eqref{l2smallcontribreg} and ~\eqref{approxvarassumptionl2bounds} we get
\[10\delta^7 \frac{|A|}{|G|} \geq \exx_a \id_A(a) \exx_{b}\Big(\id_{V \cap V -a} \conv \id_{V \cap V -a}(b)\Big)^2 \geq (c_0 - \xi) \delta^7 - 40\sqrt[4]{\varepsilon} \geq \frac{c_0\delta^7}{4},\]
using the assumptions on $\xi$ and $\varepsilon$ in the last inequality.
\end{proof}

As a corollary, we get lower bounds on the number of occurrences of another configuration in $V$. 

\begin{corollary}\label{commonadditivequads10Cor} Provided $\varepsilon \leq 2^{-200} c_0^8 \delta^{28}$, we have 
\begin{align*}&\exx_{b_1, b_2, b_3, x_2, y_3, z_1}\exx_w \id_{V \cap V - (b_1 + b_2 - b_3)} \conv\id_{V \cap V -(b_1 + b_2 - b_3)}(w) \,\,\id_{V \cap V - b_1} \conv \id_{V \cap V - b_1}(w)\,\, \id_{V \cap V - b_2} \conv \id_{V \cap V - b_2}(w)\\
&\hspace{2cm} \id_{V \cap V - b_3} \conv \id_{V \cap V - b_3}(w)\,\, \id_{V \cap V - x_2} \conv \id_{V \cap V - x_2}(w)\,\, \id_{V \cap V - (x_2 - b_2 + b_3)} \conv \id_{V \cap V - (x_2 - b_2 + b_3)}(w) \\
&\hspace{2cm} \id_{V \cap V - y_3} \conv \id_{V \cap V - y_3}(w)\,\, \id_{V \cap V - (y_3 + b_1 - b_3)} \conv \id_{V \cap V - (y_3 + b_1 - b_3)}(w) \\
&\hspace{2cm} \id_{V \cap V - z_1} \conv \id_{V \cap V - z_1}(w)\,\, \id_{V \cap V - (b_1 + b_2 - z_1)} \conv \id_{V \cap V - (b_1 + b_2 - z_1)}(w) \geq \exp\Big(- O\Big(\log^{O(1)}(2c_0^{-1})\Big)\Big) \delta^{36}.\end{align*}\end{corollary}

\begin{proof}In this proof, we set temporarily $A = A^{c_0/2}$. For each $w \in A$ we have 
\[\exx_{a} \Big(\id_{V \cap V - w} \conv \id_{V \cap V - w}(a)\Big)^2 \geq \frac{c_0}{2} \delta^7.\]
By Claim~\ref{Alargeness} we have $|A| \geq \frac{c}{40} |G|$. For each $w \in A$, apply Proposition~\ref{popDiffAddQuads} with set $P = G$ to obtain $\tilde{P}_w \subseteq G$ which contains at least $2^{-50} c_0^8 \delta^3 |G|^3$ additive quadruples, has size $|\tilde{P}_w| \leq 2^{11}c_0^{-2} \delta |G| $ and every element $a \in \tilde{P}_w$ satisfies $\id_{V \cap V - w} \conv \id_{V \cap V - w}(a) \geq \frac{1}{16}c_0 \delta^3$. Theorem~\ref{addquadsFreiman} applies to give a subspace\footnote{Averaging over elements $a_2, a_3, a_4$ in conclusion of Theorem~\ref{addquadsFreiman} we get a coset instead of a subspace for $S_w$, but then we simply take a subspace contatining $S_w$ with dimension greater by 1.} $S_w$ such that $|S_w| \geq c_1 \delta |G|$ and $|\tilde{P}_w \cap S_w| \geq c_1 |S_w|$, where $c_1 \geq \exp\Big(- O\Big(\log^{O(1)}(2c_0^{-1})\Big)\Big)$. Write $\tilde{P}'_w = \tilde{P}_w \cap S_w$.\\
 
Note that 

\begin{align*}&\exx_w \exx_{b_1, b_2, b_3, x_2, y_3, z_1} \id_{\tilde{P}_w}(b_1 + b_2 - b_3) \id_{\tilde{P}_w}(b_1) \id_{\tilde{P}_w}(b_2) \id_{\tilde{P}_w}(b_3) \id_{\tilde{P}_w}(x_2) \id_{\tilde{P}_w}(x_2 - b_2 + b_3) \\
&\hspace{6cm}\id_{\tilde{P}_w}(y_3) \id_{\tilde{P}_w}(y_3 + b_1 - b_3) \id_{\tilde{P}_w}(z_1) \id_{\tilde{P}_w}(b_1 + b_2 - z_1)\\
&\hspace{2cm}\geq \exx_w \frac{|S_w|^6}{|G|^6} \Big(\exx_{b_1, b_2, b_3, x_2, y_3, z_1 \in S_w} \id_{\tilde{P}'_w}(b_1 + b_2 - b_3) \id_{\tilde{P}'_w}(b_1) \id_{\tilde{P}'_w}(b_2) \id_{\tilde{P}'_w}(b_3) \id_{\tilde{P}'_w}(x_2) \id_{\tilde{P}'_w}(x_2 - b_2 + b_3) \\
&\hspace{6cm}\id_{\tilde{P}'_w}(y_3) \id_{\tilde{P}'_w}(y_3 + b_1 - b_3) \id_{\tilde{P}'_w}(z_1) \id_{\tilde{P}'_w}(b_1 + b_2 - z_1)\Big) \geq (c_1\delta)^6 \cdot c_1^{32},\end{align*}

where we used Lemma~\ref{multipleAddQuadsBound} for each $\tilde{P}'_w$ and $S_w$ in the last step.\\

However, every element $a \in \tilde{P}_w$ satisfies $\id_{V \cap V - w} \conv \id_{V \cap V - w}(a) \geq \frac{1}{16}c \delta^3$ so we conclude that

\begin{align}&\exx_w \exx_{b_1, b_2, b_3, x_2, y_3, z_1} \id_{\tilde{P}_w}(b_1 + b_2 - b_3) \id_{V \cap V - w} \conv \id_{V \cap V - w}(b_1 + b_2 - b_3)\nonumber \\
&\hspace{1cm}\id_{\tilde{P}_w}(b_1) \id_{V \cap V - w} \conv \id_{V \cap V - w}(b_1) \id_{\tilde{P}_w}(b_2) \id_{V \cap V - w} \conv \id_{V \cap V - w}(b_2) \id_{\tilde{P}_w}(b_3) \id_{V \cap V - w} \conv \id_{V \cap V - w}(b_3)\nonumber\\
&\hspace{1cm} \id_{\tilde{P}_w}(x_2) \id_{V \cap V - w} \conv \id_{V \cap V - w}(x_2) \id_{\tilde{P}_w}(x_2 - b_2 + b_3)  \id_{V \cap V - w} \conv \id_{V \cap V - w}(x_2 - b_2 + b_3)\nonumber\\
&\hspace{1cm}\id_{\tilde{P}_w}(y_3) \id_{V \cap V - w} \conv \id_{V \cap V - w}(y_3) \id_{\tilde{P}_w}(y_3 + b_1 - b_3) \id_{V \cap V - w} \conv \id_{V \cap V - w}(y_3 + b_1  - b_3)\nonumber \\
&\hspace{1cm}\id_{\tilde{P}_w}(z_1) \id_{V \cap V - w} \conv \id_{V \cap V - w}(z_1) \id_{\tilde{P}_w}(b_1 + b_2 - z_1) \id_{V \cap V - w} \conv \id_{V \cap V - w}(b_1 + b_2 - z_1) \geq 2^{-40} c_0^{10} c_1^{38} \delta^{36}.\label{intersectionlarge10tuples}\end{align}

Finally, notice
\begin{align*}\id_{V \cap V - w} \conv \id_{V \cap V - w}(t) = &\exx_x \id_{V \cap V - w}(x + t) \id_{V \cap V - w}(x) =  \exx_x \id_V(x + t) \id_V(x + t + w) \id_V(x) \id_V(x + w)\\
 = &\exx_x \id_{V - t}(x) \id_{V - t- w}(x) \id_V(x) \id_{V - w}(x) = \frac{|V \cap (V - t) \cap (V - w) \cap (W - w - t)|}{|G|},\end{align*}
which is symmetric in $w$ and $t$, and thus $\id_{V \cap V - w} \conv \id_{V \cap V - w}(t) =\id_{V \cap V - t} \conv \id_{V \cap V - t}(w)$. Using this identity 10 times in~\eqref{intersectionlarge10tuples} completes the proof.\end{proof}

\vspace{\baselineskip}

Next, we combine Proposition~\ref{popDiffAddQuads} with standard additive combinatorial arguments to approximate the set of large values of $b \mapsto \id_{V \cap V -a} \conv \id_{V \cap V -a}(b)$ by a union of large subspaces for each $a\in A^\xi$.\\

\begin{proposition} \label{tildeudefnpropn}Suppose that $\varepsilon \leq 2^{-500}\xi^{32}\delta^{32}$. There exist positive integers $M, K \leq \exp\Big(\log^{O(1)} (2\xi^{-1})\Big)$ such that the following holds. For each $a \in A^\xi$, we may find $m = m_a \leq M$ pairs of subspaces $U_{a, 1} \leq U'_{a,1}, U_{a, 2} \leq U'_{a,2}, \dots, U_{a, m} \leq U'_{a,m}$, such that
\begin{itemize}
\item[\textbf{(i)}] for each $i \in [m]$ the bounds $K^{-1} \delta |G| \leq |U_{a,i}| \leq |U'_{a,i}| \leq K\delta |G|$ hold,
\item[\textbf{(ii)}] for each $i \in [m]$ we have $|U_{a,i}| = p^{-1}|U'_{a,i}|$,
\item[\textbf{(iii)}] writing $\tilde{U}_a = \cup_{i \in [m]} U'_{a,i}$, we have 
\begin{equation}\exx_{b} \id_{G \setminus \tilde{U}_a}(b) \Big(\id_{V \cap V -a} \conv \id_{V \cap V -a}(b)\Big)^2 \leq \frac{\xi}{100} \delta^7,\label{subspacesl2approx}\end{equation} 
and,
\item[\textbf{(iv)}] for each $b \in \cup_{i \in [m]} U_{a,i}$ the number of ordered 8-tuples $(x_1, x_2, \dots, x_8) \in (V \cap (V-a))^8$ such that $\sum_{i \in [8]} (-1)^ix_i = b$ is at least $(\xi/2)^{O(1)}\delta^{15} |G|^7$.
\end{itemize}
\end{proposition}

Note that property \textbf{(iv)} in the conclusion can be expressed more succinctly for each $b \in \cup_{i \in [m]} U_{a,i}$ using iterated convolution notation as 
\begin{align*}(\xi/2)^{O(1)}\delta^{15} |G|^7 \leq &\sum_{x_1, \dots, x_8 \in G} \id_{V \cap V - a}(x_1) \dots \id_{V \cap V - a}(x_8) \id\Big(\sum_{i \in [8]} (-1)^ix_i = b\Big)\\
= &\sum_{x_1, \dots, x_7 \in G} \id_{V \cap V - a}(x_1) \dots \id_{V \cap V - a}(x_7) \id_{V \cap V - a}\Big(x_1 - x_2 + x_3 -x_4 + x_5 - x_6 + x_7 + b\Big)\\
= &|G|^7 {\conv}^{(8)} \id_{V \cap V-a}(-b).\end{align*}
Since each $ U_{a,i}$ is symmetric, we have ${\conv}^{(8)} \id_{V \cap V-a}(b) \geq (\xi/2)^{O(1)}\delta^{15}$ for all $b \in \cup_{i \in [m]} U_{a,i}$. 

\begin{proof}[Proof of Proposition~\ref{tildeudefnpropn}]Fix $a \in A^\xi$. We use an iterative argument, where at each step we find additional subspaces $U_{a, i}$ and $U'_{a,i}$ until~\eqref{subspacesl2approx} is satisfied. Suppose that so far we have subspaces $U_{a, 1} \leq U'_{a,1}, U_{a, 2} \leq U'_{a,2}, \dots, U_{a, i-1} \leq U'_{a,i-1}$, for some $i \geq 1$. Assume that~\eqref{subspacesl2approx} fails with the current definition of $\tilde{U}_a$. Let $P \subseteq G \setminus \tilde{U}_a$ be the set of all elements $b \in G \setminus \tilde{U}_a$ such that $\id_{V \cap V -a} \conv \id_{V \cap V -a}(b) \geq \frac{\xi}{1000}\delta^3$. Then we have
\begin{align*}\exx_{b} \id_{G \setminus (\tilde{U}_a \cup P)}(b) \Big(\id_{V \cap V -a} \conv \id_{V \cap V -a}(b)\Big)^2 \leq & \frac{\xi}{1000}\delta^3\exx_{b} \id_{G \setminus (\tilde{U}_a \cup P)}(b) \id_{V \cap V -a} \conv \id_{V \cap V -a}(b)\\
\leq &\frac{\xi}{1000}\delta^3\exx_{b} \id_{V \cap V -a} \conv \id_{V \cap V -a}(b) \leq \frac{\xi}{250}\delta^7,
\end{align*}
where we used assumption \textbf{(i)} of regularity for $a$. Since~\eqref{subspacesl2approx} fails, we have
\[\exx_{b} \id_{P}(b) \Big(\id_{V \cap V -a} \conv \id_{V \cap V -a}(b)\Big)^2 \geq \frac{\xi}{200}\delta^7.\]
By Proposition~\ref{popDiffAddQuads}, provided $\sqrt[4]{\varepsilon} \leq 2^{-105}\xi^8\delta^8$, there exists a subset $\tilde{P} \subseteq P$ such that $|\tilde{P}| \leq 2^{25}\xi^{-2} \delta |G|$ and $\tilde{P}$ has at least $2^{-106}\xi^8\delta^3|G|^3$ additive quadruples. By Theorem~\ref{addquadsFreiman}, there exists a subspace $U_{a,i} \leq G$ of size $|U_{a,i}| \geq \exp\Big(-\log^{O(1)} (2\xi^{-1})\Big) \delta |G|$ such that for each $v \in U_{a,i}$ there are at least $(\xi/2)^{O(1)}\delta^3|G|^3$ quadruples $(b_1, b_2, b_3, b_4) \in \tilde{P}^4$ such that $b_1 + b_2 - b_3 - b_4 = v$, from which the last property in the conclusion follows. (Note that this condition also implies $|U_{a,i}| \leq \exp\Big(-\log^{O(1)} (2\xi^{-1})\Big) \delta |G|$.) In particular, by averaging over $b_2, b_3, b_4$, we obtain an element $t \in G$ such that $|t + U_{a,i} \cap P| \geq \exp\Big(-\log^{O(1)} (2\xi^{-1})\Big)  \delta |G|$. We now set $U'_{a,i} = U_{a,i} + \langle t \rangle$.\\
\indent By regularity of $a$, we have $\ex_b \id_{V \cap V -a} \conv \id_{V \cap V -a}(b) \leq 4\delta^4$, so the total number of $b\in G$ such that $\id_{V \cap V -a} \conv \id_{V \cap V -a}(b) \geq \frac{\xi}{1000}\delta^3$ is at most $(2\xi^{-1})^{O(1)} \delta|G|$. The procedure terminates after at most $\exp\Big(\log^{O(1)} (2\xi^{-1})\Big)$ steps, since we cover at least $\exp\Big(-\log^{O(1)} (2\xi^{-1})\Big)  \delta |G|$ of such elements $b$ by $U'_{a,i}$ that were not covered by $\tilde{U}_a$ previously.\end{proof}

By the way we defined the set $A^\xi$ and unions of subspaces $\tilde{U}_a$, we deduce the following approximation property.

\begin{claim}\label{tildeupassclaim}Provided $\xi \leq c_0/2$ and $\varepsilon \leq 2^{-500}\xi^{32}\delta^{32}$, we have
\[\exx_{a,b} \Big(1 - \id_{A^\xi}(a)\id_{\tilde{U}_a}(b)\Big)\id_{V \cap V -a} \conv \id_{V \cap V -a}(b)^2 \leq 2\xi \delta^7 + 40\sqrt[4]{\varepsilon}.\]\end{claim}

\begin{proof}The expression in question is the sum of
\begin{equation}\label{smallAUexp1}\exx_{a,b} \Big(1 - \id_{A^\xi}(a)\Big)\id_{V \cap V -a} \conv \id_{V \cap V -a}(b)^2\end{equation}
and
\begin{equation}\label{smallAUexp2}\exx_{a,b} \id_{A^\xi}(a)\Big(1 - \id_{\tilde{U}_a}(b)\Big)\id_{V \cap V -a} \conv \id_{V \cap V -a}(b)^2\end{equation}
which we estimate in turn. The inequality~\eqref{l2smallcontribreg} shows that the expression in~\eqref{smallAUexp1} is at most $\xi \delta^7 + 40\sqrt[4]{\varepsilon}$. On the other hand, by properties in the conclusion of Proposition~\ref{tildeudefnpropn}, we get that the expression in~\eqref{smallAUexp2} is at most $\frac{\xi}{100}\delta^7$.\end{proof}

\subsection{Finding approximate quasirandom linear systems of subspaces}

The following proposition is the key step towards obtaining an approximate quasirandom linear system of subspaces.

\begin{proposition}\label{addquad10subintnlarge} There is an absolute constant $D \geq 1$ such that the following holds. Write $A = A^\xi$. Provided $\xi \leq \exp(-D \log^D (2c_0^{-1}))$ and $\varepsilon \leq 2^{-500} \xi^{32} \delta^{288}$, we have
\begin{align*}&\exx_{\ssk{b_1, b_2, b_3\\x_2, y_3, z_1}} \id_A(b_1 + b_2 - b_3)\id_A(b_1)\id_A(b_2)\id_A(b_3) \id_A(x_2) \id_A(x_2 - b_2 + b_3) \id_A(y_3)\id_A(y_3 + b_1 - b_3)\id_A(z_1)\id_A(b_1 + b_2 - z_1) \\
&\hspace{2cm}|\tilde{U}_{b_1 + b_2 - b_3} \cap \tilde{U}_{b_1} \cap \tilde{U}_{b_2} \cap \tilde{U}_{b_3} \cap \tilde{U}_{x_2} \cap \tilde{U}_{x_2 - b_2 + b_3}\cap  \tilde{U}_{y_3} \cap \tilde{U}_{y_3 + b_1 - b_3}\cap  \tilde{U}_{z_1} \cap \tilde{U}_{b_1 + b_2 - z_1}| \\
&\hspace{13cm}\geq \exp(-D \log^D (2c_0^{-1})) \delta^{6}|G|.\end{align*}
\end{proposition}

\begin{proof}Let us begin by applying Corollary~\ref{commonadditivequads10Cor}. Thus
\begin{align*}&\exx_{b_1, b_2, b_3, x_2, y_3, z_1}\exx_w \id_{V \cap V - (b_1 + b_2 - b_3)} \conv\id_{V \cap V -(b_1 + b_2 - b_3)}(w) \,\,\id_{V \cap V - b_1} \conv \id_{V \cap V - b_1}(w)\,\, \id_{V \cap V - b_2} \conv \id_{V \cap V - b_2}(w)\\
&\hspace{2cm} \id_{V \cap V - b_3} \conv \id_{V \cap V - b_3}(w)\,\, \id_{V \cap V - x_2} \conv \id_{V \cap V - x_2}(w)\,\, \id_{V \cap V - (x_2 - b_2 + b_3)} \conv \id_{V \cap V - (x_2 - b_2 + b_3)}(w) \\
&\hspace{2cm} \id_{V \cap V - y_3} \conv \id_{V \cap V - y_3}(w)\,\, \id_{V \cap V - (y_3 + b_1 - b_3)} \conv \id_{V \cap V - (y_3 + b_1 - b_3)}(w) \\
&\hspace{2cm} \id_{V \cap V - z_1} \conv \id_{V \cap V - z_1}(w)\,\, \id_{V \cap V - (b_1 + b_2 - z_1)} \conv \id_{V \cap V - (b_1 + b_2 - z_1)}(w) \geq c_1 \delta^{36},\end{align*}
where $c_1 \geq \exp\Big(- O\Big(\log^{O(1)}(2c_0^{-1})\Big)\Big)$. In the next step of the proof, we show that we may restrict attention to elements $b_1, b_2, b_3, x_2, y_3, z_1$ such that
\[b_1 + b_2 - b_3, b_1, b_2, b_3, x_2, x_2 - b_2 + b_3, y_3, y_3 + b_1 - b_3,z_1, b_1 + b_2 - z_1 \in A\]
and
\[w \in \tilde{U}_{b_1 + b_2 - b_3}, \tilde{U}_{b_1}, \tilde{U}_{b_2}, \tilde{U}_{b_3}, \tilde{U}_{x_2}, \tilde{U}_{x_2 - b_2 + b_3}, \tilde{U}_{y_3}, \tilde{U}_{y_3 + b_1 - b_3}, \tilde{U}_{z_1}, \tilde{U}_{b_1 + b_2 - z_1}.\]
To see that, we give an upper bound for
\begin{align}&\exx_{b_1, b_2, b_3, x_2, y_3, z_1}\exx_w \Big(1 - \id_{A}(b_1)\id_{\tilde{U}_{b_1}}(w)\Big) \id_{V \cap V - b_1} \conv \id_{V \cap V - b_1}(w)\,\,\id_{V \cap V - (b_1 + b_2 - b_3)} \conv\id_{V \cap V -(b_1 + b_2 - b_3)}(w) \,\, \nonumber\\
&\hspace{1cm} \id_{V \cap V - b_2} \conv \id_{V \cap V - b_2}(w)\,\,\id_{V \cap V - b_3} \conv \id_{V \cap V - b_3}(w)\,\, \id_{V \cap V - x_2} \conv \id_{V \cap V - x_2}(w) \nonumber\\
&\hspace{1cm} \id_{V \cap V - (x_2 - b_2 + b_3)} \conv \id_{V \cap V - (x_2 - b_2 + b_3)}(w)\,\,\id_{V \cap V - y_3} \conv \id_{V \cap V - y_3}(w)\,\, \id_{V \cap V - (y_3 + b_1 - b_3)} \conv \id_{V \cap V - (y_3 + b_1 - b_3)}(w) \nonumber\\
&\hspace{1cm} \id_{V \cap V - z_1} \conv \id_{V \cap V - z_1}(w)\,\, \id_{V \cap V - (b_1 + b_2 - z_1)} \conv \id_{V \cap V - (b_1 + b_2 - z_1)}(w).\label{elementsrestrictionError10subsets}\end{align}
Using Cauchy-Schwarz inequality (note below that $1 - \id_{A}(b_1)\id_{\tilde{U}_{b_1}}(w) \in \{0,1\}$ so it equals its square), this expression is at most

\begin{align}&\Big(\exx_{b_1,w} \Big(1 - \id_{A}(b_1)\id_{\tilde{U}_{b_1}}(w)\Big) \id_{V \cap V - b_1} \conv \id_{V \cap V - b_1}(w)^2\Big)^{1/2} \Big(\exx_{b_1,w} \Big(\exx_{b_2, b_3, x_2, y_3, z_1} \id_{V \cap V - (b_1 + b_2 - b_3)} \conv\id_{V \cap V -(b_1 + b_2 - b_3)}(w) \,\, \nonumber\\
&\hspace{1cm} \id_{V \cap V - b_2} \conv \id_{V \cap V - b_2}(w)\,\,\id_{V \cap V - b_3} \conv \id_{V \cap V - b_3}(w)\,\, \id_{V \cap V - x_2} \conv \id_{V \cap V - x_2}(w) \nonumber\\
&\hspace{1cm} \id_{V \cap V - (x_2 - b_2 + b_3)} \conv \id_{V \cap V - (x_2 - b_2 + b_3)}(w)\,\,\id_{V \cap V - y_3} \conv \id_{V \cap V - y_3}(w)\,\, \id_{V \cap V - (y_3 + b_1 - b_3)} \conv \id_{V \cap V - (y_3 + b_1 - b_3)}(w) \nonumber\\
&\hspace{1cm} \id_{V \cap V - z_1} \conv \id_{V \cap V - z_1}(w)\,\, \id_{V \cap V - (b_1 + b_2 - z_1)} \conv \id_{V \cap V - (b_1 + b_2 - z_1)}(w)\Big)^2\Big)^{1/2}.\label{passingtoACSstep1_10subsets}\end{align}

Claim~\ref{tildeupassclaim} implies
\begin{equation}\label{tildeupassclaimEqnBound10subsets}\exx_{b_1,w} \Big(1 - \id_{A}(b_1)\id_{\tilde{U}_{b_1}}(w)\Big)\id_{V \cap V - b_1} \conv \id_{V \cap V - b_1}(w)^2 \leq 2\xi \delta^7 + 40\sqrt[4]{\varepsilon}.\end{equation}

When it comes to the second long sum in~\eqref{passingtoACSstep1_10subsets}, note that  
\begin{align*}
&\exx_{b_1,w} \Big(\exx_{b_2, b_3, x_2, y_3, z_1} \id_{V \cap V - (b_1 + b_2 - b_3)} \conv\id_{V \cap V -(b_1 + b_2 - b_3)}(w) \,\, \id_{V \cap V - b_2} \conv \id_{V \cap V - b_2}(w)\,\,\id_{V \cap V - b_3} \conv \id_{V \cap V - b_3}(w)\nonumber\\
&\hspace{1cm} \id_{V \cap V - x_2} \conv \id_{V \cap V - x_2}(w) \,\,\id_{V \cap V - (x_2 - b_2 + b_3)} \conv \id_{V \cap V - (x_2 - b_2 + b_3)}(w)\,\,\id_{V \cap V - y_3} \conv \id_{V \cap V - y_3}(w)\nonumber\\
&\hspace{1cm} \id_{V \cap V - (y_3 + b_1 - b_3)} \conv \id_{V \cap V - (y_3 + b_1 - b_3)}(w) \,\,\id_{V \cap V - z_1} \conv \id_{V \cap V - z_1}(w)\,\, \id_{V \cap V - (b_1 + b_2 - z_1)} \conv \id_{V \cap V - (b_1 + b_2 - z_1)}(w)\Big)^2\\
= &\exx_{b_1,w} \exx_{\ssk{b_2, b_3, x_2, y_3, z_1\\b'_2, b'_3, x'_2, y'_3, z'_1}} \id_{V \cap V - (b_1 + b_2 - b_3)} \conv\id_{V \cap V -(b_1 + b_2 - b_3)}(w) \,\, \id_{V \cap V - b_2} \conv \id_{V \cap V - b_2}(w)\,\,\id_{V \cap V - b_3} \conv \id_{V \cap V - b_3}(w)\nonumber\\
&\hspace{1cm} \id_{V \cap V - x_2} \conv \id_{V \cap V - x_2}(w) \,\,\id_{V \cap V - (x_2 - b_2 + b_3)} \conv \id_{V \cap V - (x_2 - b_2 + b_3)}(w)\,\,\id_{V \cap V - y_3} \conv \id_{V \cap V - y_3}(w)\nonumber\\
&\hspace{1cm} \id_{V \cap V - (y_3 + b_1 - b_3)} \conv \id_{V \cap V - (y_3 + b_1 - b_3)}(w) \,\,\id_{V \cap V - z_1} \conv \id_{V \cap V - z_1}(w)\,\, \id_{V \cap V - (b_1 + b_2 - z_1)} \conv \id_{V \cap V - (b_1 + b_2 - z_1)}(w)\\
&\hspace{1cm}\id_{V \cap V - (b_1 + b'_2 - b'_3)} \conv\id_{V \cap V -(b_1 + b'_2 - b'_3)}(w) \,\, \id_{V \cap V - b'_2} \conv \id_{V \cap V - b'_2}(w)\,\,\id_{V \cap V - b'_3} \conv \id_{V \cap V - b'_3}(w)\nonumber\\
&\hspace{1cm} \id_{V \cap V - x'_2} \conv \id_{V \cap V - x'_2}(w) \,\,\id_{V \cap V - (x'_2 - b'_2 + b'_3)} \conv \id_{V \cap V - (x'_2 - b'_2 + b'_3)}(w)\,\,\id_{V \cap V - y'_3} \conv \id_{V \cap V - y'_3}(w)\nonumber\\
&\hspace{1cm} \id_{V \cap V - (y'_3 + b_1 - b'_3)} \conv \id_{V \cap V - (y'_3 + b_1 - b'_3)}(w) \,\,\id_{V \cap V - z'_1} \conv \id_{V \cap V - z'_1}(w)\,\, \id_{V \cap V - (b_1 + b'_2 - z'_1)} \conv \id_{V \cap V - (b_1 + b'_2 - z'_1)}(w).\end{align*}

To evade unnecessary writing, note that each of the 18 convolutions above gives rise to additional dummy variable in further expansion. Let us name these variables $v_1, \dots, v_9$ and $v'_1, \dots, v'_9$ in this order. Thus, $\id_{V \cap V - (b_1 + b_2 - b_3)} \conv\id_{V \cap V -(b_1 + b_2 - b_3)}(w)$ expands as 
\[\exx_{v_1} \id_V(v_1) \id_V(v_1 + b_1 + b_2 - b_3) \id_V(v_1 + w) \id_V(v_1 + b_1 + b_2 - b_3 + w)\]
and the last term $\id_{V \cap V - (b_1 + b'_2 - z'_1)} \conv \id_{V \cap V - (b_1 + b'_2 - z'_1)}(w)$ expands as
\[\exx_{v'_9} \id_V(v'_9) \id_V(v'_9+ b_1 + b'_2 - z'_1) \id_V(v'_9 + w) \id_V(v'_9 + b_1 + b'_2 - z'_1 + w),\]
so we get 72 terms in total. Let us omit the following 7 terms
\begin{align*}&\id_V(v_5 + x_2 - b_2 + b_3 + w),\, \id_V(v_7 + y_3 + b_1 - b_3 + w),\, \id_V(v_9 + b_1 + b_2 - z_1 + w),\, \id_V(v'_1 + b_1 + b'_2 - b'_3 + w),\\
&\hspace{2cm}\id_V(v'_5 + x'_2 - b'_2 + b'_3 + w),\, \id_V(v'_7 + y'_3 + b_1 - b'_3 + w),\, \id_V(v'_9 + b_1 + b'_2 - z'_1 + w).\end{align*}
We may use Lemma~\ref{u2controlBasic} 65 times to get the bound $\delta^{65} + 65\varepsilon$, as follows. First, use $w$ with each of 11 variables $b_1, b_2, b'_2, \dots, z_1, z'_1$ to remove the remaining 11 terms that involve $w$ and one of the mentioned variables. Then use $w$ and a variable among $v_1, \dots, v'_9$ to remove further 18 terms. The remaining 36 terms have obvious choices of variables.\\

Hence, the second long sum in~\eqref{passingtoACSstep1_10subsets} is at most $2\delta^{65}$, since $\varepsilon \leq \frac{1}{65}\delta^{65}$.\\

Let us define $t \colon G^6 \to G^{10}$
\[t(b_1, b_2, b_3, x_2, y_3, z_1) = \Big(b_1, b_2, b_3, b_1 + b_2 - b_3, x_2, x_2 - b_2 + b_3, y_3, y_3 + b_1 - b_3, z_1, b_1 + b_2 - z_1\Big)\]
which is the 10-tuple of points appearing in the expressions above. Furthermore, let $F(b_1, b_2, b_3, x_2, y_3, z_1; w)$ for the product $\prod_{i \in [10]}\id_{V \cap V - t_i(b_1, b_2, b_3, x_2, y_3, z_1)} \conv\id_{V \cap V - t_i(b_1, b_2, b_3, x_2, y_3, z_1)}(w)$. The upper bound on~\eqref{elementsrestrictionError10subsets} then becomes
\[\exx_{b_1, b_2, b_3, x_2, y_3, z_1, w} \Big(1 - \id_{A}(b_1)\id_{\tilde{U}_{b_1}}(w)\Big) F(b_1, b_2, b_3, x_2, y_3, z_1; w) \leq 2\sqrt{\xi}\delta^{36} + 10\sqrt[8]{\varepsilon}.\]

Note that the actual signs in the linear combinations such as $b_1 + b_2 -b_3$ do not play a role in the argument above, so we could have had any $\pm b_1 \pm b_2 \pm b_3$ instead. The argument can be used to prove the same inequality but with $\Big(1 - \id_{A}(s)\id_{\tilde{U}_{s}}(w)\Big)$ instead of $\Big(1 - \id_{A}(b_1)\id_{\tilde{U}_{b_1}}(w)\Big)$, where $s$ is any of the remaining 9 possibilities. Namely, for $s$ among $b_2, b_3, x_2, y_3, z_1$ we apply almost the same arguments with $s$ in place of $b_1$, with the slight difference in the terms that are neglected. On the other hand, for other possibilities, we need to change variables. When $s = b_1 + b_2 -b_3$ we replace $b_3$ by $b_1 + b_2 - b_3$ and $x_2$ by $y_3$, thus reducing that case to $s = b_3$ (with a slight difference in signs in linear combinations, which is not an issue). When $s= x_2 - b_2 + b_3$ we replace $x_2$ by $x_2 + b_2 - b_3$ to reduce to the case when $s = x_2$, with similar arguments for $s \in \{y_3 + b_1 - b_3, b_1 + b_2 - z_1\}$.\\
\indent Let $I(s; w) = \Big(1 - \id_{A}(s)\id_{\tilde{U}_{s}}(w)\Big)$, which takes values 0 and 1. In particular
\begin{align*}&\bigg|\bigg(\exx_{b_1, b_2, b_3, x_2, y_3, z_1, w} \Big(\prod_{i \in [10]} \Big(1 - I(t_i(b_1, b_2, b_3, x_2, y_3, z_1); w)\Big)\Big) F(b_1, b_2, b_3, x_2, y_3, z_1; w)\bigg) \\
&\hspace{9cm}- \bigg(\exx_{b_1, b_2, b_3, x_2, y_3, z_1, w}F(b_1, b_2, b_3, x_2, y_3, z_1; w)\bigg)\bigg|\\
&\hspace{1cm} = \bigg|\sum_{j \in [10]} \bigg(\exx_{b_1, b_2, b_3, x_2, y_3, z_1, w}\Big(\prod_{i \in [j, 10]} \Big(1 - I(t_i(b_1, b_2, b_3, x_2, y_3, z_1); w)\Big)\Big) F(b_1, b_2, b_3, x_2, y_3, z_1; w)\\
&\hspace{6cm} -\Big(\prod_{i \in [j + 1, 10]} \Big(1 - I(t_i(b_1, b_2, b_3, x_2, y_3, z_1); w)\Big)\Big) F(b_1, b_2, b_3, x_2, y_3, z_1; w)\bigg) \bigg|\\
&\hspace{1cm}\leq\sum_{j \in [10]} \bigg|\exx_{b_1, b_2, b_3, x_2, y_3, z_1, w} I(t_j(b_1, b_2, b_3, x_2, y_3, z_1); w) \Big(\prod_{i \in [j+1, 10]} \Big(1 - I(t_i(b_1, b_2, b_3, x_2, y_3, z_1); w)\Big)\Big)\\
&\hspace{12cm} F(b_1, b_2, b_3, x_2, y_3, z_1; w)\bigg|\\
&\hspace{1cm}\leq \sum_{j \in [10]} \exx_{b_1, b_2, b_3, x_2, y_3, z_1, w} I(t_j(b_1, b_2, b_3, x_2, y_3, z_1); w) F(b_1, b_2, b_3, x_2, y_3, z_1; w),\end{align*}

where we used the fact that all terms take values in $[0,1]$ to neglect some of them in the last line. Returing to our original notation, we conclude that
\begin{align*}&\exx_{b_1, b_2, b_3, x_2, y_3, z_1}\exx_w \id_{A}(b_1)\id_{\tilde{U}_{b_1}}(w) \id_{V \cap V - b_1} \conv \id_{V \cap V - b_1}(w)\,\,\id_{A}(b_2)\id_{\tilde{U}_{b_2}}(w) \id_{V \cap V - b_2} \conv \id_{V \cap V - b_2}(w) \nonumber\\
&\hspace{0.5cm}\id_{A}(b_3)\id_{\tilde{U}_{b_3}}(w)\id_{V \cap V - b_3} \conv \id_{V \cap V - b_3}(w) \,\,\id_{A}(b_1 + b_2 - b_3)\id_{\tilde{U}_{b_1 + b_2 - b_3}}(w)\id_{V \cap V - (b_1 + b_2 - b_3)} \conv\id_{V \cap V -(b_1 + b_2 - b_3)}(w)\\
&\hspace{0.5cm}\id_{A}(x_2)\id_{\tilde{U}_{x_2}}(w) \id_{V \cap V - x_2} \conv \id_{V \cap V - x_2}(w)\,\,\id_{A}(x_2 - b_2 + b_3)\id_{\tilde{U}_{x_2 - b_2 + b_3}}(w) \id_{V \cap V - (x_2 - b_2 + b_3)} \conv \id_{V \cap V - (x_2 - b_2 + b_3)}(w)\\
&\hspace{0.5cm}\id_{A}(y_3)\id_{\tilde{U}_{y_3}}(w)\id_{V \cap V - y_3} \conv \id_{V \cap V - y_3}(w)\,\, \id_{A}(y_3 + b_1 - b_3)\id_{\tilde{U}_{y_3 + b_1 - b_3}}(w) \id_{V \cap V - (y_3 + b_1 - b_3)} \conv \id_{V \cap V - (y_3 + b_1 - b_3)}(w) \nonumber\\
&\hspace{0.5cm} \id_{A}(z_1)\id_{\tilde{U}_{z_1}}(w)\id_{V \cap V - z_1} \conv \id_{V \cap V - z_1}(w)\,\, \id_{A}(b_1 + b_2 - z_1)\id_{\tilde{U}_{b_1 + b_2 - z_1}}(w) \id_{V \cap V - (b_1 + b_2 - z_1)} \conv \id_{V \cap V - (b_1 + b_2 - z_1)}(w)\\
&\hspace{3cm} \geq \frac{c_1}{2}\delta^{36},\end{align*}
provided $\xi \leq 2^{-20}c_1^2$ and $\varepsilon \leq 2^{-100}c_1^8\delta^{288}$.\\

Since every element $a \in A$ is $\sqrt[4]{\varepsilon}$-regular, by property \textbf{(ii)} we have that $\id_{V \cap V -a} \conv \id_{V \cap V -a}(b) \leq 2\delta^3$ holds for all but at most $\sqrt[4]{\varepsilon}|G|$ of $b \in G$. Combining this fact with the inequality above, we see that
\begin{align*} &\frac{c_1}{2}\delta^{36} \leq 10\sqrt[4]{\varepsilon} + 2^{10}\delta^{30}\exx_{b_1, b_2, b_3, x_2, y_3, z_1}\exx_w  \id_{A}(b_1)\id_{\tilde{U}_{b_1}}(w) \id_{A}(b_2)\id_{\tilde{U}_{b_2}}(w) \id_{A}(b_3)\id_{\tilde{U}_{b_3}}(w) \id_{A}(b_1 + b_2 - b_3)\id_{\tilde{U}_{b_1 + b_2 - b_3}}(w)\\
&\hspace{2cm}\id_{A}(x_2)\id_{\tilde{U}_{x_2}}(w)\id_{A}(x_2 - b_2 + b_3)\id_{\tilde{U}_{x_2 - b_2 + b_3}}(w) \id_{A}(y_3)\id_{\tilde{U}_{y_3}}(w)\id_{A}(y_3 + b_1 - b_3)\id_{\tilde{U}_{y_3 + b_1 - b_3}}(w)\\
&\hspace{2cm} \id_{A}(z_1)\id_{\tilde{U}_{z_1}}(w) \id_{A}(b_1 + b_2 - z_1)\id_{\tilde{U}_{b_1 + b_2 - z_1}}(w).\end{align*}

Provided $\varepsilon \leq 2^{-40} c_1^4 \delta^{144}$, the claim follows. \end{proof}

\hspace{\baselineskip}

We now set \framebox{$\xi = \exp(-D \log^{D} (2c_0^{-1}))$}, where $D$ is the constant from Proposition~\ref{addquad10subintnlarge} and write $A = A^\xi$.\\

Let us also show that subspaces $U_{a,i}$ cannot have a large intersection for different choices of $a$.

\begin{claim}\label{multipleuintnsmall} Let $r \in \mathbb{N}$ and let $\varepsilon \leq 2^{-500}\xi^{32}\delta^{32}$. There exists $K \leq \exp\Big(r\log^{O(1)} (2\xi^{-1})\Big)$ such that for all but at most $64r \delta^{-32r} \varepsilon  |G|^r$ choices of $(a_1, \dots, a_r) \in A^r$ we have 
\[|U_{a_1, i_1} \cap U_{a_2, i_2} \cap \dots \cap U_{a_r, i_r}|  \leq K \delta^r |G|\]
for all indices $i_1 \in[m_{a_1}], \dots, i_r \in [m_{a_r}]$.\end{claim}

\begin{proof}Recall from Proposition~\ref{tildeudefnpropn} which constructs subspaces $U_{a,i}$ that we have parameters $M \leq \exp\Big(\log^{O(1)} (2\xi^{-1})\Big)$ and $\alpha \geq (\xi/2)^{O(1)} \delta^{15}$ such that $m_a \leq M$ for all $a \in A$ and that ${\conv}^{(8)}\id_{V \cap V-a}(b) \geq \alpha$ for all $b \in U_{a,i}$. Fix $a_1, \dots, a_r \in A$. Note that
\begin{align}&\alpha^r M^{-r} \sum_{i_1 \in [m_{a_1}], i_2 \in [m_{a_2}], \dots, i_r \in [m_{a_r}]}|U_{a_1, i_1} \cap U_{a_2, i_2} \cap \dots \cap U_{a_r, i_r}|\nonumber\\
&\hspace{2cm}\leq M^{-r}\sum_{i_1 \in [m_{a_1}], i_2 \in [m_{a_2}], \dots, i_r \in [m_{a_r}]} \exx_{b} \id_{U_{a_1, i_1} \cap \dots \cap U_{a_r, i_r}}(b) \Big({\conv}^{(8)}\id_{V \cap V-{a_1}}(b) \Big)\Big({\conv}^{(8)}\id_{V \cap V-a_2}(b) \Big) \nonumber\\
&\hspace{12cm}\dots \Big({\conv}^{(8)}\id_{V \cap V-a_r}(b)\Big) \nonumber\\
&\hspace{2cm} = \exx_{b}  \bigg(M^{-r}\sum_{i_1 \in [m_{a_1}], i_2 \in [m_{a_2}], \dots, i_r \in [m_{a_r}]} \id_{U_{a_1, i_1} \cap \dots \cap U_{a_r, i_r}}(b) \bigg) \Big({\conv}^{(8)}\id_{V \cap V-a_1}(b) \Big) \dots \Big({\conv}^{(8)}\id_{V \cap V-a_r}(b)\Big) \nonumber\\
&\hspace{2cm}\leq \exx_{b} \Big({\conv}^{(8)}\id_{V \cap V-a_1}(b)\Big) \dots \Big({\conv}^{(8)}\id_{V \cap V-a_r}(b)\Big).\label{allsubspacesintersectionIneq}\end{align}

Let us define $F(a_1, \dots, a_r) = \ex_{b} \Big({\conv}^{(8)}\id_{V \cap V-a_1}(b)\Big) \dots \Big({\conv}^{(8)}\id_{V \cap V-a_r}(b)\Big)$. We claim that $\|F - \delta^{16r}\|^2_{L^2(G^r)} \leq 64r\varepsilon$. To see that we first approximate $\ex_{a_1, \dots, a_r} F(a_1, \dots, a_r)^2$ by $\delta^{32r}$. We have
\begin{align*}&\exx_{a_1, \dots, a_r} F(a_1, \dots, a_r)^2 = \exx_{a_{[r]}} \exx_{u, v} \Big({\conv}^{(8)}\id_{V \cap V-a_1}(u)\Big) \dots \Big({\conv}^{(8)}\id_{V \cap V-a_r}(u)\Big)\Big({\conv}^{(8)}\id_{V \cap V-a_1}(v)\Big) \dots \Big({\conv}^{(8)}\id_{V \cap V-a_r}(v)\Big)\\
&\hspace{2cm}=\exx_{a_{[r]}} \exx_{u, v} \exx_{x_{[r] \times [7]}, y_{[r]\times [7]}} \bigg(\prod_{i \in [r]} \Big(\prod_{j \in [7]} \id_V(x_{i,j}) \id_V(x_{i,j} + a_i)\Big) \id_V\Big(\sum_{\ell \in [7]}(-1)^{\ell +1}x_{i, \ell} - u\Big)\\
&\hspace{12cm}\id_V\Big(\sum_{\ell \in [7]}(-1)^{\ell +1}x_{i, \ell} - u +a_i\Big)\bigg) \\
&\hspace{4cm}\bigg(\prod_{i \in [r]} \Big(\prod_{j \in [7]} \id_V(y_{i,j}) \id_V(y_{i,j} + a_i)\Big) \id_V\Big(\sum_{\ell \in [7]}(-1)^{\ell +1}y_{i, \ell} - v\Big)\\
&\hspace{12cm}\id_V\Big(\sum_{\ell \in [7]}(-1)^{\ell +1}y_{i, \ell} - v +a_i\Big)\bigg).\end{align*}

We apply Lemma~\ref{u2controlBasic} $32r$ times to the expression above, first we use pairs $(u, a_i)$ for $\id_V\Big(\sum_{\ell \in [7]}(-1)^{\ell +1}x_{i, \ell} - u + a_i\Big)$ and $(v, a_i)$ for $\id_V\Big(\sum_{\ell \in [7]}(-1)^{\ell +1}y_{i, \ell} - v +a_i\Big)$, then pairs $(u, x_{i, 1})$ for $\id_V\Big(\sum_{\ell \in [7]}(-1)^{\ell +1}x_{i, \ell} - u\Big)$ and $(v, y_{i, 1})$ for $\id_V\Big(\sum_{\ell \in [7]}(-1)^{\ell +1}y_{i, \ell} - v\Big)$ and obvious variables for the remaining terms. Similarly, we get $\Big|\ex_{a_{[r]}} F(a_1, \dots, a_r) -\delta^{16r}\Big| \leq 16r \varepsilon$. We conclude
\[\exx_{a_{[r]}} \Big|F(a_1, \dots, a_r) -\delta^{16r}\Big|^2 \leq 64r \varepsilon,\]
as claimed. In particular, $F(a_1, \dots, a_r) \leq 2\delta^{16r}$ holds for all but at most $64r \varepsilon \delta^{-32r}$ proportion of $(a_1, \dots, a_r) \in G^r$. Returning to~\eqref{allsubspacesintersectionIneq}, the proof is complete.\end{proof}

We are now ready to pass from an approximate quadratic variety to an approximate quasirandom linear system of subspaces.

\begin{proof}[Proof of Theorem~\ref{mainstep1thm}]Recall that $\xi = \exp(-D \log^{D} (2c^{-1}))$, where $D$ is the constant from Proposition~\ref{addquad10subintnlarge} and set $A = A^\xi$. Claim~\ref{Alargeness} implies that $|A^\xi| \geq \frac{c_0}{40}|G|$, giving property \textbf{(ii)}. Apply Proposition~\ref{tildeudefnpropn} for each $a \in A^\xi$ and then take $W_a$ uniformly at random from $U_{a,1}, \dots, U_{a, m_a}$. Properties \textbf{(iii)} and \textbf{(iv)} follow immediately. Property \textbf{(v)} follows from Claim~\ref{multipleuintnsmall}.\\
\indent Proposition~\ref{addquad10subintnlarge} and the fact that $m_a \leq \exp(\log^{O(1)} (2c_0^{-1}))$ for all $a \in A$ imply
\begin{align*}&\exx \sum_{\ssk{b_1, b_2, b_3\\x_2, y_3, z_1}} \id_A(b_1 + b_2 - b_3)\id_A(b_1)\id_{A^\xi}(b_2)\id_A(b_3) \id_A(x_2) \\
&\hspace{8cm}\id_A(x_2 - b_2 + b_3) \id_A(y_3)\id_A(y_3 + b_1 - b_3)\id_A(z_1)\id_A(b_1 + b_2 - z_1) \\
&\hspace{2cm}|W_{b_1 + b_2 - b_3} \cap W_{b_1} \cap W_{b_2} \cap W_{b_3} \cap W_{x_2} \cap W_{x_2 - b_2 + b_3}\cap  W_{y_3} \cap W_{y_3 + b_1 - b_3}\cap  W_{z_1} \cap W_{b_1 + b_2 - z_1}| \\
&\hspace{13cm}\geq \exp(- \log^{O(1)} (2c_0^{-1})) \delta^{6}|G|^7.\end{align*}
As
\begin{align*}&|W_{b_1 + b_2 - b_3} \cap W_{b_1} \cap W_{b_2} \cap W_{b_3} \cap W_{x_2} \cap W_{x_2 - b_2 + b_3}\cap  W_{y_3} \cap W_{y_3 + b_1 - b_3}\cap  W_{z_1} \cap W_{b_1 + b_2 - z_1}|\\
&\hspace{2cm}\leq |W_{b_1} \cap W_{b_2} \cap W_{b_3} \cap W_{x_2} \cap  W_{y_3}\cap  W_{z_1}|\end{align*}
which is at most $\exp(\log^{O(1)} (2c_0^{-1}))  \delta^6 |G|$ for all but at most $O(\delta^{-200} \varepsilon  |G|^6)$ choices of $(b_1, b_2, b_3, x_2, y_3, z_1)$, we may choose the subspaces $W_a$ so that property \textbf{(vi)} holds. The value of $c_1$ is taken as the minimum of relevant quantities arising in the previous steps.\end{proof}

\section{Structure of approximate quasirandom linear systems of subspaces}

This section is devoted to proving a structure theorem for approximate quasirandom linear systems of subspaces. The main result is the following theorem. Note that the conditions below are reformulated and simplified conclusions of Theorem~\ref{mainstep1thm}.

\begin{theorem}\label{approximatelinearsyssub}There exists an absolute constant $D \geq 1$ such that the following holds. Let $c > 0$ and let $d$ be a positive integer. Let $A \subseteq G$ be a set of size $|A| \geq c|G|$ and let $W_a \leq G$ be a subspace of codimension $d$ for each $a \in A$. Suppose that 
\[|W_{a_1} \cap W_{a_2} \cap \dots \cap W_{a_r}| \leq Kp^{-rd}|G|\]
holds for all but at most $\eta |G|^r$ $r$-tuples $(a_1, a_2, \dots, a_r) \in A^r$ for each $r \in [9]$. Assume furthermore that for at least $c |G|^6$ triples $(a, b_1, b_2, b_3, x_2, x_3, y_1, y_3, z_1, z_2) \in A^{10}$ we have 
\begin{itemize}
\item 
\begin{equation}\label{10subspacesConditionEqualities}a = b_1 + b_2 - b_3,\,\, x_3 = x_2 - b_2 + b_3, \,\, y_1 = y_3 + b_1 - b_3,\,\, z_2 = b_1 + b_2 - z_1,\end{equation}
and
\item the subspaces
\begin{equation}W_a \cap W_{b_1} \cap W_{b_2} \cap W_{b_3} \cap W_{x_2} \cap W_{x_3} \cap W_{y_1} \cap W_{y_3} \cap W_{z_1} \cap W_{z_2}\label{10subspacesCondition}\end{equation}
has size at least $K^{-1} p^{-6d}|G|$.
\end{itemize}
Then, provided $\eta \leq 2^{-31}c^3$, there exist parameters $c' \geq \exp\Big(-\exp\Big((\log (2c^{-1}) + \log_p K)^{D}\Big)\Big)$ and $r \leq \exp\Big((\log (2c^{-1}) + \log_p K)^D\Big)$, set $A' \subseteq A$ and a map $\Phi \colon G \times \mathbb{F}_p^d \to G$, affine in the first variable and linear in the second, such that $|A'| \geq c'  |G|$ and for each $a \in A'$ we have $|\on{Im} \Phi(a, \cdot) \cap W_a^\perp| \geq c' p^d$. Moreover, there exists a subspace $\Lambda \leq \mathbb{F}_p^d$ of dimension $r$ such that whenever $\lambda \notin \Lambda$ we have
\[\exx_{x,y} \omega\Big(\Phi(x, \lambda) \cdot y\Big) \leq \Big(\eta {c'}^{-2}\Big)^{1/2r}.\]
\end{theorem}

The key observation that allows us to pass from subspaces to mutually related linear maps is the next lemma.

\begin{lemma}\label{addquadshomsexist}Let $U_1, U_2, U_3, U_4 \leq G$ be subspaces of dimension $d$. Let $K \geq 1$ be a parameter such that 
\[K^{-1} p^{3d} \leq |U_{i_1} + U_{i_2} + U_{i_3}| \leq |U_1 + U_2 + U_3 + U_4|  \leq Kp^{3d}\]
for any three distinct elements $i_1, i_2, i_3 \in [4]$. Let $\phi_4 \colon \mathbb{F}_p^d \to U_4$ be a linear isomorphism. Then there exist linear isomorphisms $\phi_i \colon \mathbb{F}_p^d \to U_i$ for $i \in [3]$ such that $\on{rank} (\phi_1 + \phi_2 - \phi_3 - \phi_4) \leq 20 \log_p K$.\end{lemma}

\begin{proof}By assumptions we have
\[|U_1 \cap (U_2 + U_3)| = \frac{|U_2 + U_3| |U_1|}{|U_1 + U_2 + U_3|} \leq  \frac{|U_2|| U_3| |U_1|}{|U_1 + U_2 + U_3|} \leq K.\]
Let $V_1 \leq U_1$ be an arbitrary subspace such that $U_1 = \Big(U_1 \cap (U_2 + U_3 )\Big) \oplus V_1$. Thus $|V_1| \geq |U_1|/K$. Similarly, we may find subspaces $V_2 \leq U_2$ and $V_3 \leq U_3$ such that $U_2 = \Big(U_2 \cap (U_3 + U_1)\Big) \oplus V_2$ and $U_3 = \Big(U_3 \cap (U_1 + U_2)\Big) \oplus V_3$ which also satisfy $|V_2| \geq |U_2|/K$ and $|V_3| \geq |U_3|/K$. We claim that $V_1 + V_2 + V_3$ is actually a direct sum. To see this, let $x \in V_1 \cap (V_2 + V_3)$ be arbitrary. Since $V_1 \cap (V_2 + V_3) \subseteq U_1 \cap (U_2 + U_3)$, we have that $x \in V_1 \cap (U_1 \cap (U_2 + U_3))$. But this intersection is $0$, proving that $x = 0$.\\
\indent Write $S =  V_1 + V_2 + V_3$. Since this is a direct sum, there exist linear maps $\pi_i \colon S \to V_i$ for $i \in [3]$ such that $s = \pi_1(s) + \pi_2(s) - \pi_3(s)$ for all $s \in S$.\\
\indent Going back to assumptions, we see that
\[| U_4 \cap S| = \frac{|U_4| |S|}{| U_4 + S|} \geq \frac{K^{-3} p^{4d}}{| U_4 + U_1 + U_2 + U_3|} \geq K^{-4}p^d.\]

Recall that we are given a linear isomorphism $\phi_4 \colon \mathbb{F}_p^d \to U_4$. For $i \in [3]$, let $\phi'_i \colon \mathbb{F}_p^d \to U_i$ be linear map defined as follows. We first define $\phi'_i(x) = \pi_i(\phi_4(x))$ for all $x \in \phi_4^{-1}(U_4 \cap S)$, and then extend to whole $\mathbb{F}_p^d$ arbitrarily. We claim that $\on{rank} \phi'_i \geq d - 5\log_p K$. To see this, we need to estimate $|\on{ker} \pi_i \cap (U_4 \cap S)|$. By definition of $\pi_i$, we have $\on{ker} \pi_i = V_j + V_k$ for $\{j,k\} = [3] \setminus \{i\}$. Thus 
\[|(V_j + V_k) \cap (U_4 \cap S)| \leq |(U_j + U_k) \cap U_4| = \frac{|U_j + U_k| |U_4|}{|U_j + U_k + U_4|} \leq \frac{|U_j| |U_k| |U_4|}{|U_j + U_k + U_4|} \leq K,\]
from which we deduce
\begin{align*}\on{rank} \phi'_i \geq \on{rank} \phi'_i|_{\phi_4^{-1}(U_4 \cap S)} &\geq \dim (U_4 \cap S) - \dim \on{ker} \phi'_i|_{\phi_4^{-1}(U_4 \cap S)}\\
&\geq d - 4 \log_p K - \dim (\on{ker} \pi_i \cap (U_4 \cap S)) \geq d - 5\log_p K.\end{align*}
Using Lemma~\ref{nearisomlemma}, we may find a linear isomorphism $\phi_i \colon \mathbb{F}_p^d \to U_i$ such that $\on{rank}(\phi_i - \phi'_i) \leq 5\log_p K$. From the definition, we see that $\phi'_1 + \phi'_2 - \phi'_3 - \phi_4$ vanishes on $\phi_4^{-1}(U_4 \cap S)$. Thus $\on{rank}(\phi'_1 + \phi'_2 - \phi'_3 - \phi_4) \leq 4 \log_p K$, and the claim follows.\end{proof}

We also need a related uniqueness result.

\begin{lemma}\label{addquadshomsuniq}Let $W, U_1, U_2, V_1, V_2 \leq G$ be subspaces of dimension $d$. Let $K \geq 1$ be a parameter such that 
\[|W \cap (U_1 + U_2 + V_1 + V_2)| \leq K.\]
Suppose that $\phi_i \colon \mathbb{F}_p^d \to U_i$, $i \in [2]$, $\psi_i \colon \mathbb{F}_p^d \to V_i$, $i \in [2]$ and $\theta \colon \mathbb{F}_p^d \to W$ are linear maps such that 
\[\on{rank} \Big(\phi_1 + \phi_2 + \psi_1 + \psi_2 + \theta\Big) \leq r.\]
Then $\on{rank} \theta \leq r + \log_p K$. 
\end{lemma}

\begin{proof}Let $I = \on{im}\Big(\phi_1 + \phi_2 + \psi_1 + \psi_2 + \theta\Big)$, which is a vector space of dimension at most $r$. Let $J = \on{im} \theta$. Then we have $J \subseteq (I + U_1 + U_2 + V_1 + V_2) \cap W$. But
\[|(I + U_1 + U_2 + V_1 + V_2) \cap W| \leq |I| |U_1 + U_2 + V_1 + V_2) \cap W| \leq p^r K,\]
from which the claim follows.\end{proof}

We are now ready to prove Theorem~\ref{approximatelinearsyssub}.

\begin{proof}[Proof of Theorem~\ref{approximatelinearsyssub}] Let us begin the proof by showing the following claim.

\begin{claim}Suppose $\eta \leq 2^{-31} c^3$. For at least $\frac{c}{2}|G|^6$ of 10-tuples $(a, b_1, b_2, b_3, x_2, x_3, y_1, y_3, z_1, z_2) \in A^{10}$ we have 
\begin{itemize}
\item $a = b_1 + b_2 - b_3,\,\, x_3 = x_2 - b_2 + b_3, \,\, y_1 = y_3 + b_1 - b_3,\,\, z_2 = b_1 + b_2 - z_1$, and
\item each of 7 subspaces
\begin{align*}W_{b_1} \cap W_{b_2} \cap &W_{b_3} \cap W_a,\,\, W_{b_1} \cap W_{x_2} \cap W_{x_3} \cap W_a,\,\, W_{b_2} \cap W_{x_3} \cap W_{x_2} \cap W_{b_3},\,\, W_{y_1} \cap W_{b_2} \cap W_{y_3} \cap W_a\\
&W_{b_1} \cap W_{y_3} \cap W_{y_1} \cap W_{b_3},\,\,W_{z_1} \cap W_{z_2} \cap W_{b_3} \cap W_a,\,\, W_{b_1} \cap W_{b_2} \cap W_{z_1} \cap W_{z_2}\end{align*}
has size at least $K^{-3} p^{-3d}|G|$.
\end{itemize}\end{claim}

\begin{proof}Let $\mathcal{T}$ be the set of all 10-tuples $(a, b_1, b_2, b_3, x_2, x_3, y_1, y_3, z_1, z_2) \in A^{10}$ that satisfy~\eqref{10subspacesConditionEqualities} and~\eqref{10subspacesCondition}. Thus, $|\mathcal{T}| \geq c|G|^6$. We now show that at most $\frac{c}{40}|G|^6$ 10-tuples in $\mathcal{T}$ have $|W_{b_1} \cap W_{b_2} \cap W_{b_3} \cap W_a| < K^{-1}p^{-3d}|G|$. Same argument will apply to other 6 subspaces and the claim will follow. (Note that the indices of each of the 7 subspaces in the claim form an additive quadruple, with any three indices behaving independently; these are the only properties that we shall use in the proof.)\\
\indent Let $\mathcal{B}$ be the set of all $(b_1, b_2, b_3) \in A^3$ such that the number of 10-tuples in $\mathcal{T}$ with these 3 elements is at least $\frac{c}{1000}|G|^3$. The number of 10-tuples in $\mathcal{T}$ whose $(b_1, b_2, b_3)$ belongs to $\mathcal{B}$ is therefore at least $\frac{999c}{1000}|G|^6$. Let $\mathcal{B}'$ be the subset $\mathcal{B}$ consisting of those triples $(b_1, b_2, b_3)$ for whom the number of $6$-tuples $(x,y,z,x',y',z') \in A^6$ with 
\[|W_{b_1} \cap W_{b_2} \cap W_{b_3} \cap W_x \cap W_y \cap W_z \cap W_{x'} \cap W_{y'} \cap W_{z'}| \leq K p^{-9d} |G|\]
is at most $\frac{c^2}{2\cdot 10^6}|G|^6$. Thus the number of 10-tuples in $\mathcal{T}$ whose $(b_1, b_2, b_3)$ belongs to $\mathcal{B}'$ is therefore at least $\frac{999c}{1000}|G|^6 - \eta 2 \cdot 10^6 c^{-2}|G|^6 \geq \frac{998c}{1000}|G|^6$, provided $\eta \leq c^3 / (2 \cdot 10^9)$.\\
\indent Pick any 10-tuple $(a, b_1, b_2, b_3, x_2, x_3, y_1, y_3, z_1, z_2)$ such that $(b_1, b_2, b_3) \in \mathcal{B}'$. Then $a = b_1 + b_2 - b_3$ and we have at least $\frac{c^2}{10^6}|G|^6$ of 6-tuples $(x,y,z,x',y',z')$ such that
\[|W_{a} \cap W_{b_1} \cap W_{b_2} \cap W_{b_3} \cap W_x \cap W_y \cap W_z|, |W_{a} \cap W_{b_1} \cap W_{b_2} \cap W_{b_3} \cap W_{x'} \cap W_{y'} \cap W_{z'}| \geq K^{-1} p^{-6d} |G|.\]
Standard properties of subspaces then imply
\begin{align*}&K^{-2}p^{-12d} |G|^2 \leq |W_{a} \cap W_{b_1} \cap W_{b_2} \cap W_{b_3} \cap W_x \cap W_y \cap W_z| \cdot |W_{a} \cap W_{b_1} \cap W_{b_2} \cap W_{b_3} \cap W_{x'} \cap W_{y'} \cap W_{z'}|\\
&\hspace{2cm} = \Big|\Big(W_{a} \cap W_{b_1} \cap W_{b_2} \cap W_{b_3} \cap W_x \cap W_y \cap W_z\Big) \,\, + \,\,\Big(W_{a} \cap W_{b_1} \cap W_{b_2} \cap W_{b_3} \cap W_{x'} \cap W_{y'} \cap W_{z'}\Big)\Big| \\
&\hspace{4cm} |W_{a} \cap W_{b_1} \cap W_{b_2} \cap W_{b_3} \cap W_x \cap W_y \cap W_z \cap W_{x'} \cap W_{y'} \cap W_{z'}|\\
&\hspace{2cm} \leq |W_{a} \cap W_{b_1} \cap W_{b_2} \cap W_{b_3}| \cdot |W_{b_1} \cap W_{b_2} \cap W_{b_3} \cap W_x \cap W_y \cap W_z \cap W_{x'} \cap W_{y'} \cap W_{z'}|.\end{align*}

But there exists a choice of $(x,y,z,x',y',z')$ for which
\[|W_{b_1} \cap W_{b_2} \cap W_{b_3} \cap W_x \cap W_y \cap W_z \cap W_{x'} \cap W_{y'} \cap W_{z'}| \leq Kp^{-9d} |G|\]
so we get $|W_{a} \cap W_{b_1} \cap W_{b_2} \cap W_{b_3}| \geq K^{-3}p^{-3d}|G|$, as required.\end{proof}

Write $U_a = W_a^\perp$ for each $a \in A$. The assumptions on $W_a$ and the claim above imply that
\begin{equation}\label{approximatelinearsyssubQRcond3}|U_{a_1} + U_{a_2} + \dots + U_{a_r}| \geq K^{-1} p^{rd} \end{equation}
holds for all but at most $\eta |G|^r$ $r$-tuples $(a_1, a_2, \dots, a_r) \in A^r$ for $r \in [9]$, and that for at least $\frac{c}{2} |G|^6$ 10-tuples $(a, b_1, b_2, b_3, x_2, x_3, y_1, y_3, z_1, z_2) \in A^{10}$ we have 
\begin{itemize}
\item $a = b_1 + b_2 - b_3,\,\, x_3 = x_2 - b_2 + b_3, \,\, y_1 = y_3 + b_1 - b_3,\,\, z_2 = b_1 + b_2 - z_1$, and
\item each of 7 subspaces
\begin{align}U_{b_1} + U_{b_2} + &U_{b_3} + U_a,\,\, U_{b_1} + U_{x_2} + U_{x_3} + U_a,\,\, U_{b_2} + U_{x_3} + U_{x_2} + U_{b_3},\,\, U_{y_1} + U_{b_2} + U_{y_3} + U_a\nonumber\\
&U_{b_1} + U_{y_3} + U_{y_1} + U_{b_3},\,\,U_{z_1} + U_{z_2} + U_{b_3} + U_a,\,\, U_{b_1} + U_{b_2} + U_{z_1} + U_{z_2}\label{sevenSubspacesCondition}\end{align}
\end{itemize}
has size at most $K^3 p^{3d}$.

Our aim is to use Lemma~\ref{addquadshomsexist} to define linear isomorphisms between $\mathbb{F}_p^d$ and $U_a$. To that end, we say that an additive quadruple $(x_1, x_2, x_3, x_4)$ (where $x_1 + x_2 = x_3 + x_4$) is \emph{good} if we have $K^{-1} p^{3d} \leq |U_{x_{j_1}} + U_{x_{j_2}} + U_{x_{j_3}}|$ for any distinct indices $j_1, j_2, j_3 \in [4]$ and $|U_{x_1} + U_{x_2} + U_{x_3} + U_{x_4}| \leq K^3p^{3d}$. Notice that the associated subspaces to elements of any good additive quadruple satisfy conditions of Lemma~\ref{addquadshomsexist}.\\
\indent Our next claim shows that we may find many $10$-tuples $(a, b_1, b_2, b_3, x_2, x_3, y_1, y_3, z_1, z_2) \in A^{10}$ which satisfy stronger conditions than guaranteed by~\eqref{sevenSubspacesCondition}, and these stronger conditions will allow us to apply Lemma~\ref{addquadshomsuniq}.

\begin{claim}\label{10tupleslongclaim}Suppose that $\eta \leq \frac{c}{100}$. For at least $\frac{c}{20} |G|^6$ 10-tuples $(a, b_1, b_2, b_3, x_2, x_3, y_1, y_3, z_1, z_2) \in A^{10}$ we have 
\begin{itemize}
\item[\textbf{(i)}] $a = b_1 + b_2 - b_3,\,\, x_3 = x_2 - b_2 + b_3, \,\, y_1 = y_3 + b_1 - b_3,\,\, z_2 = b_1 + b_2 - z_1$,
\item[\textbf{(ii)}] each of 7 subspaces
\begin{align*}U_{b_1} + U_{b_2} + &U_{b_3} + U_a,\,\, U_{b_1} + U_{x_2} + U_{x_3} + U_a,\,\, U_{b_2} + U_{x_3} + U_{x_2} + U_{b_3},\,\, U_{y_1} + U_{b_2} + U_{y_3} + U_a\nonumber\\
&U_{b_1} + U_{y_3} + U_{y_1} + U_{b_3},\,\,U_{z_1} + U_{z_2} + U_{b_3} + U_a,\,\, U_{b_1} + U_{b_2} + U_{z_1} + U_{z_2}\end{align*}
has size at most $K^3 p^{3d}$,
\item[\textbf{(iii)}] each of 16 subspaces obtained by taking sum of 3 subspaces inside any of the 4 quadruples $\{U_{b_1}, U_{b_2}, U_{b_3}, U_a\}, \{U_{b_1}, U_{x_2}, U_{x_3}, U_a\}, \{U_{y_1}, U_{b_2}, U_{y_3}, U_a\}$ and $\{U_{z_1}, U_{z_2} , U_{b_3} , U_a\}$ has size at least $K^{-1} p^{3d}$,
\item[\textbf{(iv)}] each of 3 subspaces
\begin{align*}U_{b_1} \cap (U_{b_2} + U_{x_3} + U_{x_2} + U_{b_3}),\,\, U_{b_2} \cap (U_{b_1} + U_{y_3} + U_{y_1} + U_{b_3}),\,\, U_{b_3} \cap (U_{b_1} + U_{b_2} + U_{z_1} + U_{z_2})\end{align*}
has size at most $K^4$.
\end{itemize}
\end{claim}

\begin{proof} Combining conditions~\eqref{approximatelinearsyssubQRcond3} for $r \in \{3,4\}$ and~\eqref{sevenSubspacesCondition} we get at least $(2^{-1}c - 19\eta) |G|^6$ 10-tuples $(a, b_1, b_2, b_3, x_2, x_3,$ $y_1,$ $y_3,$ $z_1, z_2) \in A^{10}$ such that conditions \textbf{(i)} and \textbf{(ii)} in the conclusion of the claim hold and additionally we have each of 16 subspaces obtained by taking sum of 3 subspaces inside any of the 4 quadruples $\{U_{b_1}, U_{b_2}, U_{b_3}, U_a\},$  $\{U_{b_1}, U_{x_2}, U_{x_3}, U_a\},$ $\{U_{y_1}, U_{b_2}, U_{y_3}, U_a\}$ and $\{U_{z_1}, U_{z_2} , U_{b_3} , U_a\}$ has size at least $K^{-1} p^{3d}$ and each of 3 subspaces
\[U_{b_1} + U_{b_2} + U_{b_3} + U_{x_2},\,\,U_{b_1} + U_{b_2} + U_{b_3} + U_{y_3},\,\,U_{b_1} + U_{b_2} + U_{b_3} + U_{z_1}\]
has size at least $K^{-1} p^{4d}$. We now show that each such 10-tuple has the properties described in the claim.\\
\indent Take any such a 10-tuple $(a, b_1, b_2, b_3, x_2, x_3, y_1, y_3, z_1, z_2) \in A^{10}$. By the way we choose 10-tuples we see that property \textbf{(iii)} holds. It remains to prove property \textbf{(iv)}. We have
\[|U_{b_1} \cap (U_{b_2} + U_{x_3} + U_{x_2} + U_{b_3})| = \frac{|U_{b_1}| |U_{b_2} + U_{x_3} + U_{x_2} + U_{b_3}|}{|U_{b_1} + U_{b_2} + U_{x_3} + U_{x_2} + U_{b_3}|} \leq \frac{p^d \cdot K^3p^{3d}}{|U_{b_1} + U_{b_2} + U_{x_2} + U_{b_3}|} \leq K^4.\]
Similar arguments prove the other two bounds.\end{proof}

\indent Let $a \in A$ be an element that we shall specify later and let $\theta \colon \mathbb{F}_p^d \to U_a$ be a linear isomorphism. For each index $i \in [3]$, let $A_i \subseteq A$ be the set of all elements $x \in A$ that appear at $i$\tss{th} position in a good additive quadruple in which $a$ appears as the last element. For each such element $x \in A_i$, we pick a random good additive quadruple $(y_1, y_2, y_3, a)$ such that $y_i = x$, uniformly among all such quadruples. We do this independently for all $x$. Once we have chosen such quadruple $(y_1, y_2, y_3, a)$ we apply Lemma~\ref{addquadshomsexist} with chosen linear isomorphism $\theta \colon \mathbb{F}_p^d \to U_a$ to obtain linear isomorphisms $\psi_j \colon \mathbb{F}_p^d \to U_{y_j}$ such that $\on{rank} \Big(\psi_1 + \psi_2 - \psi_3 - \theta\Big) \leq 60 \log_p K$. Finally, define $\phi^{i}_x = \psi_i$.\\
\indent The crucial claim is that, if we choose $a$ suitably, the system of maps $(\phi^1_x)_{x \in A^1}$ exhibits a considerable amount of additive structure.

\begin{claim}There exists an element $a \in A$ such that in the above procedure we get at least $2^{-28}c^{4} |G|^2$ additive quadruples $(b_1, b_2, b_3, a) \in A^1 \times A^2 \times A^3 \times A$ such that $\on{rank}\Big(\phi^1_{b_1} + \phi^2_{b_2} - \phi^3_{b_3} - \theta\Big) \leq 500 \log_p K$.\end{claim}

\begin{proof}Let $\mathcal{T}$ be the set of all 10-tuples satisfying properties in Claim~\ref{10tupleslongclaim}. Let $a \in A$ be such that
\begin{align*}\mathcal{T}' = \Big\{(b_1, b_2, b_3, &x_2, x_3, y_1, y_3, z_1, z_2) \in A^9 \colon\\
&(a, b_1, b_2, b_3, x_2, x_3, y_1, y_3, z_1, z_2) \in \mathcal{T}\Big\}\end{align*}
has size at least $\frac{c}{40} |G|^5$. Now take any $(b_1, b_2, b_3) \in A^3$ such that $b_1 + b_2 = b_3 + a$ and such that the number of 6-tuples $(x_2, x_3, y_1, y_3, z_1, z_2) \in A^6$ such that $(b_1, b_2, b_3, x_2,$ $x_3, y_1, y_3,$ $z_1, z_2) \in \mathcal{T}'$ is at least $\frac{c}{80}|G|^3$. By averaging, the set of such triples $\mathcal{B}$ has size at least $\frac{c}{80}|G|^2$. Let $\mathcal{S}_{b_1, b_2, b_3}$ be the set of the 6-tuples we considered above for $(b_1, b_2, b_3)$.  Let us observe for any $(b_1, b_2, b_3, x_2,$ $x_3, y_1, y_3,$ $z_1, z_2) \in \mathcal{T}'$ that properties \textbf{(ii)} and \textbf{(iii)} of Claim~\ref{10tupleslongclaim} imply that quadruples $(b_1, b_2, b_3, a)$, $(b_1, x_2, x_3, a)$, $(y_1, b_2, y_3, a)$ and $(z_1, z_2, b_3, a)$ are good and thus $b_1, y_1, z_1 \in A_1,$ $b_2, x_2, z_2 \in A_2$ and $b_3, x_3, y_3 \in A_3$. We now show that for every $(b_1, b_2, b_3) \in \mathcal{B}$ the event $\on{rank}\Big(\phi^1_{b_1} + \phi^2_{b_2} - \phi^3_{b_3} - \theta\Big) \leq 150\log_p K$ occurs with probability at least $2^{-18}c^3$.\\
\indent To see that, fix $(b_1, b_2, b_3) \in \mathcal{B}$. Apply Lemma~\ref{addquadshomsexist} to quadruple of subspaces $(U_{b_1}, U_{b_2}, U_{b_3}, U_a)$ and isomorphism $\theta \colon \mathbb{F}_p^d \to U_a$. We thus get linear isomorphisms $\theta_i \colon \mathbb{F}_p^d \to U_{b_i}$ for $i \in [3]$ such that $\on{rank} \Big(\theta_1 + \theta_2 - \theta_3 - \theta\Big) \leq 60 \log_p K$. Since $|\mathcal{S}_{b_1, b_2, b_3}| \geq \frac{c}{80}|G|^3$, we see in particular that there at least $\frac{c}{80}|G|$ of choices of $x_2 \in A_2$ such that for $x_3 = b_1 + x_2 - a$ we have that additive quadruple $(b_1, x_2, x_3, a)$ is good and $|U_{b_1} \cap (U_{b_2} + U_{x_3} + U_{x_2} + U_{b_3})| \leq K^4$ (using property \textbf{(iv)} in the conclusion of Claim~\ref{10tupleslongclaim}). If it happens that  $(b_1, x_2, x_3, a)$ is used to define $\phi^1_{b_1}$, then we get also linear isomorphisms $\rho_2 \colon \mathbb{F}_p^d \to U_{x_2}$ and $\rho_3 \colon \mathbb{F}_p^d \to U_{x_3}$ such that $\on{rank}\Big(\phi^1_{b_1} + \rho_2 - \rho_3 - \theta\Big) \leq 60\log_p K$. But then
\[\on{rank}\Big(\Big(\phi^1_{b_1} - \theta_1\Big)+ \rho_2 - \rho_3  - \theta_2 + \theta_3\Big) \leq 120\log_p K\]
so by Lemma~\ref{addquadshomsuniq} we get $\on{rank}(\phi^1_{b_1} - \theta_1) \leq 124 \log_p K$. In particular, this happens with probability at least $\frac{c}{80}$.\\
\indent Similarly, events $\on{rank}(\phi^2_{b_2} - \theta_2) \leq 124 \log_p K$ and $\on{rank}(\phi^3_{b_3} - \theta_3) \leq 124 \log_p K$ happen with probability at least $\frac{c}{80}$ each, and all three events are independent, so we get 
\[\on{rank}\Big(\phi^1_{b_1} + \phi^2_{b_2} - \phi^3_{b_3} - \theta\Big) \leq \on{rank} \Big(\theta_1 + \theta_2 - \theta_3 - \theta\Big) + \on{rank}(\phi^1_{b_1} - \theta_1) + \on{rank}(\phi^2_{b_2} - \theta_2) + \on{rank}(\phi^3_{b_3} - \theta_3) \leq 432 \log_p K\]
with probability at least $2^{-21}c^3$.\\
\indent Thus, the expected number of triples $(b_1, b_2, b_3) \in B$ for which $\on{rank}\Big(\phi^1_{b_1} + \phi^2_{b_2} - \phi^3_{b_3} - \theta\Big) \leq 500 \log_p K$ holds is at least $2^{-28}c^4|G|^2$, proving the claim.\end{proof}

We say that an additive quadruple $(b_1, b_2, b_3, a) \in A^1 \times A^2 \times A^3 \times A$ is \emph{very good} if $\on{rank}\Big(\phi^1_{b_1} + \phi^2_{b_2} - \phi^3_{b_3} - \theta\Big) \leq 500 \log_p K$. By the claim above we have $c'|G|^2$ very good additive quadruples, where $c' = 2^{-28}c^4$. Consider the bipartite graph whose vertex classes $\mathcal{C}_1$ and $\mathcal{C}_2$ are copies of $G$ and where we put edge between $b_2 \in \mathcal{C}_1$ and $b_3\in \mathcal{C}_2$ if $(b_3 + a - b_2, b_2, b_3, a)$ is very good. This is a graph of density $c'$ and hence it contains at least ${c'}^4|G|^4$ ordered 4-cycles. \\
Let $(b_2, b_3, b'_2, b'_3)$ be an ordered 4-cycle in the above graph. Let $x_1 = b_3 + a - b_2$, $x_2 = b_3 + a - b'_2$, $x_3 = b'_3 + a - b'_2$ and $x_4 = b'_3 + a - b_2$, so additive quadruples 
\[(x_1, b_2, b_3, a),\,\,(x_2, b'_2, b_3, a),\,\,(x_3, b'_2, b'_3, a),\,\,(x_4, b_2, b'_3, a)\]
are very good. Observe also that $x_1 + x_3 = x_2 + x_4$ and that
\begin{align*}\phi^1_{x_1} + \phi^1_{x_3} - \phi^1_{x_2} - \phi^1_{x_4} = \Big(\phi^1_{x_1} + \phi^2_{b_2} - \phi^3_{b_3} - \theta\Big) + \Big(\phi^1_{x_3} + &\phi^2_{b'_2} - \phi^3_{b'_3} - \theta\Big) \\
- &\Big(\phi^1_{x_2} + \phi^2_{b'_2} - \phi^3_{b_3} - \theta\Big) - \Big(\phi^1_{x_4} + \phi^2_{b_2} - \phi^3_{b'_3} - \theta\Big),\end{align*} 
so $\on{rank}\Big(\phi^1_{x_1} + \phi^1_{x_3} - \phi^1_{x_2} - \phi^1_{x_4}\Big) \leq 2000 \log_p K$. Thus, by double-counting, there are at least ${c'}^4|G|^3$ additive quadruples $(x_1, x_3, x_2, x_4) \in (A^1)^4$ with $\on{rank}\Big(\phi^1_{x_1} + \phi^1_{x_3} - \phi^1_{x_2} - \phi^1_{x_4}\Big) \leq 2000 \log_p K$, as each such additive quadruple can arise from at most $|G|$ 4-cycles considered above.\\
\indent By Corollary~\ref{friemanforlinearhom} there exists a map $\Phi \colon G \times \mathbb{F}_p^d \to G$, which is affine in the first coordinate and linear in the second, such that $\on{rank}(\Phi(a, \cdot) - \phi^1_a) \leq R$ for at least $c_1|G|$ elements $a \in A$, where $c_1 \geq \exp\Big(-\exp\Big((\log {c'}^{-1} + \log_p K)^{O(1)}\Big)\Big)$ and $R \leq \exp\Big((\log {c'}^{-1} + \log_p K)^{O(1)}\Big)$. Let $B$ be the set of such $a$.\\
\indent Finally, we prove that $\Phi$ is highly quasirandom. Let $r = \lceil 2R + \log_p K \rceil$ and $\alpha = (2\eta c_1^{-2})^{1/2r}$. We claim there are no $r$ linearly independent elements $\lambda_1, \dots, \lambda_r \in \mathbb{F}_p^d$ such that for each $i \in [r]$ the bilinear form $(x,y) \in G^2 \to \Phi(x, \lambda_i) \cdot y$ has bias at least $\alpha$. Suppose the contrary. This implies that for each $i \in [r]$, there are at least $\alpha |G|$ elements $x$ for which $\Phi(x, \lambda_i) = 0$. Intersecting all these sets, we get a subspace $S \leq G$ of size $|S| \geq \alpha^r |G|$ such that for each $x \in S$ we have $\Phi(x, \lambda_i) = 0$ for all $i \in [r]$. Averaging over cosets of $S$, we may find $t$ such that $|B \cap (t +S)| \geq c_1|S|$.\\
\indent Whenever $a \in B \cap (t +S)$, the condition $\on{rank}(\Phi(a, \cdot) - \phi^1_a) \leq R$ implies that $\Phi(a, \lambda) = \phi^1_a(\lambda)$ holds for at least $p^{d-R}$ elements $\lambda \in \mathbb{F}_p^d$. Hence, $|\on{Im} \Phi(a, \cdot) \cap \on{Im} \phi^1_a| \geq p^{d-R}$, so recalling that $U_a = \on{Im} \phi^1_a$, we may find a subspace $T_a$ of dimension at most $R$ such that $U_a \subseteq T_a + \on{Im} \Phi(a, \cdot)$. In particular, for any two elements $a_1, a_2 \in B \cap (t +S)$ we have
\begin{align*}|U_{a_1} + U_{a_2}| \leq &|(T_{a_1} + \on{Im} \Phi(a_1, \cdot)) + (T_{a_2} + \on{Im} \Phi(a_2, \cdot))| \\
\leq &|T_{a_1}| |T_{a_2}| |\on{Im} \Phi(a_1, \cdot) + \on{Im} \Phi(a_2, \cdot)| \leq p^{2R} |\on{Im} \Phi(a_1, \cdot) + \on{Im} \Phi(a_2, \cdot)|.\end{align*}
However, since $a_1, a_2 \in B \cap (t +S)$, whenever $i \in [r]$, we have $\Phi(a_1, \lambda_i) = \Phi(a_1 - t, \lambda_i) + \Phi(t, \lambda_i) = \Phi(t, \lambda_i) = \Phi(a_2, \lambda_i)$. Thus $|\on{Im} \Phi(a_1, \cdot) + \on{Im} \Phi(a_2, \cdot)| \leq p^{2d - r}$ and we get
\[|U_{a_1} + U_{a_2}| \leq p^{2R + 2d -r} < K^{-1} p^{2d}\]
by the choice of $r$. However, this is in contradiction with~\eqref{approximatelinearsyssubQRcond3}, since $\eta < c_1^2 \alpha^{2r}$ by the choice of $\alpha$.\end{proof}

\section{Structure of approximate quadratic varieties}

In the final part of the proof, we use the map $\Phi$ to deduce the structure of the set $V$. The last step is articulated as the following proposition.

\begin{proposition}\label{finalsteprop}There is an absolute constant $D \geq 1$ such that the following holds. Let $c, \delta, \varepsilon > 0$ and $d \in \mathbb{N}$ be such that $c \leq \delta p^{d} \leq c^{-1}$. Suppose that $\varepsilon \leq (2^{-1}c \delta)^D$. Let $V \subseteq G$ be a set of density $\delta$ such that $\|\id_V - \delta\|_{\mathsf{U}^2} \leq \varepsilon$. Suppose that we are also given a subset $A \subseteq G$ of size $|A| \geq c |G|$, a subspace $W_a \leq G$ for each $a \in A$ and a bilinear map $\beta \colon G \times G \to \mathbb{F}_p^d$ such that 
\begin{itemize}
\item[\textbf{(i)}] for each $\lambda \in \mathbb{F}_p^d \setminus \{0\}$ we have $\on{bias} \lambda \cdot \beta \leq \varepsilon$,
\item[\textbf{(ii)}] for each $a \in A$ we have $|W_a \cap \{b \in G \colon \beta(a,b) = 0\}| \geq c p^{-d} |G|$,
\item[\textbf{(iii)}] for each $a \in A$ and $b \in W_a$ we have ${\conv}^{(8)} \id_{V \cap V - a}(b) \geq c \delta^{15}$. 
\end{itemize}
Then there exists a quadratic variety $Q \subseteq G$ of size $|Q| \leq (2c^{-1})^D \delta |G|$ such that $|Q \cap V| \geq \exp\Big(-\log^D (2c^{-1})\Big)\delta |G|$. Moreover, $Q$ is defined as $\{x \in G \colon \gamma(x,x) - \psi(x) = \mu\}$ for a symmetric bilinear map $\gamma \colon G \times G \to \mathbb{F}_p^{\tilde{d}}$, an affine map $\psi \colon G \to \mathbb{F}_p^{\tilde{d}}$ and $\mu \in \mathbb{F}_p^{\tilde{d}}$, where $\on{bias} \lambda \cdot \gamma \leq \varepsilon$ for all $\lambda \not= 0$, for some $d - O(\log_p (2c^{-1})) \leq \tilde{d} \leq d$.
\end{proposition} 

\begin{proof}The proof consists of three steps. In the first one we show that $\beta$ is approximately symmetric, in the second we pass to an exactly symmetric bilinear map and in the final step we use the symmetry to find the desired quadratic variety.\\

\noindent\textbf{Approximate symmetry.} This step of the proof is based on the symmetry argument of Green and Tao~\cite{GreenTaoU3}. We may combine assumptions \textbf{(ii)} and \textbf{(iii)} to deduce that
\begin{align*}&\exx_{a,b \in G} \id(\beta(a,b) = 0) {\conv}^{(8)} \id_{V \cap V -a}(b) \geq \exx_{a,b \in G} \id_A(a) \id(\beta(a,b) = 0) {\conv}^{(8)} \id_{V \cap V -a}(b) \\
&\hspace{2cm} \geq \exx_{a,b \in G} \id_A(a) \id(\beta(a,b) = 0) \id_{W_a}(b) {\conv}^{(8)} \id_{V \cap V -a}(b)  \\
&\hspace{2cm} \geq \exx_{a\in G} \id_A(a) \exx_b \Big(\id(\beta(a,b) = 0) \id_{W_a}(b)\Big) {\conv}^{(8)} \id_{V \cap V -a}(b)  \geq c^{2} p^{-d} \delta^{15} \geq c^3 \delta^{16}.\end{align*}
Thus
\begin{align*}c^3 \delta^{16} \leq &\exx_{a,b \in G} \id(\beta(a,b) = 0) {\conv}^{(8)} \id_{V \cap V -a}(b)\\
= & p^{-d}\sum_{\lambda \in \mathbb{F}^d_p} \exx_{a,b \in G}\omega^{\lambda \cdot \beta(a,b)} \exx_{x_1, \dots, x_7} \id_V(x_1)\id_V(x_1 + a) \dots \id_V(x_7) \id_V(x_7 + a)\\
& \hspace{6cm}\id_V(x_1 - x_2 + \dots -x_6 + x_7 - b)\id_V(x_1 - x_2 + \dots -x_6 + x_7 - b + a).\end{align*}
Observe furthermore that ${\conv}^{(8)} \id_{V \cap V -a}(b) = {\conv}^{(8)} \id_{V \cap V + a}(b)$. Using this fact, we see
that
\begin{align*}\overline{\exx_{a,b \in G}\omega^{\lambda \cdot \beta(a,b)} {\conv}^{(8)} \id_{V \cap V -a}(b)} = &\exx_{a,b \in G}\omega^{- \lambda \cdot \beta(a,b)} {\conv}^{(8)} \id_{V \cap V -a}(b) = \exx_{a,b \in G}\omega^{\lambda \cdot \beta(-a,b)} {\conv}^{(8)} \id_{V \cap V -a}(b) \\
= &\exx_{a,b \in G}\omega^{\lambda \cdot \beta(a,b)} {\conv}^{(8)} \id_{V \cap V + a}(b) = \exx_{a,b \in G}\omega^{\lambda \cdot \beta(a,b)} {\conv}^{(8)} \id_{V \cap V -a}(b),\end{align*}
implying that $\ex_{a,b \in G}\omega^{\lambda \cdot \beta(a,b)} {\conv}^{(8)} \id_{V \cap V -a}(b)$ is a real number.\\

Since for each $\lambda \in \mathbb{F}^d_p$
\begin{align*}&\bigg|\exx_{a,b \in G}\omega^{\lambda \cdot \beta(a,b)} {\conv}^{(8)} \id_{V \cap V -a}(b)\bigg| = \bigg|\exx_{a,b}\omega^{\lambda \cdot \beta(a,b)} \exx_{x_1, \dots, x_7} \id_V(x_1)\id_V(x_1 + a) \dots \id_V(x_7) \id_V(x_7 + a)\\
&\hspace{7cm} \id_V(x_1 - x_2 + \dots -x_6 + x_7 - b)\id_V(x_1 - x_2 + \dots -x_6 + x_7 - b + a)\bigg|\\
\leq & \exx_{a,b} \exx_{x_1, \dots, x_7} \id_V(x_1)\id_V(x_1 + a) \dots \id_V(x_7) \id_V(x_7 + a)\\
&\hspace{7cm} \id_V(x_1 - x_2 + \dots -x_6 + x_7 - b)\id_V(x_1 - x_2 + \dots -x_6 + x_7 - b + a)\\
= & \exx_a |G|^{-8} |V \cap V - a|^8 \leq 2\delta^{16},\end{align*}
by Claim~\ref{varietyTranslatesprops1} and provided $\varepsilon \leq 2^{-100}\delta^{64}$, we obtain a set $\Lambda \subset \mathbb{F}_p^d$ of size $|\Lambda| \geq 2^{-2}c^3  p^d$ such that 
\begin{align*}\exx_{a,b}\omega^{\lambda \cdot \beta(a,b)} \exx_{x_1, \dots, x_7} &\id_V(x_1)\id_V(x_1 + a) \dots \id_V(x_7) \id_V(x_7 + a)\\
& \id_V(x_1 - x_2 + \dots -x_6 + x_7 - b)\id_V(x_1 - x_2 + \dots -x_6 + x_7 - b + a) \geq 2^{-1}c^3 \delta^{16}.\end{align*}
We now show that for such a $\lambda$ we have that $\lambda \cdot (\beta - \beta \circ (1\,\,2))$ is of low rank.\\

Introduce an auxiliary variable $t$, and make a change of variables $w_i = x_i - t$. Then
\begin{align*}\exx_{a,b}\omega^{\lambda \cdot \beta(a,b)} \exx_{t, w_{[7]}} &\id_V(w_1 + t)\id_V(w_1 + t + a) \dots \id_V(w_7 + t) \id_V(w_7 + t + a)\\
& \id_V(w_1 - w_2 + \dots -w_6 + w_7 + t - b)\id_V(w_1 - w_2 + \dots -w_6 + w_7 + t - b + a) \geq 2^{-1}c^3 \delta^{16}\end{align*}

Make a further change of variables $x = t, y = t + a, z = 2t + a - b$ instead of $t,a,b$. Then
\begin{align*}2^{-1}c^3 \delta^{16} \leq & \exx_{x,y}\omega^{\lambda \cdot \beta(y-x, x + y - z)} \exx_{z, w_{[7]}} \id_V(w_1 + x)\id_V(w_1 + y) \dots \id_V(w_7 + x) \id_V(w_7 + y)\\
&\hspace{3cm} \id_V(w_1 - w_2 + \dots -w_6 + w_7 + z-y)\id_V(w_1 - w_2 + \dots -w_6 + w_7 + z - x)\\
=&\exx_{x,y,z}\omega^{\lambda \cdot \beta(y,x) - \lambda \cdot \beta(x,y)} \omega^{\lambda \cdot \beta(x-y,z) - \lambda \cdot \beta(x,x) + \lambda \cdot \beta(y,y)} \exx_{w_{[7]}} \id_V(w_1 + x)\id_V(w_1 + y) \dots \id_V(w_7 + x) \id_V(w_7 + y)\\
&\hspace{3cm} \id_V(w_1 - w_2 + \dots -w_6 + w_7 + z-y)\id_V(w_1 - w_2 + \dots -w_6 + w_7 + z - x).\end{align*}

Write $\rho(x,y) = \lambda \cdot \beta(y,x) - \lambda \cdot \beta(x,y)$ and for $\pmb{w} = w_{[7]}$ let 
\[f_{\pmb{w}}(t) = \prod_{i \in [7]}\id_V(w_i + t)\]
and
\[g_{\pmb{w}}(t) = \id_V(w_1 - w_2 + \dots -w_6 + w_7 + t).\]
With the new notation we have
\[2^{-1}c^3 \delta^{16} \leq\exx_{\ssk{x,y,z\\\pmb{w}}} \omega^{\rho(x,y)} \Big(\omega^{\lambda \cdot \beta(x,z) - \lambda \cdot \beta(x,x)}f_{\pmb{w}}(x) g_{\pmb{w}}(z - x)\Big)\Big(\omega^{-\lambda \cdot \beta(y,z) + \lambda \cdot \beta(y,y)}f_{\pmb{w}}(y) g_{\pmb{w}}(z - y)\Big).\]

Note that $f_{\pmb{w}}(y) g_{\pmb{w}}(z - y)$ takes values 0 and 1 so $f_{\pmb{w}}(y) g_{\pmb{w}}(z - y) = f_{\pmb{w}}(y)^2 g_{\pmb{w}}(z - y)^2$. Applying Cauchy-Schwarz inequality we get 
\[2^{-2}c^6 \delta^{32} \leq \bigg(\exx_{\ssk{y,z\\\pmb{w}}} f_{\pmb{w}}(y) g_{\pmb{w}}(z - y)\Big|\exx_x \omega^{\rho(x,y)} \omega^{\lambda \cdot \beta(x,z) - \lambda \cdot \beta(x,x)}f_{\pmb{w}}(x) g_{\pmb{w}}(z - x)\Big|^2\bigg) \bigg(\exx_{\ssk{y,z\\\pmb{w}}} f_{\pmb{w}}(y) g_{\pmb{w}}(z - y)\bigg).\]
Observe that the second term in the product above equals $\delta^8$. Thus
\[2^{-2}c^6 \delta^{24} \leq \exx_{\ssk{y,z,x,x'\\\pmb{w}}} f_{\pmb{w}}(y) g_{\pmb{w}}(z - y)  \omega^{\rho(x -x',y)} \omega^{\lambda \cdot \beta(x - x',z) - \lambda \cdot \beta(x,x) + \lambda \cdot \beta(x',x')}f_{\pmb{w}}(x) g_{\pmb{w}}(z - x)f_{\pmb{w}}(x') g_{\pmb{w}}(z - x').\]
Another application of the Cauchy-Schwarz inequality gives
\begin{align*}2^{-4}c^{12} \delta^{48} \leq& \bigg(\exx_{\ssk{z,x,x'\\\pmb{w}}} f_{\pmb{w}}(x) g_{\pmb{w}}(z - x)f_{\pmb{w}}(x') g_{\pmb{w}}(z - x')\Big| \exx_y f_{\pmb{w}}(y) g_{\pmb{w}}(z - y)  \omega^{\rho(x -x',y)} \Big|^2\bigg) \\
&\hspace{2cm}\bigg(\exx_{\ssk{z,x,x'\\\pmb{w}}} f_{\pmb{w}}(x) g_{\pmb{w}}(z - x)f_{\pmb{w}}(x') g_{\pmb{w}}(z - x')\bigg).\end{align*}
Usual argument based on Lemma~\ref{u2controlBasic} allows to bound the second term in the product from above by $\delta^{16} + 16\varepsilon \leq 2\delta^{16}$. Thus
\begin{align*}&2^{-5}c^{12} \delta^{32} \leq \exx_{\ssk{z,x,x'\\y,y', \pmb{w}}} f_{\pmb{w}}(x) g_{\pmb{w}}(z - x)f_{\pmb{w}}(x') g_{\pmb{w}}(z - x') f_{\pmb{w}}(y) g_{\pmb{w}}(z - y) f_{\pmb{w}}(y') g_{\pmb{w}}(z - y')  \omega^{\rho(x -x',y - y')}\\
&\hspace{2cm}=\exx_{\ssk{z,x,x'\\y,y', \pmb{w}}} \Big(\prod_{i \in [7]}\id_V(w_i + x)\Big)\Big(\prod_{i \in [7]}\id_V(w_i + x')\Big)   \Big(\prod_{i \in [7]}\id_V(w_i + y)\Big) \Big(\prod_{i \in [7]}\id_V(w_i + y')\Big) \\
&\hspace{4cm}\id_V(w_1 - w_2 + \dots -w_6 + w_7 + z-x) \id_V(w_1 - w_2 + \dots -w_6 + w_7 + z - x') \\
&\hspace{4cm}\id_V(w_1 - w_2 + \dots -w_6 + w_7 + z - y) \id_V(w_1 - w_2 + \dots -w_6 + w_7 + z - y')  \omega^{\rho(x -x',y - y')}.\end{align*}

Again, usual argument based on Lemma~\ref{u2controlBasic} allows\footnote{Note that Lemma~\ref{u2controlBasic} is stated for function of a single variable, while we have additional term $\omega^{\rho(x -x',y - y')}$ in the expression above. However, the proof of Lemma~\ref{u2controlBasic} only uses the fact that some variables do not appear in the functions that we want to remove (and that those functions take values in $\mathbb{D}$) and variables $z,w_1, \dots, w_7$ do not appear in $\omega^{\rho(x -x',y - y')}$.} to replace all terms involving $\id_V$ by $\delta$, namely we use variable $z$ and one of $x, x', y, y'$ for each of the last 4 such terms, and then the two variables appearing in each of the remaining terms. Thus
\[2^{-6}c^{12} \delta^{32} \leq 2^{-5}c^{12} \delta^{32} - 32 \varepsilon \leq \exx_{\ssk{x,x'\\y,y'}} \delta^{32} \omega^{\rho(x -x',y - y')},\]
from which we see that $\on{bias} \rho \geq 2^{-6}c^{12}$, provided $\varepsilon \leq 2^{-12}c^{12}\delta^{32}$.\\

\noindent \textbf{Obtaining exact symmetry.} Let $R \colon \mathbb{F}_p^d \times G \times G \to \mathbb{F}_p$ be the trilinear form defined by $R(\lambda, x,y) = \lambda \cdot (\beta(x,y) - \beta(y,x))$. From the work above,
\[\on{bias} R \geq \frac{|\Lambda|}{p^d} \cdot 2^{-6}c^{12} \geq 2^{-12}c^{15}.\]
By the Theorem~\ref{invTheoremBiased} we obtain an integer $s \leq O(\log_p (2c^{-1}))$, linear forms $\theta_1, \dots, \theta''_s$ and bilinear forms $\rho_1, \dots, \rho''_s$ such that 
\[R(\lambda, x, y) = \sum_{i \in [s]} \theta_i(\lambda) \rho_i(x,y) + \sum_{i \in [s]} \theta'_i(x) \rho'_i(\lambda,y) + \sum_{i \in [s]} \theta''_i(y) \rho''_i(\lambda,x).\]
Define subspace $U = \{u \in G \colon \theta'(u) = 0, \theta''(u) = 0\}$, which has codimension at most $2s$ in $G$. Let $e_1, \dots, e_{\tilde{d}}$ be a basis of the subspace $\{\lambda \in \mathbb{F}_p^d \colon \theta(\lambda) = 0\}$, where $\tilde{d} \geq d - s$. Let $\tilde{\beta} \colon G \times G \to \mathbb{F}_p^{\tilde{d}}$ be a bilinear map given by $\tilde{\beta}_i(x,y) = e_i \cdot \beta(x,y)$. Then we have $\tilde{\beta}(u,v) = \tilde{\beta}(v,u)$ for all $u,v \in U$. We now want to replace $\beta$ by a map which is symmetric on the whole space $G$.\\

Going back to assumptions of the proposition, recall that we have a set $A$ of size $|A| \geq c |G|$, a subspace $W_a$ for each $a \in A$ such that 
\[|W_a \cap  \{b \in G \colon \beta(a,b) = 0\}| \geq c p^{-d}|G|\]
and for each $b \in W_a$ we have
\[{\conv}^{(8)} \id_{V \cap V -a}(b) \geq c \delta^{15}.\]

Since $W_a$ is a subspace for each $a \in A$, we conclude that (observe that $b$ ranges over $U$ in the expectation below)
\begin{align*}&\exx_{a \in G, b \in U} \id(\tilde{\beta}(a,b) = 0) {\conv}^{(8)} \id_{V \cap V -a}(b) \geq \exx_{a \in G} \id_A(a) \exx_{b \in G} \id(\beta(a,b) = 0) \id_U(b) {\conv}^{(8)} \id_{V \cap V -a}(b) \\
&\hspace{2cm}\geq c \delta^{15} \exx_{a \in G} \id_A(a) \exx_{b \in G} \id(\beta(a,b) = 0) \id_U(b) \id_{W_a}(b)\\
&\hspace{2cm}\geq c \delta^{15} \exx_{a \in G} \id_A(a)  \frac{|\{b \in G \colon \beta(a,b) = 0\} \cap W_a \cap U|}{|G|}\\
&\hspace{2cm}\geq c p^{-2s} \delta^{15} \exx_{a \in G} \id_A(a)  \frac{|\{b \in G \colon \beta(a,b) = 0\} \cap W_a|}{|G|}\\
&\hspace{2cm}\geq c^3 p^{-2s-d} \delta^{15} \geq c^4 p^{-2s} \delta^{16}.\end{align*}

Let $T \leq G$ be an arbitrary subspace with the property that $G = U \oplus T$ and let $c_1 = c^4 p^{-2s}$. Thus,
\[c_1 \delta^{16} \leq \exx_{t \in T, a, b \in U} \id(\tilde{\beta}(a + t,b) = 0) {\conv}^{(8)} \id_{V \cap V - a - t}(b).\]
By averaging, there exists $t_0 \in T$ such that
\begin{align*} c_1 \delta^{16} \leq &\exx_{a, b \in U} \id(\tilde{\beta}(a + t_0,b) = 0) {\conv}^{(8)} \id_{V \cap V - a - t_0}(b). \\
 = &\exx_{\lambda \in \mathbb{F}_p^{\tilde{d}}} \exx_{\ssk{a, b \in U\\x_1, \dots, x_7 \in G}} \omega^{\lambda \cdot \tilde{\beta}(a + t_0,b)} \id_V(x_1) \id_V(x_1 + a + t_0) \dots \id_V(x_7) \id_V(x_7 + a + t_0)\\
&\hspace{5cm}\id_V(x_1 - x_2 + \dots + x_7 - b) \id_V(x_1 -x_2 + \dots + x_7 - b + a + t_0).\end{align*}

By Cauchy-Schwarz inequality we get
\begin{align*}&c_1^2\delta^{32} \leq \bigg(\exx_{\ssk{b \in U, \lambda \in \mathbb{F}_p^{\tilde{d}}\\x_1, \dots, x_7 \in G}} \id_V(x_1) \dots \id_V(x_7) \id_V(x_1 - x_2 + \dots + x_7 - b)\bigg)\\
&\hspace{2cm}\bigg(\exx_{\ssk{b \in U, \lambda \in \mathbb{F}_p^{\tilde{d}}\\x_1, \dots, x_7 \in G}} \id_V(x_1) \dots \id_V(x_7) \id_V(x_1 - x_2 + \dots + x_7 - b) \\
&\hspace{4cm}\Big|\exx_{a \in U}\omega^{\lambda \cdot \tilde{\beta}(a + t_0, b)} \id_V(x_1 + a + t_0) \dots \id_V(x_7 + a + t_0) \id_V(x_1 -x_2 + \dots + x_7 - b + a + t_0)\Big|^2\bigg).\end{align*}

The first long sum above equals $\delta^8$ and hence
\begin{align*}&c_1^2\delta^{24} \leq \exx_{\ssk{a,b,u \in U, \lambda \in \mathbb{F}_p^{\tilde{d}}\\x_1, \dots, x_7 \in G}} \id_V(x_1) \dots \id_V(x_7) \id_V(x_1 - x_2 + \dots + x_7 - b)\omega^{\lambda \cdot \tilde{\beta}(u, b)}\\
&\hspace{2cm}\id_V(x_1 + a + t_0) \dots \id_V(x_7 + a + t_0) \id_V(x_1 -x_2 + \dots + x_7 - b + a + t_0)\\
&\hspace{2cm}\id_V(x_1 + a + t_0 + u) \dots \id_V(x_7 + a + t_0 + u) \id_V(x_1 -x_2 + \dots + x_7 - b + a + t_0 + u)\\
&\hspace{1cm}=\exx_{\ssk{a,b,u \in U, \lambda \in \mathbb{F}_p^{\tilde{d}}\\x_1, \dots, x_7 \in G}} \id_V(x_1 - a) \dots \id_V(x_7 - a) \id_V(x_1 - x_2 + \dots + x_7 - a - b)\omega^{\lambda \cdot \tilde{\beta}(u, b)}\\
&\hspace{2cm}\id_V(x_1 + t_0) \dots \id_V(x_7  + t_0) \id_V(x_1 -x_2 + \dots + x_7 - b  + t_0)\\
&\hspace{2cm}\id_V(x_1 + t_0 + u) \dots \id_V(x_7 + t_0 + u) \id_V(x_1 -x_2 + \dots + x_7 - b + t_0 + u),\end{align*}
where we made a change of variables by shifting $x_i$ by $a$.\\
\indent Applying Lemma~\ref{u2controlBasic},\footnote{As remarked before, the lemma still applies with the additional bilinear phase term $\omega^{\lambda \cdot \tilde{\beta}(u, b)}$.} we may replace the first 8 $\id_V$ terms by $\delta^8$ using variables $b,a$ for $\id_V(x_1 - x_2 + \dots + x_7 - a + b)$ and then $x_i,a$ for each $\id_V(x_i - a)$. Hence, provided $\varepsilon \leq 2^{-4}c_1^2\delta^{24}$,

\begin{align*}&2^{-1}c_1^2\delta^{24}\leq c_1^2\delta^{24} - 8\varepsilon \leq \delta^8 \exx_{\ssk{b,u \in U, \lambda \in \mathbb{F}_p^{\tilde{d}}\\x_1, \dots, x_7 \in G}}\omega^{\lambda \cdot \tilde{\beta}(u, b)} \id_V(x_1 + t_0) \dots \id_V(x_7 + t_0) \id_V(x_1 -x_2 + \dots + x_7 - b + t_0)\\
&\hspace{2cm}\id_V(x_1 + t_0 + u) \dots \id_V(x_7 + t_0 + u) \id_V(x_1 -x_2 + \dots + x_7 - b + t_0 + u)\\
&\hspace{1cm}=\delta^8 \exx_{\ssk{b,u \in U\\x_1, \dots, x_7 \in G}} \id(\tilde{\beta}(u, b) = 0) \id_{V-t_0}(x_1) \dots \id_{V-t_0}(x_7) \id_{V-t_0}(x_1 -x_2 + \dots + x_7 - b)\\
&\hspace{2cm}\id_{V-t_0}(x_1 + u) \dots \id_{V-t_0}(x_7 + u) \id_{V-t_0}(x_1 -x_2 + \dots + x_7 - b+ u).\end{align*}

Let us misuse the notation and write $V$ instead of $V - t_0$, which is fine as the quadratic variety that we obtain for $V-t_0$ can then be shifted by $t_0$ to get the desired conclusion. Recall that the subspace $U$ has codimension at most $2s$ in $G$. Extend $\tilde{\beta}|_{U \times U}$ arbitrarily to a symmetric bilinear map $\gamma \colon G \times G\to \mathbb{F}_p^{\tilde{d}}$. Then $\on{rank} \lambda \cdot \gamma \geq \on{rank} \lambda \cdot \tilde{\beta}$ for all $\lambda \not= 0$ and hence 
\[\on{bias} \lambda \cdot \gamma = p^{-\on{rank} \lambda \cdot \gamma } \leq p^{-\on{rank} \lambda \cdot \tilde{\beta} } = \on{bias} \lambda \cdot \tilde{\beta} \leq \varepsilon.\] 
Furthermore,
\begin{align*}2^{-1}p^{-4s}c_1^2\delta^{16} \leq &\exx_{\ssk{a,b \in G\\x_1, \dots, x_7 \in G}} \id(\gamma(a,b) = 0) \id_V(x_1) \dots \id_V(x_7) \id_V(x_1 -x_2 + \dots + x_7 - b)\\
&\hspace{2cm}\id_V(x_1 + a) \dots \id_V(x_7 + a) \id_V(x_1 -x_2 + \dots + x_7 - b + a).\end{align*}

\noindent\textbf{Finding quadratic variety.} Let $q(x) = \frac{1}{2}\gamma(x,x)$. Observe that for each $a, t \in G$ we have 
\[q(t+a) - q(t) = \frac{1}{2}\gamma(t+a, t+a) - \frac{1}{2}\gamma(t,t) = \gamma(a,t) + \frac{1}{2}\gamma(a,a) = \gamma(a,t) + q(a).\]

For $a \in G$ and $\mu \in \mathbb{F}_p^{\tilde{d}}$, let us define $\rho_a(\mu) = \ex_x \id_{V \cap V - a}(x) \id(\gamma(a,x) = \mu)$. Let $M_a = \max_{\mu \in \mathbb{F}_p^{\tilde{d}}} \rho_a(\mu)$. Writing $c_2 = 2^{-1}p^{-4s}c_1^2$, we have
\begin{align}c_2\delta^{16} \leq &\exx_{\ssk{a,b \in G\\x_1, \dots, x_7 \in G}} \id(\gamma(a,b) = 0) \id_{V \cap V-a} (x_1) \dots \id_{V \cap V-a}(x_7) \id_{V \cap V-a} (x_1 -x_2 + \dots + x_7 - b)\nonumber\\
= &\exx_{\ssk{a \in G\\x_1, \dots, x_8\in G}} \id(\gamma(a,x_1 - x_2 + \dots - x_8) = 0) \id_{V \cap V-a} (x_1) \dots \id_{V \cap V-a}(x_7) \id_{V \cap V-a} (x_8)\nonumber\\
= & \sum_{\mu_1, \dots, \mu_8 \in \mathbb{F}_p^{\tilde{d}}} \exx_{\ssk{a \in G\\x_1, \dots, x_8\in G}} \id(\mu_1 - \mu_2 + \dots -\mu_8 = 0) \prod_{i \in [8]} \Big( \id_{V \cap V - a}(x_i) \id(\gamma(a, x_i) = \mu_i)\Big)\nonumber\\
= & \sum_{\mu_1, \dots, \mu_7 \in \mathbb{F}_p^{\tilde{d}}} \exx_a \rho_a(\mu_1) \dots \rho_a(\mu_7) \rho_a(\mu_1 - \mu_2 + \dots + \mu_7)\nonumber\\
\leq &\exx_a \sum_{\mu_1, \dots, \mu_7 \in \mathbb{F}_p^{\tilde{d}}} \rho_a(\mu_1) \dots \rho_a(\mu_7) M_a .\label{rhomuineq}\end{align}

Note that 
\[\sum_{\mu \in \mathbb{F}_p^{\tilde{d}}} \rho_a(\mu) = \sum_{\mu \in \mathbb{F}_p^{\tilde{d}}} \exx_x \id_{V \cap V - a}(x) \id(\gamma(a,x) = \mu) = \exx_x \id_{V \cap V - a}(x) = \frac{|V \cap V -a|}{|G|}.\]

By Claim~\ref{varietyTranslatesprops1}, we see for all but at most $8\sqrt{\varepsilon}|G|$ of $a \in G$ we have that the sum above is at most $2\delta^2$. Let $P \subseteq G$ be the set of all $a \in G$ such that $M_a \geq 2^{-10}c_2 \delta^2$ and $|V \cap V -a| \leq 2\delta^2 |G|$. Inequality~\eqref{rhomuineq} gives 
\begin{align*}c_2\delta^{16} \leq 8\sqrt{\varepsilon} + 2^7 \delta^{14} \Big(2^{-10}c_2 \delta^2 + \exx_a \id_P(a) M_a\Big),\end{align*}
from which we obtain $\ex_a \id_P(a) M_a \geq 2^{-10}c_2 \delta^2$, provided $\varepsilon \leq 2^{-8}c_2^2 \delta^{32}$.\\

For each $a \in P$, let $\mu(a) \in \mathbb{F}_p^{\tilde{d}}$ be an argument for which $M_a$ is attained. Write $\tilde{\mu}(t) = \mu(t) + q(t)$. Then
\begin{align*}2^{-10}c_2 \delta^2 \leq &\exx_a \id_P(a) M_a = \exx_{a, x} \id_P(a) \id_{V \cap V - a}(x) \id(\gamma(a,x) = \mu(a)) \\
= &\exx_{a, x} \id_P(a) \id_{V}(x) \id_{V}(x + a) \id(\gamma(a,x) = \mu(a)) \\
= &\exx_{a, x} \id_P(a)  \id_{V}(x) \id_{V}(x + a) \id(q(a + x) - q(x) = \mu(a) + q(a)) \\
= &\exx_{x, y} \id_P(y-x) \id_{V}(x) \id_{V}(y) \id(q(y) - q(x) = \tilde{\mu}(y-x)).\end{align*}

Let $\Gamma$ be the bipartite graph whose vertex classes $\mathcal{C}_1$ and $\mathcal{C}_2$ are two copies of $V$, with $(x,y) \in \mathcal{C}_1 \times \mathcal{C}_2$ being an edge if $y-x \in P$ and $q(y) - q(x) = \tilde{\mu}(y-x)$. The inequality above shows that the density of $\Gamma$ is at least $2^{-10}c_2$. This means that $\Gamma$ has at least $2^{-40}c_2^4 |V|^4$ ordered 4-cycles. Each ordered 4-cycle $(x_1, y_1, x_2, y_2)$ gives rise to an additive quadruple in $P$ respected by $\tilde{\mu}$ as $(y_1 - x_1) + (y_2 - x_2) = (y_1 - x_2) + (y_2 - x_1)$ and
\begin{align*}&\tilde{\mu}(y_1-x_1) + \tilde{\mu}(y_2-x_2) - \tilde{\mu}(y_1-x_2) - \tilde{\mu}(y_2-x_1)\\
&\hspace{2cm} = \Big(q(y_1) - q(x_1)\Big) + \Big(q(y_2) - q(x_2)\Big) - \Big(q(y_1) - q(x_2)\Big)- \Big(q(y_2) - q(x_1)\Big) = 0.\end{align*}

By Claim~\ref{varietyTranslatesprops1} and provided $\varepsilon \leq \delta^{16}$, all but at most $16\sqrt{\varepsilon} |G|^3$ additive quadruple $(a, a', b, b')$ in $P$ can arise from at most $2\delta^4 |G|$ ordered 4-cycles, since this corresponds to finding $x$ such that $x, x + a, x + a - b, x + a - b + a' \in V$, i.e. $x \in V \cap (V - a) \cap (V- a + b) \cap (V - a + b - a')$. Hence, as long as $\varepsilon \leq 2^{-100} c_2^{8}\delta^{8}$, $\tilde{\mu}$ respects at least least $2^{-42}c_2^4 |G|^3$ additive quadruples in $P$.\\

We may apply Theorem~\ref{approxHommInv} to find an affine map $\psi \colon G \to \mathbb{F}_p^{\tilde{d}}$ and a parameter $c_3 \geq \exp\Big(-\log^{O(1)}(2c_2^{-1})\Big)$ such that $\psi(a) = \tilde{\mu}(a)$ holds for at least $c_3 |G|$ elements $a \in P$. Let $P'$ be the set of such elements. Hence
 \begin{align*}2^{-10} c_2 c_3\delta^2 \leq &\exx_{a} \id_{P'}(a) \rho_a(\mu(a)) = \exx_{a} \id_{P'}(a) \Big(\exx_x \id_{V \cap V - a}(x) \id(\gamma(a,x) = \mu(a))\Big)\\
 = &\exx_{a} \id_{P'}(a) \Big(\exx_x \id_{V \cap V - a}(x) \id(\gamma(a,x) = \tilde{\mu}(a) - q(a))\Big)\\
= &\exx_{a} \id_{P'}(a) \Big(\exx_x \id_{V \cap V - a}(x) \id(\gamma(a,x) = \psi(a) - q(a))\Big)\\
\leq &\exx_{a,x} \id_{V \cap V - a}(x) \id(\gamma(a,x) = \psi(a) - q(a))\\
 = &\exx_{a,x} \id_V(x) \id_V(x + a) \id(q(x + a) - q(x) - q(a) = \psi(a) - q(a)) \\
= &\exx_{a,x} \id_V(x) \id_V(x + a) \id\Big(q(x + a) -\psi(x + a)  = q(x) - \psi(x) + \psi(0)\Big)\\
= & \exx_{x, y} \id_V(x) \id_V(y) \id\Big(q(y) -\psi(y)  = q(x) - \psi(x) + \psi(0)\Big)\\
= & \sum_{\mu \in \mathbb{F}_p^{\tilde{d}}} \Big(\exx_x \id_V(x)  \id(q(x) - \psi(x) = \mu)\Big) \Big(\exx_y \id_V(y)  \id(q(y) - \psi(y) = \mu + \psi(0))\Big)\\
\leq & \max_{\nu \in \mathbb{F}_p^{\tilde{d}}} \Big(\exx_x \id_V(x)  \id(q(x) - \psi(x) = \nu)\Big) \cdot \sum_{\mu \in \mathbb{F}_p^{\tilde{d}}} \Big(\exx_y \id_V(y)  \id(q(y) - \psi(y) = \mu)\Big)\\
\leq & \delta  \max_{\nu \in \mathbb{F}_p^{\tilde{d}}} \Big(\exx_x \id_V(x)  \id(q(x) - \psi(x) = \nu)\Big),
\end{align*}

so there exists $\mu$ such that
\[2^{-10} c_2 c_3\delta \leq \exx_x \id_V(x)  \id(q(x) - \psi(x) = \mu) = \frac{|V \cap \{x \in G \colon q(x) - \psi(x) = \mu\}|}{|G|}.\]
Thus, the quadratic variety $Q = \{x \in G \colon q(x) - \psi(x) = \mu\}$ satisfies $|V \cap Q| \geq 2^{-10} c_2 c_3 |V|$. Finally, notice that
\begin{align*}\Big||G|^{-1}|Q| - p^{-\tilde{d}}\Big| = &\Big|\exx_{x \in G} \id(q(x) - \psi(x) = \mu) - p^{-\tilde{d}}\Big| = \Big|\exx_{x \in G} \Big(p^{-\tilde{d}}\sum_{\lambda \in \mathbb{F}_p^{\tilde{d}}} \omega^{\lambda \cdot (q(x) - \psi(x) - \mu)}\Big) - p^{-\tilde{d}}\Big|\\
\leq & p^{-\tilde{d}}\sum_{\lambda \in \mathbb{F}_p^{\tilde{d}}\setminus \{0\}} \Big| \exx_{x \in G} \omega^{\lambda \cdot (q(x) - \psi(x) - \mu)}\Big|.\end{align*}

It remains to give a bound on $| \ex_{x \in G} \omega^{\lambda \cdot (q(x) - \psi(x) - \mu)}|$ for non-zero $\lambda$. Note that
\begin{align*}&\Big| \exx_{x} \omega^{\lambda \cdot (q(x) - \psi(x) - \mu)}\Big|^4 = \Big|\exx_{x, y} \omega^{\lambda \cdot (q(x) - \psi(x) - \mu) - \lambda \cdot (q(y) - \psi(y) - \mu)}\Big|^2\\
&\hspace{2cm} = \Big|\exx_{x, a} \omega^{\lambda \cdot (q(x + a) - q(x) - \psi(a) + \psi(0))}\Big|^2 = \Big|\exx_a\Big(\exx_{x} \omega^{\lambda \cdot (q(x + a) - q(x) - \psi(a) + \psi(0))}\Big)\Big|^2\\
&\hspace{2cm} \leq \exx_a\Big|\exx_{x} \omega^{\lambda \cdot (q(x + a) - q(x) - \psi(a) + \psi(0))}\Big|^2\hspace{2cm}(\text{by Cauchy-Schwarz})\\
&\hspace{2cm} = \exx_{x,a,b} \omega^{\lambda \cdot (q(x + a + b) -  q(x + a) - q(x + b) + q(x))} = \exx_{a,b} \omega^{\lambda \cdot \gamma(a,b)} = \on{bias} \Big(\lambda \cdot \gamma\Big) \leq \varepsilon.\end{align*}

Thus $|Q| \leq \Big(p^{-\tilde{d}} + \sqrt[4]{\varepsilon}\Big) |G| \leq \Big(p^{s-d} + \sqrt[4]{\varepsilon}\Big) |G|$, as desired.\end{proof}

Finally, we may prove Theorem~\ref{approxQuadVarThm}.

\begin{proof}[Proof of Theorem~\ref{approxQuadVarThm}] Let $D$ be the maximum of the three absolute constants appearing in Theorems~\ref{mainstep1thm} and~\ref{approximatelinearsyssub} and Propostion~\ref{finalsteprop}. Let $V \subseteq G$ be the given $(c_0, \delta, \varepsilon)$-approximate quadratic variety. Apply Theorem~\ref{mainstep1thm} to obtain a positive quantity $c_1$, a set $A \subseteq G$ and a collection of subspaces $W_a \leq G$ indexed by $a \in A$ which satisfy properties \textbf{(i)}-\textbf{(vi)} from the conclusion of that theorem. Let $d = \lceil  \log_p(c_1^{-1} \delta^{-1}) \rceil$. By property~\textbf{(iv)}, the codimension of each $W_a$ is at most $d$, so we may misuse the notation and write $W_a$ for an arbitrary subspace inside it of codimension exactly $d$. We need to replace $c_1$ by $c_2 = c_1^{11}$ so that the conditions \textbf{(i)}-\textbf{(vi)} still hold.\\
\indent Apply Theorem~\ref{approximatelinearsyssub} with parameters $c = c_2, K = c_2^{-1}$ and $\eta = D \varepsilon \delta^{-288}$. We thus obtain parameters $c' \geq \exp\Big(-\exp\Big(\log^D (2c_2^{-1})\Big)\Big)$ and $r \leq \exp\Big(\log^D (2c_2^{-1})\Big)$, set $A' \subseteq A$ and a map $\Phi \colon G \times \mathbb{F}_p^d \to G$, affine in the first variable and linear in the second, such that $|A'| \geq c'  |G|$ and for each $a \in A'$ we have $|\on{Im} \Phi(a, \cdot) \cap W_a^\perp| \geq c' p^d$. Moreover, there exists a subspace $\Lambda \leq \mathbb{F}_p^d$ of dimension $r$ such that whenever $\lambda \notin \Lambda$ we have
\[\exx_{x,y} \omega\Big(\Phi(x, \lambda) \cdot y\Big) \leq \varepsilon',\]
where $\varepsilon' = \Big(\eta {c'}^{-2}\Big)^{1/2r}$.\\
\indent Let $M$ be a subspace of $\mathbb{F}_p^d$ of codimension $r$ such that $\mathbb{F}_p^d = \Lambda + M$ and let $e_1, \dots, e_{d-r}$ be a basis of $M$. Let us define bilinear map $\beta \colon G \times G \to \mathbb{F}_p^{d-r}$ by $\beta_i(x,y) = \Phi(x, e_i) \cdot y$ for $i \in [d-r]$. Given any $\lambda \in \mathbb{F}_p^{d-r} \setminus \{0\}$, we have
\[\on{bias} \lambda \cdot \beta = \exx_{x,y} \omega^{\sum_{i \in [d-r]} \lambda_i \beta_i(x,y)} = \exx_{x,y} \omega^{\sum_{i \in [d-r]} \lambda_i \Phi(x, e_i) \cdot y} = \exx_{x,y} \omega^{ \Phi\big(x, \sum_{i \in [d-r]} \lambda_i e_i\big) \cdot y} \leq \varepsilon',\]
since $\sum_{i \in [d-r]} \lambda_i e_i \notin \Lambda$.\\
\indent Next, we have $|\on{Im} \Phi(a, \cdot) \cap W_a^\perp| \geq c' p^d$ for all $a \in A'$. Thus,
\[{c'}^{-1} p^{-d} |G| \geq |\on{Im} \Phi(a, \cdot)^\perp + W_a| = \frac{|\on{Im} \Phi(a, \cdot)^\perp||W_a|}{|\on{Im} \Phi(a, \cdot)^\perp \cap W_a|}\]
so
\[|\on{Im} \Phi(a, \cdot)^\perp \cap W_a| \geq \frac{c' p^d |\on{Im} \Phi(a, \cdot)^\perp||W_a|}{|G|} \geq c' |\on{Im} \Phi(a, \cdot)^\perp| \geq c' p^{-d} |G|.\]
However, when $b \in \on{Im} \Phi(a, \cdot)^\perp$ then we have $0 = \Phi(a, e_i) \cdot b = \beta_i(a,b)$ for all $i \in [d-r]$, and hence $\beta(a,b) = 0$. Thus
\[|\{b \in G \colon \beta(a,b) = 0\} \cap W_a| \geq c' p^{-d} |G|.\]
All conditions of Propostion~\ref{finalsteprop} are satisfied, as long as $\varepsilon' \leq (2^{-1} c' \delta)^D$, hence there exists a quadratic variety $Q \subseteq G$ of size $|Q| \leq (2{c'}^{-1})^D \delta |G|$ such that $|Q \cap V| \geq \exp\Big(-\log^D (2{c'}^{-1})\Big)\delta |G|$. 
\end{proof}

\section{Approximate polynomials and approximate varieties}

In this appendix we show the connection between approximate quadratic polynomials and approximate quadratic varieties. As in the rest of the paper, $G$ and $H$ are finite-dimensional vector spaces over $\mathbb{F}_p$. We need some preliminary lemmas. The first one estimates the probability that a given tuple of vectors belongs to a random coset of fixed codimension. 

\begin{lemma}\label{probrcestimate}Let $v_1, \dots, v_r \in G$ and let $n = \dim G$. Let $U \leq G$ be a random subspace of codimension $d$, among such subspaces. Then, if the maximum size of independent set inside $v_1, v_2, \dots, v_r$ has size $m$, we have
\begin{equation}\label{probRC}\mathbb{P}(v_1, \dots, v_r \in U) = \frac{(p^{n-d} - 1)(p^{n-d} - p) \cdots (p^{n-d} - p^{m-1})}{(p^{n} - 1)(p^{n} - p) \cdots (p^{n} - p^{m-1})}.\end{equation}
In particular, provided $m + d < n - 2$,
\begin{equation}\label{probRCpart}p^{-md} - 4p^{m + d} p^{-n} \leq \mathbb{P}(v_1, \dots, v_r \in C) \leq p^{-md}.\end{equation}
\end{lemma}

\begin{proof}We may view random model in this lemma as follows. Let $\mathcal{L} \in \on{GL}(G)$ be a linear automorphism chosen uniformly at random. Fix arbitrary subspace $U_0$ of codimension $d$ and set $U = \mathcal{L}(U_0)$. Let $u_1, \dots, u_m$ be a maximal independent subset of $v_1, v_2, \dots, v_r$. Thus
\[\mathbb{P}\Big(v_1, \dots, v_r \in U\Big) = \mathbb{P}\Big(u_1, u_2, \dots, u_m \in U\Big).\]
However, recalling that $U = \mathcal{L}(U_0)$, we have that $u_1, u_2, \dots, u_m \in U$ if and only $\mathcal{L}^{-1}(u_1), \dots, \mathcal{L}^{-1}(u_m) \in U_0$. Since $\mathcal{L}$ is chosen uniformly at random and $u_1, \dots, u_m$ are independent, the $m$-tuple $\Big(\mathcal{L}^{-1}(u_1), \dots, \mathcal{L}^{-1}(u_m)\Big)$ is uniformly distributed over all $m$-tuples of independent vectors. Thus, $\mathbb{P}(u_1, u_2, \dots, u_m \in U) = N_U / N_G$, where $N_U$ is the number of independent ordered $m$-tuples in $U$, and $N_G$ is the corresponding quantity in $G$. Equality~\eqref{probRC} follows from a direct counting argument.\\
\indent To conclude~\eqref{probRCpart}, note first that $\frac{p^{n-d} - k}{p^n - k} \leq p^{-d}$, giving the upper bound. For the lower bound, we use elementary inequalities $1 - 2x \leq \exp(-2x) \leq 1 - x$, which holds for $x \in [0,1/2]$. We have
\[\frac{p^{n-d} - p^k}{p^n - p^k} = p^{-d} \frac{p^{n-d} - p^k}{p^{n-d} - p^{k - d}} = p^{-d} \Big(1 - \frac{p^k - p^{k-d}}{p^{n-d} - p^{k-d}}\Big) \geq p^{-d}\Big(1 - p^{k + d-n} \Big) \geq p^{-d} \exp( -2p^{k + d - n}),\]
as long as $k + d + 2 < n$. Thus
\begin{align*}\mathbb{P}\Big(v_1, \dots, v_r \in C\Big) \geq \prod_{k = 0}^{m-1} \Big(p^{-d}\exp( -2p^{k+d - n})\Big) = p^{-md}  \exp\Big( -2\sum_{k = 0}^{m-1} p^{k + d - n}\Big) \geq &p^{-md} \exp(-4p^{m + d-n}) \\
\geq &p^{-md} - 4p^{m + d -n},\end{align*}
as long as $m + d + 2 < n$.\end{proof}

The second lemma is a version of Lemma~\ref{u2controlBasic} in the context of maps between vector spaces, rather than $\mathbb{C}$-valued functions on a vector space.

\begin{lemma}\label{dependencetoadditivequads}Suppose that $A \subset G$ and that $F \colon A \to H$ is a map. Let $r,s \in \mathbb{N}$ and let $\lambda_{i\,j} \in \mathbb{F}_p$ for $i\in [r], j \in [s]$ and $\mu_i \in \mathbb{F}_p$ for $i \in [r]$ be scalars. Assume that
\begin{equation}\label{muLCcondition}\sum_{i \in [r]} \mu_i F\Big(\sum_{j \in [s]} \lambda_{i\,j} u_j\Big) = 0\end{equation}
holds for at least $\alpha|G|^s$ of $s$-tuples of $(u_1, \dots, u_s) \in G^s$. Suppose that there are distinct indices $a$ and $b$ in $[s]$ and an index $i_0 \in [r]$ such that $\mu_{i_0} \not= 0$ and $\lambda_{i\,a}\lambda_{i\,b} \not=0$ holds if and only if $i = i_0$. Then $F$ respects at least $\alpha^4 |G|^3$ additive quadruples in $A$.  
\end{lemma}

As in the case of Lemma~\ref{u2controlBasic}, the lemma still applies if the situation is simpler and there is a variable $u_i$ that appears in a single copy $F$ in the expression above.

\begin{proof}Let $I$ be the set of indices $i \in [r]$ such that $\lambda_{i\,a} \not = 0$. For given $u_{[s] \setminus \{a\}}$, let $N(u_{[s] \setminus \{a\}})$ be the number of $u_a$ such that~\eqref{muLCcondition} holds. By Cauchy-Schwarz inequality, we have
\[\sum_{u_{[s] \setminus \{a\}}} N(u_{[s] \setminus \{a\}})^2 \geq |G|^{1-s} \Big(\sum_{u_{[s] \setminus \{a\}}} N(u_{[s] \setminus \{a\}})\Big)^2 \geq \alpha^2 |G|^{s+1}.\]
In particular, we have at least $\alpha^2 |G|^{s+1}$ $(s+1)$-tuples $(u_{[s] \setminus \{a\}}, v_a, v'_a)$ such that 
\begin{align}0 = &\bigg(\sum_{i \in [r]} \mu_i F\Big(\sum_{j \in [s] \setminus \{a\}} \lambda_{i\,j} u_j + \lambda_{i\,a} v_a\Big)\bigg) - \bigg(\sum_{i \in [r]} \mu_i F\Big(\sum_{j \in [s] \setminus \{a\}} \lambda_{i\,j} u_j + \lambda_{i\,a} v'_a\Big)\bigg)\nonumber\\
= & \sum_{i \in I} \mu_i F\Big(\sum_{j \in [s] \setminus \{a\}} \lambda_{i\,j} u_j + \lambda_{i\,a} v_a\Big) - \mu_i F\Big(\sum_{j \in [s] \setminus \{a\}} \lambda_{i\,j} u_j + \lambda_{i\,a} v'_a\Big).\label{2ndLCcondition}\end{align}

Let $M(u_{[s] \setminus \{a, b\}}, v_a, v'_a)$ be the number of $u_b$ such that~\eqref{2ndLCcondition} holds for $(s+1)$-tuple $(u_{[s] \setminus \{a\}}, v_a, v'_a)$. By Cauchy-Schwarz inequality we get
\[\sum_{u_{[s] \setminus \{a, b\}}, v_a, v'_a} M(u_{[s] \setminus \{a, b\}}, v_a, v'_a)^2 \geq |G|^{-s} \Big(\sum_{u_{[s] \setminus \{a, b\}}, v_a, v'_a} M(u_{[s] \setminus \{a, b\}}, v_a, v'_a)\Big)^2 \geq \alpha^4 |G|^{s+2}.\]

Hence, we get at least $|G|^{s + 2}$ of $(s+2)$-tuples $(u_{[s] \setminus \{a, b\}}, v_a, v'_a, w_b, w'_b)$ such that 
\begin{align*}0 = &\bigg(\sum_{i \in I} \mu_i F\Big(\sum_{j \in [s] \setminus \{a, b\}} \lambda_{i\,j} u_j + \lambda_{i\,a} v_a + \lambda_{i\,b} w_b\Big) - \mu_i F\Big(\sum_{j \in [s] \setminus \{a, b\}} \lambda_{i\,j} u_j + \lambda_{i\,a} v'_a + \lambda_{i\,b} w_b\Big)\bigg)\\
&\hspace{2cm}-\bigg(\sum_{i \in I} \mu_i F\Big(\sum_{j \in [s] \setminus \{a, b\}} \lambda_{i\,j} u_j + \lambda_{i\,a} v_a + \lambda_{i\,b} w'_b\Big) - \mu_i F\Big(\sum_{j \in [s] \setminus \{a, b\}} \lambda_{i\,j} u_j + \lambda_{i\,a} v'_a + \lambda_{i\,b} w'_b\Big)\bigg)\\
= & \mu_{i_0} \bigg(F\Big(\sum_{j \in [s] \setminus \{a, b\}} \lambda_{i_0\,j} u_j + \lambda_{i_0\,a} v_a + \lambda_{i_0\,b} w_b\Big) - F\Big(\sum_{j \in [s] \setminus \{a, b\}} \lambda_{i_0\,j} u_j + \lambda_{i_0\,a} v'_a + \lambda_{i_0\,b} w_b\Big) \\
&\hspace{2cm} -F\Big(\sum_{j \in [s] \setminus \{a, b\}} \lambda_{i_0\,j} u_j + \lambda_{i_0\,a} v_a + \lambda_{i_0\,b} w'_b\Big) + F\Big(\sum_{j \in [s] \setminus \{a, b\}} \lambda_{i_0\,j} u_j + \lambda_{i_0\,a} v'_a + \lambda_{i_0\,b} w'_b\Big)\bigg).\end{align*}

Thus, averaging over $u_{[s] \setminus \{a,b\}}$ we see that $F$ respects at least $\alpha^4 |G|^3$ additive quadruples.\end{proof}

Finally, we make use of Green's $\mathsf{U}^2$-arithmetic regularity lemma.

\begin{lemma}[Theorem 2.1~\cite{GreenReg}] \label{u2arithmreg}Given a set $A \subseteq G$ and a parameter $\varepsilon > 0$, there exists a decomposition $G = K \oplus T$, where $\dim T \leq W(O(\varepsilon^{-3}))$, where $W(t)$ is a tower of twos of height $t$, such that for all but at most $\varepsilon |T|$ of $t \in T$, we have
\begin{equation}\label{uniformonsubspace}\max_{k \notin K^\perp} \Big|\exx_{x \in K} \id_A(x + t) \omega^{-k \cdot x}\Big| \leq \varepsilon.\end{equation}
\end{lemma}

If condition~\eqref{uniformonsubspace} holds, we say that $t$ is \emph{$\varepsilon$-regular}.\\

We are now ready to prove Proposition~\ref{approximatePolynomialsAndApproximateVars}.

\begin{proof}[Proof of Proposition~\ref{approximatePolynomialsAndApproximateVars}] By assumptions, we have a set $\mathcal{G}$ of quadruples $(x,a,b,c) \in G^4$ of size at least $c_0 |G|^4$ such that $\Delta_a \Delta_b\Delta_c F(x) = 0$ and all 8 arguments of $F$ belong to $A$.\\
\indent In the first step of the proof, we need to pass to a sufficiently regular subset of $A$. Let $\eta > 0$ be a parameter to be specified later. Apply Lemma~\ref{u2arithmreg} to find a decomposition $G = K \oplus T$, with $\dim T \leq W(O(\eta^{-3}))$, such that $t$ is $\eta$-regular for all but at most $\eta |T|$ of elements $t \in T$. Let $T_{\text{reg}}$ be the set of $\eta$-regular elements $t \in T$. Then, using decomposition $G = K \oplus T$, we have 
\begin{align*}c_0 \leq &\exx_{x,a,b,c \in G} \id_A(x)\id_A(x + a) \dots \id_A(x + a + b +c) \id(\Delta_{a,b,c} F(x) = 0) \\
&\hspace{1cm}= \exx_{x',a',b',c' \in K} \exx_{x'', a'', b'', c'' \in T} \id_A(x' + x'') \id_A(x' +x''+ a' + a'') \dots \\
&\hspace{4cm}\id_A(x' + x'' + a' + a'' + b' + b'' + c' + c'') \id(\Delta_{a' + a'',b' + b'',c' + c''} F(x' + x'') = 0).\end{align*}
On the other hand, the fact that $|T \setminus T_{\text{reg}}| \leq \eta |T|$ implies
\[\exx_{\ssk{x',a',b',c' \in K\\x'', a'', b'', c'' \in T}} \id_{T \setminus T_{\text{reg}}}(x'') \id_A(x' + x'') \id_A(x' +x''+ a' + a'') \dots \id_A(x' + x'' + a' + a'' + b' + b'' + c' + c'') \leq \eta,\]
as well as similar inequalities with any element among $x'' + a'', x'' + b'', \dots, x'' + a'' + b'' + c''$ instead of $x''$ as the argument of $\id_{T \setminus T_{\text{reg}}}$. It follows that, if we choose $\eta \leq c_0 /16$, 
\begin{align*}c_0/2 \leq &\exx_{\ssk{x',a',b',c' \in K\\x'', a'', b'', c'' \in T}} \id_{T_{\text{reg}}}(x'') \dots \id_{T_{\text{reg}}}(x'' + a'' + b'' + c'') \id_A(x' + x'') \id_A(x' +x''+ a' + a'') \dots \\
&\hspace{4cm}\id_A(x' + x'' + a' + a'' + b' + b'' + c' + c'') \id(\Delta_{a' + a'',b' + b'',c' + c''} F(x' + x'') = 0).\end{align*}

Thus, by averaging, we may find $x'', a'', b'', c'' \in T$ such that all 8 elements $x'',$ $x'' + a'',$ $x'' + b'', \dots,$ $x'' + a'' + b'' + c'' + T$ are $\eta$-regular and
\begin{align*}c_0/2 \leq &\exx_{x',a',b',c' \in K} \id_A(x' + x'') \id_A(x' +x''+ a' + a'') \dots \\
&\hspace{4cm}\id_A(x' + x'' + a' + a'' + b' + b'' + c' + c'') \id(\Delta_{a' + a'',b' + b'',c' + c''} F(x' + x'') = 0).\end{align*}

A Cauchy-Schwarz inequlity based argument like that in Lemma~\ref{dependencetoadditivequads} allows us to pass to a subset of a single coset $A \cap (x'' + K)$ on which $F$ respects at least $\frac{c_0^8}{2^8}|K|^4$ additive cubes, with $x'' \in T_{\text{reg}}$. Thus, writing $c_1 = \frac{c_0^8}{2^8}$ and misusing the notation and writing $A$ instead of  $(A - x'') \cap  K$, and $G$ instead of $K$, we may assume that $|\widehat{\id_A}(r)| \leq \eta$ for all $r \in G \setminus \{0\}$. Let $c_2$ be the density of $A$, hence $c_2 \geq c_1$. We also have that $F$ respects at most $\varepsilon' |G|^3$ additive quadruples in $A$, where $\varepsilon' = \varepsilon p^{3 \dim T}$. In particular, the number of additive quadruples in $A$ is then
\[\sum_{x,a,b \in G} \id_A(x) \id_A(x + a) \id_A(x + b) \id_A(x + a + b) = |G|^3 \sum_{r \in G} |\widehat{\id_A}(r)|^4 \leq (c_2^4 + \eta^2) |G|^3.\]

\indent Let $U$ be a random subspace of $H$ of codimension $d$. Let $V = \{x \in A \colon F(x) \in U\}$. Let $X$ be the number of additive quadruples in $V$, and $Y$ the number of quadruples $(x,a,b,c) \in \mathcal{G}$ such that 8 points $F(x), F(x + a), \dots, F(x + a + b + c)$ belong to $U$.\\
\indent We use the second moment method to prove the theorem. We use Lemma~\ref{probrcestimate} frequently below, without additional comments. Firstly, note that
\[\exx |V| = \sum_{x \in A} \mathbb{P}(F(x) \in U) \geq (p^{-d} - 4p^{1 + d -n}) |A|,\]
and
\begin{align*}\exx |V|^2 = \sum_{x, y \in A} \mathbb{P}(F(x), F(y) \in U) \leq &p^{-2d} \sum_{x, y \in A} \id(\on{rank}\{F(x) , F(y)\} = 2) + \sum_{x, y \in A} \id(\on{rank}\{F(x), F(y)\} \leq 1).\end{align*}

Note that the number of $x \in A$ such that $F(x) = 0$ is at most $\sqrt[4]{\varepsilon'} |G|$, as otherwise we get more than $\varepsilon' |G|^3$ additive quadruples in $A$ that are respected by $F$. Furthermore, for each $\lambda, \mu \not =0 $ we get at most $\sqrt[4]{\varepsilon'} |G|^2$ pairs $(x,y) \in A^2$ such that $ \lambda F(x) +  \mu F(y) = 0$, by Lemma~\ref{dependencetoadditivequads} (after a change of variables $y = x + a$), again as $F$ respects few additive quadruples in $A$. Thus
\[\exx |V|^2 \leq p^{-2d} |A|^2 + (p^2 + 2) \sqrt[4]{\varepsilon'} |G|^2.\]
By Chebyshev's inequality we have
\begin{equation}\label{veqn2ndmoment}\mathbb{P}\Big(\Big||V| - p^{-d}|A|\Big| \leq \sqrt{\eta}|G|\Big) \geq 1 - \Big(2c_2^2 \eta^{-1} p^{2} \sqrt[4]{\varepsilon'} + 4\eta^{-1}p^{1 + d -n}\Big).\end{equation}
Thus, this probability is at least $0.999$ if $\varepsilon'$ is sufficiently small in terms of other parameters.

\hspace{\baselineskip}

Next, consider random variable $X$. We have
\begin{align*}\exx X = &\sum_{x, a,b \in G} \id_A(x) \id_A(x + a) \id_A(x + b) \id_A(x + a + b) \mathbb{P}(F(x), F(x + a), F(x + b), F(x + a + b) \in U)\\
\geq &\Big(p^{-4d} - p^{4 + d-n}\Big) \sum_{x, a,b \in G} \id_A(x) \id_A(x + a) \id_A(x + b) \id_A(x + a + b)\\
\geq &\Big(p^{-4d} - p^{4 + d-n}\Big) c_2^4 |G|^3.\end{align*}

Next, we have
\begin{align*}\exx X^2 = &\sum_{\ssk{x, a,b \in G\\x', a',b' \in G}} \id_A(x) \id_A(x + a) \id_A(x + b) \id_A(x + a + b) \mathbb{P}(F(x), F(x + a), F(x + b), F(x + a + b) \in U)\\
&\hspace{1cm}\id_A(x') \id_A(x' + a') \id_A(x' + b') \id_A(x' + a' + b') \mathbb{P}(F(x'), F(x' + a'), F(x' + b'), F(x' + a' + b') \in U)\\
\leq &p^{-8d} \Big(\sum_{x, a,b \in G} \id_A(x) \id_A(x + a) \id_A(x + b) \id_A(x + a + b)\Big)^2 \\
&\hspace{1cm}+ \sum_{\ssk{x, a,b \in G\\x', a',b' \in G}} \id_A(x) \dots \id_A(x' + a' + b') \id\Big(\on{rank}\{F(x), \dots, F(x' + a' + b')\} \leq 7\Big)\\
\leq &p^{-8d}(c_2^4 + \eta^2)^2 |G|^6 + p^{8} \sqrt[4]{\varepsilon'} |G|^6,\end{align*}
again using Lemma~\ref{dependencetoadditivequads} in the last line. By Chebyshev's inequality we have
\begin{equation}\label{xeqn2ndmoment}\mathbb{P}\Big(\Big|X - p^{-4d}c_2^4 |G|^3\Big| \leq \sqrt{\eta} |G|^3\Big) \geq 1 - \Big(3\eta + p^8 \sqrt[4]{\varepsilon'} + p^{4 +d -n}\Big),\end{equation}
which is at least $0.999$, again provided $\varepsilon'$ is sufficiently small in terms of other parameters.

\hspace{\baselineskip}

Finally, we consider random variable $Y$. Recall that if $(x,a,b,c) \in \mathcal{G}$, then $\Delta_{a,b,c} F(x) = 0$ (and all 8 points that are arguments of $F$ belong to $A$). Then 8 values $F(x), \dots, F(x + a + b + c)$ belong to $U$ if and only if any 7 of them belong to $U$. Hence
\begin{align*}\exx Y = \sum_{(x,a,b,c) \in \mathcal{G}} \mathbb{P}(F(x + a), F(x + b), \dots, F(x + a + b + c) \in U) \geq \Big(p^{-7d} - 4 p^{7 + d - n}\Big) |\mathcal{G}|.\end{align*} 

On the other hand,
\begin{align*}\exx Y^2 = &\sum_{(x,a,b,c), (x',a',b',c') \in \mathcal{G}} \mathbb{P}(F(x + a), \dots, F(x + a + b + c), F(x' + a'), \dots, F(x' + a' + b' + c') \in U)\\ 
\leq & p^{-14d} |\mathcal{G}|^2 + \sum_{x,a,b,c,x',a',b',c' \in G} \id\Big(\on{rank}\{F(x + a), \dots, F(x' + a' + b' + c')\} \leq 13\Big).\end{align*}

Finally, for each non-trivial linear combination
\[\lambda_1 F(x+a) + \lambda_2 F(x + b) + \dots + \lambda_{14} F(x' + a' + b' + c')\]
we claim that there are at most $\sqrt[4]{\varepsilon'}|G|^8$ choices of $(x,x', \dots, c, c') \in G^8$ such that it vanishes. Make a change of variables $x \mapsto x - a -b-c$ and $x' \mapsto x' - a' - b' -c'$. Then the claim follows from Lemma~\ref{dependencetoadditivequads} if any of $\lambda_1, \lambda_2, \lambda_3, \lambda_8, \lambda_9, \lambda_{10}$ is non-zero, using adequate two variables among $a,b,c,a',b',c'$. For the remaining 8 coefficients, we use $x$ or $x'$ and possibly one more variable. Thus,
 \[\exx Y^2 \leq p^{-14d} |\mathcal{G}|^2 + p^{14} \sqrt[4]{\varepsilon'} |G|^8.\]

Using Chebyshev's inequality one more time, we obtain

\begin{equation}\label{yeqn2ndmoment}\mathbb{P}\Big(\Big|Y - p^{-7d} |\mathcal{G}|\Big| \leq \sqrt{\eta} |G|^4\Big) \geq 1 - \Big(p^{14} \sqrt[4]{\varepsilon'} + 4p^{7 + d - n}\Big),\end{equation}
which is at least 0.999, provided $\varepsilon'$ is sufficiently small.\\ 
\indent To obtain desired bounds, we set $\eta = 10^{-24}\xi^8c_0^8 p^{-2d}$ and $\varepsilon = W(K \eta^{-3})$ for a sufficiently large absolute constant $K$. To sum everything up, combining~\eqref{veqn2ndmoment},~\eqref{xeqn2ndmoment} and~\eqref{yeqn2ndmoment} we see that with probability at least 0.99, we have
\[\Big||V| - p^{-d}|A|\Big| \leq \sqrt{\eta}|G|,\,\,\Big|X - p^{-4d}c_2^4 |G|^3\Big| \leq \sqrt{\eta} |G|^3,\text{ and }\Big|Y - p^{-7d} |\mathcal{G}|\Big| \leq \sqrt{\eta} |G|^4.\]
For such a choice of $V$, we see that $V$ has density $\delta \in [2^{-9}c_0^8 p^{-d}, 2p^{-d}]$, $\|\id_V - \delta\|_{\mathsf{U}^2} \leq 100\eta^{1/8} \leq \xi$ and the number of additive cubes in $V$ is at least $p^{-7d} |\mathcal{G}| - \sqrt{\eta}|G|^4 \geq p^{-7d}c_0^82^{-9}|G|^4$. On the other hand, since $\|\id_V - \delta\|_{\mathsf{U}^2} \leq 100\eta^{1/8}$, Lemma~\ref{u2controlBasic} implies that the number of 7-tuples of the form $(x + a, x+ b, \dots, x+ a+ b+c)$ in $V$ is at most $2\delta^7|G|^4$. Hence, $V$ is a $(c', \delta, \xi)$-approximate quadratic variety for some $c' \in [c_0^82^{-9}, 2]$, as claimed.\end{proof}

\end{document}